\documentclass[a4paper,12pt]{article}

\usepackage{amsmath}
\usepackage{amsthm}
\usepackage{amssymb}
\usepackage{latexsym}
\usepackage{graphicx}
\usepackage{stmaryrd}
\usepackage{mathrsfs}
\usepackage{enumerate}
\usepackage{cases}
\usepackage{color}

\newcommand{\BLACK}{\color{black}}

\definecolor{dGREEN}{rgb}{0.0,0.5,0.5}

\newcommand{\glalign}[2]{\lower.6ex\vbox{
\baselineskip\lineskip\ialign{$#1\hfil##\hfil$\crcr#2\crcr=\crcr}}}

\newcommand{\del}{\partial}

\newcommand{\Om}{\Omega}

\newcommand{\dlambda}{\,{\rm d}\lambda}

\newcommand{\dtau}{\,{\rm d}\tau}

\newcommand{\dx}{\,{\rm d}x}
\newcommand{\dy}{\,{\rm d}y}

\renewcommand{\div}{\mbox{\rm div}\,}

\newcommand{\supp}{\mbox{\rm supp}\,}

\newcommand{\IR}{\mathbb{R}}
\newcommand{\IC}{\mathbb{C}}
\newcommand{\IN}{\mathbb{N}}

\newcommand{\loc}{\mathrm{loc}}

\newcommand{\D}{\mathcal{D}}
\def\eqn#1$$#2$${\begin{equation}\label#1#2\end{equation}}

\numberwithin{equation}{section}

\newtheorem{defi}{Definition}[section]
\newtheorem{thm}[defi]{Theorem}

\newtheorem{prop}[defi]{Proposition}
\newtheorem{lem}[defi]{Lemma}
\newtheorem{rem}[defi]{Remark}

\def\eqn#1$$#2$${\begin{equation}\label#1#2\end{equation}}

\numberwithin{equation}{section}
\numberwithin{equation}{section}
\allowdisplaybreaks[4]

\begin{document}

\title{
\bf \large
The time periodic problem for the Navier-Stokes equations\\ in exterior domains in weighted spaces\BLACK}
%the Muckenhoupt class} 
\author{{\normalsize
Reinhard Farwig\footnote{
Fachbereich Mathematik, Technische Universit\"at Darmstadt, Schlossgartenstr. 7, 64289 Darmstadt, \quad Germany, \texttt{farwig@mathematik.tu-darmstadt.de}} \;
and \;
Kazuyuki Tsuda\footnote{
Kyushu Sangyo University, 3-1 Matsukadai 2-chome,
Higashi-ku, Fukuoka,
813-8503 Japan, \texttt{k-tsuda@ip.}\texttt{kyusan-u.ac.jp}}   }\\[2ex]
%{\normalsize\it }
}
\date{}
\maketitle

\begin{abstract}
\noindent
The paper considers the time periodic problem of the Navier-Stokes system in an exterior domain under time periodic external forces.  
Existence of periodic mild solutions is obtained in the critical scale invariant space $C(\IR;L^n)$ $(n \geq 4)$  if the external force is small without exploiting any divergence form. 
Previous studies mainly rely on either potential theoretical estimates or time-space integral estimates in Lorentz spaces introduced by Yamazaki (Math. Ann. {317}  (2000)).  
In this article, a new method based on radially symmetric Muckenhoupt weights in space is used.  
To apply these weights, we reconsider weighted $L^p$-$L^q$ decay estimates for the Stokes semigroup introduced  by Kobayashi and Kubo (2012-2015).

\end{abstract}

\noindent {\bf Key Words and Phrases.} Navier-Stokes equations; Muckenhoupt weights; two-weight $L^p$-$L^q$ Stokes semigroup estimates; local decay estimates;
mild time periodic solutions
\\

\noindent {\bf 2010 Mathematics Subject Classification Numbers.} 35Q30; 35B10; 76D05\\[1ex]

\section{Introduction}
We consider the time periodic problem of the Navier-Stokes equations in an exterior domain. % maximal regularity class.
Let $n \geq 3$ and $\Omega \subset\mathbb{R}^n$ be an exterior domain with $C^3$ boundary. 
The Navier-Stokes equations on $\Omega$ are described by  
\begin{align}\label{equ:ns}
	\begin{aligned}
		v_t - \nu\Delta v + v\cdot \nabla v + \nabla p & = f \quad\! \text{in }\;\, \Omega,\\
		 \div v & = 0\quad   \text{in }\;\, \Omega,\\
		v & = 0  \quad  \text{on }\,  \partial \Omega.
	\end{aligned}
\end{align}
Here $v$ denotes the unknown velocity of an incompressible, viscous  fluid and $p$ denotes the pressure, respectively, at time $t\in \IR$ and position $x\in\Omega$. Moreover, let $f=f(x,t)$ be a given time periodic external force with period $T$, {\em i.e.}  $f(x,t+T)=f(x,t)$. For simplicity, the coefficient of viscosity $\nu$ equals $1$.

%In this paper, we consider the time periodic problem of \eqref{equ:ns} with $T$-periodic external force in spaces with radially symmetric weights in Muckenhoupt classes. 

Time periodic problems constitute one of the fundamental phenomena for incompressible fluid flow and have been studied by many researchers. Especially, the exterior domain case is on the one hand important for applications in engineering and on the other hand mathematically challenging because of the restricted choice of adequate function spaces and their embedding properties \BLACK as well as the non-invertibility of the relevant Stokes operator.  

As the oldest result, Serrin \cite{Serrin} describes a time periodic solution on  bounded domains which may even depend on time. 
Kaniel and Shinbrot \cite{Kaniel-Shinbrot}  
show for a small time periodic external force the existence of a strong time periodic solution. 
The case $\IR^3$ was considered by Maremonti \cite{Marem91}. Then Kozono and Nakao \cite{Kozono-Nakao} prove existence of a unique  strong time periodic solution for a given small periodic external force on $\mathbb{R}^n$, $n\geq3$, and exterior domains when $n \geq 4$. The exterior domain case in $\mathbb{R}^3$ is solved by Yamazaki \cite{Yamazaki}, introducing a weak mild solution and proving existence and stability in the space $C_{\rm per}(\mathbb{R}; L^{n,\infty}(\Omega))$, where $L^{n,\infty}(\Omega)$ equals the weak $L^n$ space, with small time periodic external forces $f=\div F$ with 
$F \in L^{\frac{n}{2},\infty}(\Omega)$. 
His method based on an $L^1$ integral with weights in time is also known as the Yamzazaki estimate. 
Geissert,  Hieber and Nguyen \cite{GHN} consider the time periodic problem of an abstract parabolic equation as a generalization of this problem on Lorentz spaces. 
Galdi and Kyed \cite{GK} obtain time periodic solutions by combining a potential theoretical method for the stationary part with the transference of multipliers on the genuine oscillatory part of solutions.   
Recently Hieber and Nguyen \cite{HN} show existence of a unique $T$-periodic as well as almost time periodic mild solution and their stability in Lorentz spaces on the half space $\mathbb{R}^n_+$ under a small periodic external force $f=\div F$. %They also consider almost time periodic solutions. 

Concerning the time periodic exterior domain case, further realistic situations have been studied.  
Galdi and Sohr \cite{GS} consider a Navier-Stokes liquid that is moving in a domain under the action of a small $T$-periodic body force $f=\div F$ and show the existence of a time periodic solution behaving like $1/|x|$ at spatial infinity;  for similar results we refer to \cite{Galdi-2020}.   
The problem of attainability was considered by Galdi and Hishida \cite{Galdi-Hishida}; based on Yamazaki's estimate they analyze the situation that a body $\mathcal B$ starts from rest by a translational motion in water, but attains after some time a spinless oscillatory motion of period $T$, leading to a Navier-Stokes solution with the same period.
%\BLUE (Concerning the moving boundary case like Kyed and Shibata's results, I consider to  write details on the second paper. )\BLACK

\vspace{1ex}

In this paper we consider the classical time periodic Navier-Stokes problem \eqref{equ:ns} in exterior domains. However, we use a new method compared to the above approaches, {\em viz.} using radially symmetric Muckenhoupt weights in space. On the other hand, duality arguments and a representation of the  external \BLACK force $f$ in divergence form are not needed.

\begin{defi}{(Muckenhoupt class $\mathscr{A}_q$)}
Let $1<q <\infty$. A weighted function $0 \leq w \in L^1_{\rm loc}(\mathbb{R}^n)$ belongs to the Muckenhoupt class $\mathscr{A}_q(\IR^n) $ if the function $w$ satisfies 
$$
\sup_Q \bigg(\frac{1}{|Q|}\displaystyle\int_Q w \dx\bigg)\bigg(\frac{1}{|Q|}\displaystyle\int_Q w^{-1/(q-1)}\dx\bigg)^{q-1} \leq C <\infty
$$
for all cubes $Q \subset \mathbb{R}^n$, where $|Q|$ denotes the Lebesgue measure of $Q$.  
\end{defi}

For example, the radially symmetric weight functions $w(x)=\langle x \rangle^\alpha := (1+|x|^2)^{\alpha/2}$, $-n < \alpha < n(q-1),$ belong to the Muckenhoupt class $\mathscr{A}_q $. By analogy, in the homogeneous case,   $w(x)=|x|^\alpha \in \mathscr{A}_q $. One of the first applications of Muckenhoupt weights to the Stokes system was a weighted $L^2$ energy estimates with weights $\langle x \rangle^\alpha$ and resolvent estimates using the full Muckenhoupt class by Farwig and Sohr, see   \cite{FS96}, \cite{FS}. Since then numerous authors used weighted estimates to refine former estimates from $L^q$ norms to weighted $L^q$ norms. Among others we mention results by Bae, Brandolese and Jin \cite{Bae-Br-Jin}, Bae and Jin \cite{Bae-Jin}, 
Bae and Roh \cite{Bae-Roh}; %, He and Miyakawa \cite{He-Miya-2009JDE}(2008) as well as P. Han  \cite{Han2018}; 
see also further references in those articles. We refer to Takahashi \cite{Takahashi} for the current status about weighted norm estimates including anisotropic weights for the Navier-Stokes-Oseen system. Recently Maremonti and Pane \cite{Marem-Pane} constructed local-in-time mild solutions for the $n$-dimensional Cauchy problem in the intersection of three spaces with and without radial weights and proved uniqueness.
%; to this aim, the authors use the intersection of three spaces with and without radial weights for the unknown velocity. \BLACK

In the present paper we mainly focus on the weight $\langle x \rangle^\alpha$ such that we obtain periodic solutions in the critical scale invariant space $C(\mathbb{R}; L^n)$. 
We define the weighted $L^q$ space $L^q_s (\Omega)$ with weight $\langle x \rangle^{sq}$ and its solenoidal subspace $L^q_{s,\sigma} (\Omega)$ in Section 2 below. Our main result is stated as follows.  

\vspace{2ex}

\begin{thm}\label{existence-per_2} 
Let $\Omega\subset\IR^n$, $n \geq 4$, be an exterior domain  of class $C^3$.  Moreover, let $1< q_1 < n$, $n/2 <q_2<n$, and let ${q}_{12} = \frac{q_1 q_2}{q_1+q_2}$, ${q}^*_{22} =\frac{{q_2^*} q_2}{{q_2^*}+q_2}$ where
 $q_2^*=\frac{n q_2}{n-q_2}$. We suppose that $-\frac{n}{q_1}<s_1 < \frac{n}{q'_1}$, $0<s_2< \frac{n}{q'_2}$,   
$0 \leq s_1+s_2 < \frac{n}{q_{12}'}$, and
\begin{align} \label{As1} 
1 < & \;\frac{n}{q_1} + s_1 < n-2,\\
 2 < & \;\frac{n}{q_2} + s_2 <  \frac{n+1}{2} +\frac{s_2}{2}\BLACK . \label{As11}
\end{align}
\BLACK 
%\begin{align*}
%2< n-\frac{n}{q_1}+|s_1|,  
%\end{align*}
%$$
%\frac{n}{q_2}< n+1 - \frac{s_2}{2}.
%$$  
For the external force   
$ f\in  L^\infty_{\rm per}\big(\mathbb{R}; L^{{q}_{12}}_{s_1+s_2}(\Omega)\cap L^{{q}^*_{22}}_{2s_2}(\Omega)\big) $
assume that the norm 
\begin{align}\label{As2}
|f|_s  := \|f\|_{L^\infty_{\rm per}(\IR; L^{{q}_{12}}_{s_1+s_2}(\Omega))} + \|f\|_{L^\infty_{\rm per}(\IR; L^{{q}^*_{22}}_{2 s_2}(\Omega))} 
\end{align}
is sufficiently small. 
Then there exists a $T$-periodic solution $u \in C^0_{per}(\mathbb{R}; L^{n}_{\sigma}(\Omega))$  for either $s_1\geq 0$ or $s_1<0$ with small $|s_1|$ \BLACK  and
with $\nabla u \in L^\infty_{\rm per}(\mathbb{R};  L^{q_2}_{s_2}(\Omega))$ of \eqref{equ:ns} 
\BLACK
such that
\begin{align}\label{est-u-nablau-2}
 \sup\nolimits_{t\in [0,T]} \big(\big\|u(t)\big\|_{L^{q_1}_{s_1}(\Omega)} + \big\|\nabla u(t)\big\|_{L^{q_2}_{s_2}(\Omega)}\big) \leq C|f|_s. 
 \end{align}

Finally, $\nabla u \in C^0_{per}(\IR;L^{q_2}_{s_2}(\Omega))$ provided
$q_1\leq q_2$ and $\frac{1}{q_1}-\frac{1}{q_2} < \frac1n$. \BLACK
\end{thm}

\begin{rem}\label{rem0}
{\rm 

(1) Each of the conditions \eqref{As1} and \eqref{As11} immediately imply that Theorem \ref{existence-per_2} does not hold in the case $n=3$. 

(2) There is an obvious argument why in the linear case, omitting the nonlinear term $u\cdot\nabla u$, either $q_1>n$ or $s_1<0$ are necessary for $n=3$.
Actually, consider the fundamental solution, $E_{{\rm St}}(x)$, of the Stokes operator in $\IR^3$, modified by a cut off function and an application of the Bogovskii operator to get a smooth solenoidal vector field vanishing on $\partial\Omega$ and coinciding with $E_{{\rm St}}(x)$ for large $|x|$.  In this way we get a smooth stationary solution of \eqref{equ:ns} with a compactly supported and smooth right hand side, hence a $T$-periodic solution. In order to satisfy the smallness condition on $f$ in Theorem \ref{existence-per_2} we may multiply by a sufficiently small positive constant. Since $|E_{{\rm St}}(x)|\sim 1/r$, we need $1/r\in L^{q_1}_{s_1}(\Omega)$ or, in other words, $\frac{3}{q}+s<1$. Thus either $s_1 <0$ or $q_1>3$ are necessary if $q_1<3$ or $s_1>0$, respectively. However, for $n=4$ and $|E_{{\rm St}}(x)|\sim r^{-2}$, condition \eqref{As1} with $q_1<4$ can be satisfied with both positive and negative $s_1$.
}
\end{rem}
%\vspace{1ex}

To the best of our knowledge, there are no results using 
Muckenhoupt weights in $L^q$ class for $1< q <\infty$ 
to construct time periodic solutions to \eqref{equ:ns} in the
exterior domain case. 
We note that for the specific case $q=2$ and  unbounded domains, 
there are a few results on $T$-periodic solutions using weighted $L^2$ energy estimates, see {\em e.g.} Keblikas \cite{Keblikas} for domains with cylindrical outlets to infinity  using different $L^2$ weights in different outlets. Moreover, Tsuda 
\cite{Tsuda} considered the compressible Navier-Stokes equations on the whole space with an analysis of the high frequency part via weighted $L^2$ energy estimates. 

Many results for the time periodic problem on unbounded domains use potential
theoretical estimates, {\em i.e.,} they are based on the analysis for stationary solutions in weighted $L^\infty$ norms, see
for example, \cite{GK, GS} 
and 
%Eiter and Shibata (2023) (exterior domain),  Eiter and Kyed (2021) (exterior domain),
Nakatsuka \cite{Nakatsuka}, or on Yamazaki type estimates, see for example,  \cite{Galdi-Hishida, GHN, HN, Yamazaki}.

 %Hieber and Nguyen (2022) (half space), Galdi and Hishida (2021) (exterior domain), Geissert, Matthias, Hieber (2016) and Yamazaki (2006) (exterior domain) 
 
Furthermore, we obtain in Theorem \ref{existence-per_2} time periodic solutions in the critical scale invariant space $C(\mathbb{R}; L^n)$ and consider the non-divergence form in the assumption of the external force as in the study of Okabe and Tsutsui \cite{OTs} for the whole space case in Lorentz spaces. Our approach will be extended to the time periodic moving boundary problem, see the forthcoming paper \cite{FaTs-movperext}. 

%(\BLUE Comment for the moving boundary problem;  
% The above approach by  theMuckenhoupt weights in $L^q$ is effective for our periodic moving boundary problem although we need the strict restriction $q< \frac{n}{2}$ in the exponent of $L^q$ due to using the uniform resolvent estimates of the Stokes operator, while  the usual Yamazaki estimate approach uses 
%$q \leq n$ and thus it is not expected to work. On the other hand, a typical example of the potential theoretical estimate with transference of multipliers is \cite{ES}. 
%Eiter and Shibatacan consider a translating body. But they can not handle the critical scale invariant space $C(\mathbb{R}; L^n)$ dueto the use of multiplier theorems.  
%Recent studies are focusing on scale invariant spaces 
%to work with maximal $L^1$ regularity, see, {\rm e.g, } \cite{OS}. 
%Their method heavily relies on an explicit fundamental solution formula, and they assume that the external forces have the divergence form. 

%In contrast to \cite{ES}, we consider the time periodicsolution in the critical space $C(\mathbb{R}; L^n)$. 
%Our approach does not need any explicit solution formula and is based on abstract evolution operator theory, and thus  is 
%applicable to other problems of parabolic type. Moreover, we do not need the assumption on the external forces.

%In the proof, the Caffarelli-Kohn-Nirenberg inequality and the moment inequality  work well to estimate additional perturbation terms related to  the evolution operator. )

%\vspace{1ex}
The crucial point to obtain the main theorem is the sufficiently fast 
decay of the Stokes semigroup in time; this is reached by introducing  
Muckenhoupt weights in $L^q$. 
Actually, we use for any time $t$ an integral formulation on $(-\infty,t)$ as in \cite{Kozono-Nakao}, see \eqref{equ:ns-per} below.
Since the time integral is a global one, we need for the Stokes operator $A$ on $\Omega$ integrability of the Stokes semigroup $e^{-(t-\tau)A}$ for $\tau\in(-\infty,t)$. 

Usually, the decay rate is obtained by a duality argument combined with the divergence form of the nonlinear terms. 
To overcome the difficulty of decay in  exterior domains, Yamazaki uses in 
\cite[Theorem 2.2]{Yamazaki} estimates of the type $\|e^{-\tau A}u\|_{q,1} \leq C\tau^{n/2q-n/2p}\|u\|_{p,1}$ in endpoint Lorentz spaces $L^{q,1}(\Omega)$ and corresponding integral estimates 
in \cite[Corollary 2.3]{Yamazaki},  
%$\int_0^\infty \tau^{n/2p-n/2q-1} \|e^{-\tau A}u\|_{q,1}\dtau \leq C\|u\|_{p,1}$ in Lorentz spaces $L^{q,1}(\Omega)$, 
nowadays also called Yamazaki estimates.  
%{\rm i.e., } some suitable time-space integral estimate in the  Lorentz spaces. 

We derive a faster decay by a weighted $L^p$-$L^q$ estimate of the non-stationary Stokes problem stated in Theorem \ref{Stokes-w}, {\em viz.} with the power $\tau^{-(n/2p-n/2q)} (1+\tau)^{-(s-s_0)/2}$ where  $q\geq p$ \BLACK and $ s\geq s_0$ control radial weights. This important result was announced by Kobayashi and Kubo \cite{Kobayashi-Kubo} about ten years ago, and in \cite{Kubo}  Kubo gave an outline of the proof. 
%Although the result is  important and was announced in 2015, which had been first received in 2011,   
%as far as we investigated, the proof has not been published yet. 
In this paper, we give a rigorous proof and, as an important application, solve the time periodic problem for the Navier-Stokes equations on an exterior domain. 
%See {\em e.g.} Takahashi \cite{Takahashi} for the current status about weighted norm estimates including anisotropic weights for the Navier-Stokes-Oseen system. 
 
 %\vspace{1ex}

The proof of Theorem \ref{Stokes-w} is based on the decomposition of the problem to a domain far away from the boundary and an inner domain near the boundary by use of cut-off functions together with the Bogovski\u{\i} operator. 
%
%On the interior domain part, local decay estimates established by Dan, Kobayashi and Shibata \cite{DKS} will be exploited. 
A key point is that 
since domains are restricted to some ball for the inner domain problem, the weight functions have supports which are bounded uniformly. Thus we can apply estimates of \cite{DKS} directly to the inner domain part. 
%Note that usually local decay estimates require a hard and lengthy analysis.
%We  outline the proof of local decay estimates by \cite{DKS}.  
%Deriving the local decay estimate, a crucial property is the resolvent expansion formula around $\lambda=0$. 
%Roughly speaking, we decompose the resolvent problem to the whole space part and interior part by using the Bogovski\u{\i} operator.  Since the leading part is the whole space part, we first establish a resolvent expansion formula on the whole space case by concrete fundamental solution formulae due to the Fourier transform.  
The reduction to the whole space problem is also important for our aim. %(\BLUE Comment; This is also important for recent studies of free boundary problem by $\mathcal R$-boundedness by Shibata, Ogawa and Shimizu.) \BLACK
Concerning estimates on domains far away from the boundary as well as on several inner domains, we can apply estimates on the whole space. 

%Finally note that first we consider solutions to the heat equation and later turn towards estimates of the non-stationary Stokes problem which are based on results for the heat equation due to the difficulty of handling the pressure term. 

\vspace{1ex}

%(\BLUE Comment; In contrast to Kobayashi and Kubo's preprint, we need to take care of $\Delta \mathcal{U}= \Delta(u-v)$ later. Please see  \eqref{def-g} and Lemma \ref{est-G1} below. 
%If  we use the integral formula of $\mathcal U$ in \eqref{form-mathcalU} in the (old) GREEN part, we see $(t-\tau)^{-1}$ in the time integral which is not integrable.  I do not know how they avoid this problem because their proof does not write the computation. 

% They also do not write estimates including $t^{-\frac{1}{2}}$ like Lemma \ref{interior-domain1-2} and Lemma \ref{est-mathcalU-tsmall} which are important to estimate the time integrals. I found a shorter proof for the interior domain case in Proposition \ref{interior-domain2}. Maybe it is better to emphasize  this. 
 %They did not mention 
 %Remars \ref{rem1} but this has also an important rule and used to prove the weighted estimate.)\BLACK 
%\vspace{2ex}

The organization of the present paper is as follows. 
In Sect.~2 notation and basic properties of the modified Stokes operators  are described.
The focus of Sect.~3 is on the rigorous analysis of weighted $L^p$-$L^q$ decay estimates for the Stokes semigroup. 
Finally, in Sect.~4, global-in-time nonlinear estimates are derived to obtain time periodic solutions by Banach's fixed point theorem via a nonlinear integral equation as in \cite{Kozono-Nakao}.

\section{Preliminaries}

\subsection{Notation} 

Let $\Omega \subset \IR^n$ be either a    bounded or an unbounded standard domain, {\em i.e.} either a whole or (perturbed) half space, or an exterior domain. Then
$L^q(\Omega)$, $1\leq q \leq \infty$, denotes the usual Lebesgue space with norm $\|\cdot\|_{L^q(\Omega)}$.
%and $L^p(J;L^q(\Omega))$, $1\leq p,q\leq\infty$, the Lebesgue-Bochner space equipped with the norm $\|\cdot\|_{L^p(J;L^q(\Omega))}$.
Furthermore, $W^{k,q}(\Omega)$ is the standard Sobolev space with norm denoted by $\|\cdot\|_{W^{k,q}(\Omega)}$.
The spaces of test functions and of solenoidal test functions are denoted by
$C^{\infty}_0(\Omega)$  and $C^{\infty}_{0,\sigma}(\Omega):=\{\varphi \in C^{\infty}_{0}(\Omega): \div \varphi=0\}$, respectively. Then
$L^q_{\sigma}(\Omega):= \overline{C^{\infty}_{0,\sigma}(\Omega)}^{\|\cdot\|_{L^q(\Omega)}}$, $1<q<\infty$, is the $L^q$ space of weakly solenoidal vector fields, $u$, with vanishing normal component, $u\cdot \textsl{n}$, on $\partial\Om$.
Note that we use the same notation $L^q(\Omega)$ {\em etc.}  for functions and vector fields. 

For a standard domain $\Omega\subset\IR^n$  with boundary of class $C^1$ \BLACK there exists the Helmholtz projection, {\em i.e.,} the projection $P_q\colon L^q(\Omega)\to L^q_{\sigma}(\Omega) \subset L^q(\Omega)$, $1<q<\infty$, such that the kernel of $P_q$ equals the space $G_q(\Omega)$ of all weak gradient fields in $L^q(\Omega)$. The index $q$ in the projection is omitted if no confusion will occur; actually, $P_q u = P_ru$ for all vector fields $u\in L^q(\Omega)\cap L^r(\Omega)$. We also note that the adjoint of $P_q$ coincides with the Helmholtz projection $P_{q'}: L^{q'}(\Omega) \to L^{q'}(\Omega)$ where $\frac1q+\frac1{q'}=1$.

%Given the Helmholtz projection $P=P_q$ the Stokes operator \RED on an unbounded standard domain \BLACK $\Omega\subset \mathbb R^n$ is defined by
%$$
%	A_q := -P\Delta_{\Omega}\colon \D(A_q)= W^{2,q}(\Omega)\cap W^{1,q}_0(\Omega)\cap L^q_{\sigma}(\Omega)\subset L^q_{\sigma}(\Omega)\to L^q_{\sigma}(\Omega).
%$$
%	 Since $A_qu=A_pu$ for $u\in \D(A_q) \cap \D(A_p)$ for $1<q,p<\infty$, we also write $A$ for $A_q$ {\em etc.}

The weighted $L^q$ space with Muckenhoupt weight $w \in \mathscr{A}_q$, $1<q < \infty$, is defined as 
$$
L^q_w (\Omega) = \Bigg\{ u \in L^1_{{\rm loc}} (\bar{\Omega}): \|u\|_{L^q_w (\Omega)}= \Big(\displaystyle\int_{\Omega} |u|^q w \dx\Big)^{1/q}<\infty\Bigg\}.  
$$ 
Similarly, weighted non-homogeneous and homogeneous Sobolev spaces are denoted by by 
\begin{align*}
& W^{k,q}_{w}(\Omega) = \{u \in L^q_{w}(\Omega): \nabla^\alpha u\in L^q_{w}(\Omega),\, 
|\alpha | \leq k\}, \\
& \widehat{W}^{k,q}_{w}(\Omega) = \{u \in W^{k,1}_{\rm loc}(\Omega): \nabla^\alpha u\in L^q_{w}(\Omega),\, 
|\alpha | = k\}
\end{align*}
with the norms 
\begin{align*}
\|u\|_{W^{k,q}_{w}(\Omega)}= \Bigg(\sum_{|\alpha|\leq k} \|\nabla^\alpha u\|_{L^q_{w}(\Omega)}^q\Bigg)^{1/q}, \\
\|u\|_{\widehat{W}^{k,q}_{w}(\Omega)}= \Bigg(\sum_{|\alpha|= k} \|\nabla^\alpha u\|_{L^q_{w}(\Omega)}^q\Bigg)^{1/q} 
\end{align*}
for $1<q < \infty$, $k \in \mathbb{N}_0$ and $w\in \mathscr{A}_q$, respectively; here $\alpha\in \IN_0^n$ is a multi-index, and $|\alpha|$ is the length of $\alpha$. \BLACK
Finally, we define for smooth domains the weighted Sobolev space $W^{k,q}_{0,w}(\Omega)$, $k\in\IN$, with the vanishing boundary condition by  
\begin{align*}
W^{k,q}_{0,w}(\Omega) = \{u \in W^{1,q}_{w}(\Omega): u=0 ,\ldots,\nabla^{k-1}u=0\mbox{ on }\del \Omega\},  
\end{align*}
where $\nabla^{\ell}u$ denotes the set of all partial derivatives of order $|\alpha|=\ell$. \BLACK

A very general extension property for homogeneous weighted Sobolev spaces is proved by Chua \cite{Chua} for general $(\varepsilon,\infty)$ domains, including smooth exterior domains.

\begin{lem}\label{extention} {\rm(\cite[Theorem 1.5]{Chua}, \cite[Theorem 2.2]{FroehII})}
Let $1<q <\infty$, $w\in \mathscr{A}_q$, and $k_1<\ldots<k_N\in\mathbb N_0$. Further let $\Omega\subset\IR^n$ be an exterior Lipschitz domain. \BLACK Then there exists a linear extension operator $E: \bigcap_{i=1}^N \widehat{W}^{k_i,q}_{w}(\Omega)\rightarrow  \bigcap_{i=1}^N \widehat{W}^{k_i,q}_{w}(\mathbb{R}^n)$ such that 
$$
\|\nabla^{k_i} Eu \|_{L^q_w(\mathbb{R}^n)} \leq C_i \|\nabla^{k_i} u\|_{L^q_w(\Omega)}
$$
for all $i=1, \ldots,N$ and $u \in \bigcap_{i=1}^N \widehat{W}^{k_i,q}_{w}(\Omega)$. 
\end{lem}

Let $M$ denote the (centered) Hardy-Littlewood maximal operator, {\em i.e.,} 
$$ Mu(x) = \sup_{r>0} \frac{1}{|B_r(x)|} \int_{B_r(x)} |u(y)| \dy, $$
where $B_r(x) = \{y\in\IR^n: |y-x|<r\}.$ %((had not been defined yet)) 
\BLACK 
By \cite[Theorems 2.1.6 and 7.1.9]{GrafakosI} $M$ is a bounded operator not only on $L^q(\IR^n)$ for each $1<q<\infty$, but also on $L^q_w(\IR^n)$ for each $w\in \mathscr A_q$: %\RED ((I deleted the possibility $q=\infty$ since otherwise we must differ between $q<\infty$ and $q=\infty$ from time to time)) \BLACK
\begin{equation}\label{Muu}
\|Mu\|_{L^q_w(\IR^n)} \leq C\|u\|_{L^q_w(\IR^n)}. \end{equation}

For later use we mention the following pointwise estimate: Let $\rho:\IR^n \to [0,\infty)$ be a decreasing, radially symmetric smooth function satisfying $\int_{\IR^n} \rho\dx = 1$, and put $\rho_\varepsilon(x) = \varepsilon^{-n} \rho\big(\frac{x}{\varepsilon}\big)$ for any $\varepsilon>0$. Then by \cite[Theorem 2.1.10]{GrafakosI}, in the pointwise sense,
\begin{equation}\label{rho-Mu}
\sup_{\varepsilon>0}\, (\rho_\varepsilon*|u|)(x) \leq Mu(x). 
\end{equation}
 
As a further property of Muckenhoupt weights Riesz operators and, more generally, Calder\'on-Zygmund operators are bounded on $L^q_w(\IR^n)$ for $w\in \mathcal A_q(\IR^n)$, see \cite[Theorem 7.4.6]{GrafakosI}. 
This fact allows to prove that 
\begin{equation}\label{Delta12}  
\|\nabla v\|_{L^q_w(\IR^n)} \leq c\big\|(-\Delta)^{\frac12} v\big\|_{L^q_w(\IR^n)}
\end{equation}
with norm bound $c=c_{q,w}>0.$ Indeed, we use the identity $i \xi_j=i\frac{\xi_j}{|\xi|}\,|\xi|$ with multiplier function $i\frac{\xi_j}{|\xi|}$. Then the boundedness of Riesz operators on $L^q_w(\IR^n)$ yields \eqref{Delta12}.
%of H\"ormander-Mikhlin and Marcinkiewicz type, see \cite[Theorem 7.4.6]{GrafakosI}, \cite[Theorems 3 and 4]{Kurtz}, respectively.

%((\BLACK I tried to prove the norm equivalence \RED $ \|\nabla v\|_{L^q_s(\IR^n)} \sim \big\|(-\Delta^{\frac12}) v\big\|_{L^q_s(\IR^n)}$. But for the reverse direction we use e.g. that $|\xi| = \sum_{j=1}^n \big\{\big(\frac{|\xi|}{\sum_i |\xi_i|} \big) \frac{|\xi_j|}{\xi_j} \big\} \xi_j$. Here $\frac{|\xi_j|}{\xi_j} = \frac{\xi_j}{|\xi_j|}$ describes a one-dimensional Riesz operator which is bounded only for weights in the Muckenhoupt class $\mathcal A_q(\IR^1)$, but not 
%$\mathcal A_q(\IR^n)$. I do not know whether this %drawback can be solved by other tools. Anyway, it will not be used. Then also the newly added reference Kurtz [35]) can be deleted again.))  

Weight functions used in this article are of the radially symmetric type $w=\langle x\rangle^{sq} = \big(\sqrt{1+|x|^2}\big)^{s q}$  defining the weighted space 
$$
L^q_s (\Omega) := \bigg\{ u \in L^1_{\rm  loc} (\bar{\Omega}); \|u\|_{L^q_s (\Omega)}= \Big(\displaystyle\int_{\Omega} |u|^q \langle x\rangle^{sq} \dx\Big)^{1/q}<\infty\bigg\}  
$$
for $1<q < \infty$,  {\em i.e.,} 
we replace the index $w$ by $s$. 
Similarly, weighted Sobolev spaces $W^{k,p}_{s}$, $W^{k,p}_{0,s}$ are defined. 
Concerning function spaces of solenoidal vector fields let $L^q_{\sigma,s}(\Omega)$ denote the closure of $C^\infty_{0,\sigma}(\Omega)$ with respect to the norm of $L^q_s(\Omega)$. As already mentioned, $w =\langle x\rangle^{sq}\BLACK \in \mathscr A_q(\IR^n)$ if and only if 
$-\frac{n}{q} < s < n\big(1-\frac{1}{q}\big) = \frac{n}{q'}$. 
%\BLUE( I think that it is better to avoid using $\frac{n}{q'}$ because in the proofs we always use $n\big(1-\frac{1}{q}\big)$. %\RED If you prefer, we can do, but I already used $q'$ also in proofs and assertions several times since then assumptions look more symmetric.)

The following embedding result for weighted Sobolev spaces is well-known for $\IR^n, \IR^n_+$ and bounded domains, but, to the best of our knowledge, has not been described for homogeneous Sobolev spaces on \BLACK exterior domains. 
Therefore, we will give a short proof  for the space $\widehat W^{1,q}_{0,s}(\Omega) = \{U\in \widehat W^{1,q}_{0,s}(\Omega): U|_{\partial\Omega}=0\}$; here $U|_{\partial\Omega}=0$ in the sense that there exists a representant $u\in W^{1,q}_{\rm loc}(\bar\Omega)$ of $U$ such that $u=0$ on $\partial\Omega$. %((was not yet defined)) 
\BLACK

\vspace{1ex} 
\begin{prop}\label{embed-weighted}
Let $1<q<\infty$, $-\frac{n}{q} < s < \frac{n}{q'}$, and let $q^*=\frac{nq}{n-q}$. 
Then there holds the embedding $\widehat W^{1,q}_{0,s}(\Omega) \hookrightarrow L^{q^*}_s(\Omega)$, {\em i.e.}, $\|u\|_{L^{q^*}_s} \leq C\|\nabla u\|_{L^{q}_s }$ for all $u\in \widehat W^{1,q}_{0,s}(\Omega)$. 

%(i) The set $C^\infty_{0}(\overline\Omega)$ is dense in the homogeneous weighted Sobolev space $\widehat W^{k,q}_{s}(\Omega)$ and $C^\infty_{0}(\Omega)$ is dense in 
%$$ \widehat W^{k,q}_{0,s}(\Omega) := \big\{[u]\in \widehat W^{k,q}_{s}(\Omega): \exists u\in[u] \textrm{ such that } u|_{\partial\Omega} =0\big\}.$$

%(ii) Assume that $1<q<n$ and let $q*=\frac{nq}{n-q}$. Then there holds the embedding $\widehat W^{k,q}_{0,s}(\Omega) \hookrightarrow L^{q*}_s(\Omega)$, {\em i.e.}, there holds $\|u\|_{L^{q*}_s} \leq C\|\nabla u\|_{L^{q}_s}$ for all $u\in \widehat W^{k,q}_{0,s}(\Omega)$.
\end{prop}

%\vspace{1ex}
\begin{proof}
First we extend $u\in \widehat W^{1,q}_{0,s}(\Omega)$ by $0$ to a function in $\widehat  W^{1,q}_{s}(\IR^n)$, 
%(\BLUE I think that because derivative is destroyed on the boundary, we should use the extension operator $\tilde{u}=Eu$ on the homogeneous weighted Sobolev spaces by Chua as  
%$\|\tilde{u}\|_{L^{q^*}_s(\mathbb{R}^n)} \leq \|\nabla \tilde{u}\|_{L^{q}_s(\mathbb{R}^n)}\leq C\|\nabla u\|_{L^{q}_s(\Omega)}.$) 
called $u$ again. By \cite[Corollary 4.3]{FS} there exists a sequence $\{u_k\}\subset C_{0}^\infty(\IR^n)$ such that $u_k\to u$ in $\widehat  W^{1,q}_{s}(\IR^n)$ as $k\to\infty$. Then a weighted embedding estimate in its inhomogeneous form, see \cite[Theorem 9]{MuWh}, states that
\begin{align}\label{Sobolev-weight}
\|u_k\|_{L^{q^*}_s(\mathbb{R}^n)} \leq C\big(\|u_k\|_{L^{q}_s(\mathbb{R}^n)} + \|\nabla u_k\|_{L^{q}_s(\mathbb{R}^n)}\big).
\end{align}
However, the proof shows that for $u_k\in C^\infty_0(\IR^n)$ the term $\|u_k\|_{L^{q}_s(\mathbb{R}^n)}$ may be omitted, {\em i.e.}, we get the homogeneous estimate $\|u_k\|_{L^{q^*}_s (\mathbb{R}^n)} \leq C\|\nabla u_k\|_{L^{q}_s (\mathbb{R}^n)}$. Since this estimate also holds for $u_k-u_\ell$, we conclude that $\{u_k\}$ converges in $L^{q^*}_s(\mathbb{R}^n)$ to $u$. Hence, as $k\to \infty$, we obtain that
$\|u\|_{L^{q^*}_s(\IR^n)} \leq C\|\nabla u\|_{L^{q}_s(\IR^n)}$ and the same estimate on $\Omega$ replacing $\IR^n$. \BLACK
\end{proof}
\BLACK

%(i) For the denseness of $C^\infty_{0}(\overline\Omega)$ in $\widehat W^{k,q}_{s}(\Omega)$ we refer to \cite[Lemma 5.1]{FS96}. We note that this result holds even for any weight $w\in\mathscr A(\IR^n)$ by \cite[Corollary 4.1]{FroehI}. Then the assertion for $\widehat W^{k,q}_{0,s}(\Omega)$ is immediate \RED (not checked)\BLACK. 

%(ii) The embedding is proved in \cite[Theorem 9]{MuWh} in the inhomogeneous form that 
%$$\|u\|_{L^{q*}_s} \leq C\big(\|u\|_{L^{q}_s} + \|\nabla u\|_{L^{q}_s}\big).$$ 
%However, the proof shows that for $v\in C^\infty_0(\Omega)$ there holds the estimate $\|v\|_{L^{q*}_s} \leq C\|\nabla u\|_{L^{q}_s}$. Hence, if $\{u_k\}_k\subset C^\infty_0(\Omega)$ converges to $u$ in $\widehat W^{k,q}_{0,s}(\Omega)$, then $\{u_k\}_k$ is a Cauchy sequence in $L^{q*}_s(\Omega)$. Passing to the limit as $k\to\infty$, the embedding is proved.
%\end{proof}

\vspace{2ex}
We define the Bogovski\u{\i} operator on the annulus
$D_R = \{x: R<|x|<R+1\},\, R>0$,  such that $\IR^n\setminus\Omega\subset B_R=\{x: |x|<R\}$. 
For $1<q < \infty$ and $k \in \mathbb{N}_0$, let 
$$
W^{k, q}_m(D_R) = W^{k,q}_0(D_R) \cap L^q_m(D_R),  
$$
 where $L^q_m(D_R)$ is the set of all $u\in L^q(D_R)$ such that the integral mean on $D_R$ satisfies $\int_{D_R} u\dx=0$.
Then there exists a linear bounded operator 
$\mathbb{B}: W^{k,q}_m(D_R) \rightarrow {W}^{k+1,q}(\mathbb{R}^n)$, called the Bogovski\u{\i} operator, satisfying that 
\begin{align}\label{B}
\div \mathbb{B}[f]\big|_{D_R} = f,  \BLACK\ \ \supp \mathbb{B}[f] \subset \overline{D_R} \;\; \mbox{ and } \ 
\|\nabla^{k+1}\mathbb{B}[f]\|_{L^q(\mathbb{R}^n)} \leq C \|\nabla^k f\|_{L^q(D_R)}
\end{align}
for $f \in W^{k,q}_m(D_R)$.  
For details, see {\em e.g.} \cite{BH, Galdi-1}. 
From \cite[Chapter III.3]{Galdi-1} we also get that 
\begin{align}\label{B-2}
\|\mathbb{B}[\div f]\|_{L^q(\mathbb{R}^n)} \leq C\|f\|_{L^q(D_R)}, 
\end{align}
{\em i.e.,} $\mathbb{B}$ admits an extension to an operator from $W^{-1,q}$ to $L^q$. 
Due to \cite{GHM}, $\mathbb{B}$ can be  extended to Sobolev spaces $W^{\alpha,p}_0$ \BLACK with negative exponents $\alpha>-2+\frac1p$.

%\begin{lem}\label{morrey}{(Morrey type inequality on weighted spaces)} 
%For $n<p< \infty$ and $0< \beta <s-1$. Then it holds for $f \in \widehat{W}^1_s(\mathbb{R}^n)$ that
%For $n<p< \infty$ and the weight function $w(x)=\langle x\rangle^{sp}$ with $s>1$, it holds that for $f \in \widehat{W}^1_s(\mathbb{R}^n)$ and $0< \beta <s$ satisfying $s-\beta>1$, 
%$$
%\|f\|_{L^\infty_\beta (\mathbb{R}^n)} \leq C \|\nabla f\|_{L^p_s(\mathbb{R}^n)}.  
%$$
%\end{lem}

%\vspace{1ex}

%\begin{proof}
%The Morrey inequality \RED (citation!) \BLACK implies that 
%$$
%\|f\|_{L^\infty_\beta (\mathbb{R}^n)} \leq C(\| \langle x\rangle^{\beta }f\|_{L^p(\mathbb{R}^n)}+ \|\langle x\rangle ^{\beta}\nabla f\|_{L^p(\mathbb{R}^n)})
%\leq C(\| f\|_{L^p_\beta(\mathbb{R}^n)}+ \|\nabla f\|_{L^p_s(\mathbb{R}^n)})
%. 
%$$
%Since $s > \beta+1$ \RED (What about $s=\beta+1$?) \BLACK, the Hardy inequality in weighted spaces \RED (citation!) \BLACK derives that
%$$
%\| f\|_{L^p_\beta(\mathbb{R}^n)}\leq C \Big\| \frac{f}{|x|}\Big\|_{L^p_{\beta+1}(\mathbb{R}^n)} \leq \|\nabla f\|_{L^p_s(\mathbb{R}^n)}.
% $$
%\RED Do we need this Lemma? \BLACK\end{proof}

\vspace{2ex}

%We note that homogeneous weighted Sobolev space are important to use extension operators. 
%Indeed, we have 

%We note that the above lemma holds for more general domains, {\em e.g. } for unbounded $(\epsilon, \infty)$ domains as reference domains $\Omega_0$. In this paper the homogeneous estimate in Lemma \ref{extention} has \RED important rules rather than \BLUE $(\epsilon, \infty)$ domains.\RED (what do you mean by this ?) \BLUE $\rightarrow$ I modified  the word "appication".   For future works, for example, study about the Lipschitz boundary, it maybe be needs. \BLACK 

\subsection{The Stokes operator in spaces with Muckenhoupt weights}

Given the Helmholtz projection $P=P_q$, $1<q<\infty$, the Stokes operator on an unbounded standard domain $\Omega\subset \IR^n$ is defined by 
$$
	A_q := -P\Delta_{\Omega}\colon \D(A_q)= W^{2,q}(\Omega)\cap W^{1,q}_0(\Omega)\cap L^q_{\sigma}(\Omega)\subset L^q_{\sigma}(\Omega)\to L^q_{\sigma}(\Omega).
$$
Since $A_qu=A_pu$ for $u\in \D(A_q) \cap \D(A_p)$ for $1<q,p<\infty$, we also write $A$ for $A_q$ {\em etc.}

Let us recall several results on the Stokes operator on an exterior domain $\Omega\subset \IR^n$ with boundary of class $C^{2 }$ in $L^q$ spaces without and with weights of Muckenhoupt class $\mathscr{A}_q$. All results in this subsection hold for all dimensions $n\geq 3$. \BLACK

First of all, the Stokes operator $A_q$, $1<q<\infty$, is injective and has dense domain $\mathcal D(A_q)$ and  dense range $\mathcal R(A_q)$ in $L^q_\sigma(\Omega)$. Its adjoint $(A_q)^*$ coincides with the Stokes operator $A_{q'}$ on $L^{q'}_\sigma(\Omega)$.
For $\lambda$ in the sector $\Sigma_\omega =\{0\neq \mu\in\IC: |{\rm arg}\, \mu|<\omega\}$ where $0<\omega<\pi$  the Stokes resolvent problem
$\lambda u + A_q u =f$ 
has a unique solution $u\in {\D}(A_q)$ with the resolvent estimate
\begin{align}\label{equ:rse}
	\|\lambda u\|_{L^q(\Omega)}  + \|A_q u\|_{L^q(\Omega)}  \leq c_\omega \|f\|_{L^q(\Omega)}.
\end{align}
Hence $\|\lambda(\lambda+A_q)^{-1}\|_{\mathcal L(L^q_\sigma(\Omega))} \leq c_\omega$, $A_q$ is a sectorial operator, and $\mathcal D(A_q^k) \cap \mathcal R(A_q^k)$ is dense in $L^q_\sigma(\Omega)$ for any $k\in\IN$. 

In addition, the following more detailed resolvent estimate  
\begin{align}\label{equ:rse2}
	\|\lambda u\|_{L^q(\Omega)}  +  |\lambda|^\frac12 \|\nabla u\|_{L^q(\Omega)}  +\|\nabla^2 u\|_{L^q(\Omega)}  \leq c \|(\lambda+A_q)u\|_{L^q(\Omega)}
\end{align}
for $u\in\mathcal D(A_q)$ and all $\lambda\in \Sigma_\omega$ holds. In fact, we get \eqref{equ:rse2} with a constant $c=c_{q,\omega,\delta}>0$ when $|\lambda|>\delta>0$. If $1<q< n/2$,
%\begin{equation}\label{qOmega}
%	q_0(\Omega) = \begin{cases} \infty,\quad & \textrm{if } \Omega=\IR^n ,\; \Om\textrm{ is a bounded domain},\\
%		n,\quad & \textrm{if } \Omega = \IR^n_+,\\
%		\frac{n}{2},\quad & \textrm{if } \Omega\; \textrm{ is an exterior domain},
%	\end{cases}	
%\end{equation}
there exists a constant $c=c_{q,\omega}>0$ such that \eqref{equ:rse2} even holds for all $\lambda\in \Sigma_\omega$.  
%The restriction $q_0(\Omega)$ is important and is given by 
%Especially, in our case where $q_0(\Omega)=n/2$, this restriction is important, and requires special care of the analysis below.   

Concerning the exterior domain case, the following $L^p$-$L^q$ decay estimates without weights are obtained by 
\cite{DKS, Iwashita, Shi-monogr}. 

\vspace{1ex}
\begin{thm}\label{Stokes-est-exterior} 
Let $\Omega\subset \IR^n$, $n\geq 3$, be an exterior domain with $C^2$ boundary, $1< p \leq q < \infty$. 
Then for $u \in L^p_{\sigma}(\Omega)$ the following estimates holds:  
\begin{align}\label{e-tAt}
\begin{aligned}
\big\|e^{-t A} u\big\|_{L^q (\Omega)} & \leq C t^{-\frac{n}{2}\big(\frac{1}{p}-\frac{1}{q}\big)}\| u \|_{L^p(\Omega)}, \ \ \mbox{for} \ \  0<t, \\
\big\| \nabla^\alpha  e^{-t A} u\big\|_{L^q(\Omega) } & \leq C \Big(t^{-\frac{n}{2}\big(\frac{1}{p}-\frac{1}{q}\big)-\frac{|\alpha|}{2}} +  t^{-\frac{n}{2p}}\Big)\| u \|_{L^p(\Omega)},  \quad |\alpha|=1,2, \ \  
\mbox{for} \ \ 1< t, \\
\big\|\nabla^\alpha e^{-t A} u\big\|_{L^q(\Omega) } & \leq C t^{-\frac{n}{2}\big(\frac{1}{p}-\frac{1}{q}\big)-\frac{|\alpha|}{2}} \| u \|_{L^p(\Omega)}, 
\quad |\alpha|=1,2, \ \ \mbox{for} \ \  0< t \leq 1.   
\end{aligned}
\end{align}   
\end{thm}

\begin{rem}\label{01t}
{\rm
The improvement $\eqref{e-tAt}_3$  compared to the estimate \BLACK $\eqref{e-tAt}_2$ is obtained by 
the resolvent estimate \eqref{equ:rse2} and Dunford's integral calculus to represent $e^{-t A}$. Recall that $e^{-tA}$ can be expressed as a Cauchy integral $\frac{1}{2\pi i} \int_{\Gamma_{t,\theta}} e^{-\lambda t} (\lambda-A)^{-1} \dlambda$ along the unbounded contour $\Gamma_{t,\theta}$ consisting of two rays $re^{\pm i\theta}$, $r\geq \frac1t$, $\theta\in (0,\frac{\pi}{2})$, and a circular arc of radius $\frac1t$ surrounding the origin in the left complex half  plane. For $0<t<1$ the contour $\Gamma_{t,\theta}$ can be chosen with radius $1$ independent of $t$. \BLACK Thus \eqref{equ:rse2} is exploited with the same constant $c=c_{q,\omega,1}>0$ independent of $\lambda \in \Gamma_{1,\theta}$.
}
\end{rem}
\BLACK

\vspace{1ex}

In the Muckenhoupt class $\mathscr{A}_q$, 
Sohr and the first author of this article \cite{FS} analyze the Helmholtz projection and \BLACK the Stokes operator to obtain the following properties. 

\vspace{1ex}
\begin{thm} {\em \cite[Theorem 1.3]{FS}}
Let $n\geq 3$ and $\Omega \subset \mathbb{R}^n$ be an exterior domain with $C^1$ boundary, let $q \in (1, \infty)$ and $w \in \mathscr{A}_q$. 
\begin{enumerate}
\item[{\rm (i)}] $L^q_w(\Omega)$ has a unique algebraic and topological decomposition 
$$
L^q_w(\Omega) = L^q_{\sigma,w}(\Omega) \oplus
G^q_w (\Omega), 
$$
where %$L^q_{\sigma,w}(\Omega)$ is the closure of $C^\infty_{0,\sigma}$ with respect to the norm $\|\cdot\|_{L^q_w(\omega)}$ and 
$ G^q_w (\Omega) = \{\nabla \pi \in L^q_w(\Omega): \pi \in L^1_{\loc}(\bar{\Omega})\} $.
In particular, there exists a unique bounded projection
$P= P_{q,w}: L^q_{w}(\Omega) \rightarrow L^q_{\sigma,w}(\Omega)$ with null space  $G^q_w (\Omega)$.

\item[{\rm (ii)}] $(P_{q,w})^*=P_{q', w'}$ and $(L^q_{\sigma,w}(\Omega))^*= L^{q'}_{\sigma,w'}(\Omega)$, where $q'= q/(q-1)$ and $w'= w^{-1/ (q-1)}$.
\end{enumerate}
\end{thm}

\vspace{2ex}

Given the Helmholtz projection $P_{q,w}$, the Stokes operator $A_{q,w}$ is defined by 
$$
A_{q,w}=-P_{q,w}\Delta_{\Omega}\colon \D(A_{q,w}) = W^{2,q}_w(\Omega)\cap W^{1,q}_{0,w}(\Omega)\cap L^q_{\sigma,w}(\Omega)\subset L^q_{\sigma,w}(\Omega)\to L^q_{\sigma,w}(\Omega).
$$
 Since $A_{q,w}u = A_{q,\tilde w}u$ for $u\in \D(A_{q,w}) \cap \D(A_{q,\tilde w})$, a dense subset of $\D(A_q)$, we will also write $A$ for $A_{q,w};$ here $w,\tilde w$ are any weights in $\mathscr{A}_q$.  \BLACK

\vspace{1ex}

\begin{thm}\label{res-weighted} {\em \cite[Theorems 1.5 and 5.5]{FS}}
Let $n \geq 3$, $\Omega \subset \mathbb{R}^n$ be an exterior domain with boundary of class $C^2$, let $q \in (1, \infty)$ and $w\in \mathscr{A}_q$. 
%Further let $\Sigma_\omega$ denote the complex sector $\Sigma_\omega =\{0\neq \mu\in\IC: |{\rm arg}\, \mu|<\omega\}$ with $0<\omega<\pi$,
\begin{enumerate}
\item[{\rm (i)}]
For $\lambda\in\Sigma_\omega$, $0<\omega<\pi$, satisfying $|\lambda|\geq\delta>0$ the Stokes resolvent problem
$\lambda u + A_{w,q} u =f$ 
has a unique solution $u\in {\D}(A_{w,q})$ satisfying the resolvent estimate
\begin{align}\label{equ:rse-w}
	\|\lambda u\|_{L^q_w(\Omega)}  + \|A_{w,q} u\|_{L^q_w(\Omega)}  \leq c_{\omega,\delta} \|f\|_{L^q_w(\Omega)}.
\end{align}
\item[{\rm (ii)}]
%For $u\in\mathcal D(A_{w,q})$ and all $\lambda\in \Sigma_\omega$ but restricted $|\lambda| \geq \delta$. 
If $w(x)=(1+|x|)^\alpha$    
with $2q-n < \alpha < n(q-1)$, {\em i.e.} $w\in \mathscr{A}_q$, the constant $c_{\omega,\delta}$ can be replaced by $c_{\omega}$, a constant independent of $\delta$. 
\item[{\rm (iii)}]
Let $x_0 \in \mathbb{R}^n \setminus \overline\Omega\BLACK$ and %(correction of misprints)\BLACK
$$
w^{s/q}|\cdot-x_0|^{-\gamma s} \in  \mathscr{A}_s, \BLACK  $$
where 
$\gamma=n\big(\frac{2}{n}+\frac{1}{s}-\frac{1}{q}\big) \geq 0 $ and $s \geq q$. Then 
the resolvent estimates 
\begin{align}\label{equ:rse-w2}
	\|\lambda u\|_{L^q_w(\Omega)}  +  |\lambda|^\frac12 \|\nabla u\|_{L^q_w(\Omega)}  +\|\nabla^2 u\|_{L^q_w(\Omega)}  \leq c \|(\lambda+A_{q,w})u\|_{L^q_w(\Omega)}
\end{align}
holds for $u\in\mathcal D(A_{w,q})$ and all $\lambda\in \Sigma_\omega$ uniformly.  
In particular, if  $w=\langle x\rangle^\alpha$ \BLACK and
$$
n\geq 3,  \ \ 2q-n < \alpha < n(q-1), $$
then \eqref{equ:rse-w2} is satisfied. 
\end{enumerate}
\end{thm}

%\vspace{1ex}

\begin{rem}
We apply Theorem \ref{res-weighted}  (ii)
with both $w=(\sqrt{1+|x|^2})^{s q}$ and $w=(1+|x|)^{s q}$ defining  equivalent norms on $L^q_w(\IR^n)$.
Hence we need that
\begin{align}\label{crucial-restriction}
n \geq 3,  \ \ 
2-\frac{n}{q} < s < n\Big(1-\frac{1}{q}\Big) = \frac{n}{q'}. 
\end{align}
%If $q< \frac{n}{2}$, any positive exponent $s$ with 
%$s< n\big(1-\frac{1}{q}\big)$ satisfies \eqref{crucial-restriction}. Hence we can take small $s>0$ below.
%The restriction $q< \frac{n}{2}$ corresponds to 
%\eqref{qOmega} to obtain the uniform resolvent estimates \eqref{equ:rse-w2}.
\end{rem}

\section{$L^p$-$L^q$ estimates of the Stokes semigroup in weighted spaces}

The following two-weight $L^p$-$L^q$ estimates of the Stokes semigroup are the crucial results in this section,  using the weights $w(x)=\langle x\rangle^{s_0}$ and $w(x)=\langle x\rangle^{s}$ where $-\frac{n}{q} < s_0\leq s < n(1-\frac{1}{p})=\frac{n}{p'}$, $p\leq q$, and $s\geq 0$.  Only for the whole space results in Proposition \ref{whole-est-basis} the restriction $s\geq 0$ is unnecessary.  All results in Sect. 3 hold for $n \geq 3$. \BLACK

%\begin{thm}\label{Stokes-w1} Let $n \geq 2$ and $1< p \leq q <\infty$. Let $-n/q < s < n(1-\frac{1}{p})$ for the exterior domain case $\Omega_0$. Then for $a \in L^p_{\sigma, 1} (\Omega)$ we have
%$$
%\|\langle x \rangle^s \nabla^{\alpha} e^{-A(0)}a\|_{L^q } \leq C t^{-\frac{n}{2}\Big(\frac{1}{p}-\frac{1}{q}\Big)-\frac{|\alpha|}{2}}\|\langle x \rangle^s a \|_{L^p}
%$$
%for $0<t <2$ and $|\alpha|=0,1$. 

%\end{thm}

\begin{thm}\label{Stokes-w} 
Let $\Omega \subset\IR^n$, $n \geq 3$, be an exterior domain and $1< p \leq q <\infty$. Further assume $-\frac{n}{q} \leq s_0 \leq s < \frac{n}{p'}$ and $s\geq 0$.
%\BLUE In \eqref{etA>1}, it holds for $-n/q <  s_0 \leq s < \frac{n}{p'}$ with $0 \leq s$ because we use the condition $0 \leq s_0$ in only pp.23.  
%We take $s_0=s_2 >0$ in our main results. Please also find blue parts in pp.23 and 24. \RED I do not yet understand: What is needed in (3.1), (3.2) and (3.3)? In \eqref{etA>1},   $-n/q < s_0 \leq s < \frac{n}{p'}$ and $s\geq 0$, but  $0 \leq s_0 \leq s < \frac{n}{p'}$ in \eqref{n-etA>} ???
%\BLUE (Yes.)\BLACK 

Then for $u_0 \in L^p_{\sigma,s}(\Omega)$ the following estimates hold: For $t>1$, 
\begin{equation}
\big\|e^{-t A} u_0\big\|_{L^q_{s_0}(\Omega) } \leq C t^{-\frac{n}{2}\big(\frac{1}{p}-\frac{1}{q}\big)}t^{-\frac{s-s_0}{2} }\|u_0\|_{L^p_s(\Omega) }, \label{etA>1}\end{equation}
and, if additionally $s_0 \geq 0$,
\begin{equation}
\big\|\nabla  e^{-t A} u_0\big\|_{L^q_{s_0}(\Omega)} \leq C \Big(t^{-\frac{n}{2}\big(\frac{1}{p}-\frac{1}{q}\big)-\frac{1}{2}}t^{-\frac{s-s_0}{2}} +t^{-\frac{n}{2p}}t^{-\frac{s}{2} }\Big)\|u_0\big\|_{L^p_s(\Omega)}.\label{n-etA>}
\end{equation}
\medskip

%Hence the interpolation inequality derives that 
%$$
%\|E_t \ast u_0\|_{L^p_{s_0}} \leq Ct^{\frac{s-s_0}{2}}\|u_0\|_{L^p_{s_0}}
%$$
%for $0 \leq s \leq s_0$. Therefore the same argument as that in our proof yields the estimate on the whole space. 
%Concerning another terms,  almost of them have the cut off property. We combine the whole space estimate with the cutoff property in the following proof. The  Proof of prop 3.9 combines the whole space estimate with the cut off property of $G$. I highlighted all of corresponding parts by BLUE in the following proofs. Prop.3.9 with Lemma 3.11 (and the local decay estimates)  imply that we obtain \eqref{etA>1-2} on the exterior domain case. )
%(We also need a condition about $q_1$ and $s_1$, \eqref{q_1-s_1} below.)
%
For $t<2$ and any multi-index $\alpha\in\IN_0$ with $|\alpha|=0,1$  such that $\frac{n}{2}\big(\frac{1}{p}-\frac{1}{q}\big)+\frac{|\alpha|}{2}<1$ it holds that 
\begin{equation}\label{etA<}
\big\|\nabla^{\alpha} e^{-tA} u_0 \big\|_{L^q_{s_0}(\Omega) } \leq C t^{-\frac{n}{2}\big(\frac{1}{p}-\frac{1}{q}\big)-\frac{|\alpha|}{2}}\|u_0\|_{L^p_s(\Omega)}.   
\end{equation} 
\end{thm}

\vspace{1ex}

To prove Theorem \ref{Stokes-w} we start with the following two propositions.  Recall that $R>0$ is chosen such that $\IR^n\setminus\Omega\subset B_R$. 

\vspace{1ex}

\begin{prop}\label{interior-est-base}
Let $n \geq 3$ and $1<q<\infty$. There exists a positive constant 
$C_{q,R}$ such that 
\begin{align}\label{local-decay}  
\|e^{-t A}u_0\|_{L^q(\Omega_{R})} & \leq 
    C_{q,R}t^{-\frac{n}{2}}\|u_0\|_{L^q(\Omega)}, \quad t>1,\\
\|\del_t^j \nabla^\alpha e^{-t A}u\|_{L^q(\Omega_{R})} & \leq 
    C_{q,R}t^{-\frac{n}{2}-j}\|u_0\|_{L^q(\Omega)}, \quad t>1, \label{local-decay-k}
\end{align}
%\RED((not $t^{-\frac{n}{2}-j}$ in (3.5) as in [7, Thm. 1.1.1]??) \BLUE $\rightarrow$ I added the power $-j$.)  \BLACK 
for 
$u_0 \in L^q_{\sigma}(\Omega)$ in \eqref{local-decay} and for $u_0 \in L^q_{\sigma,R}(\Omega) = \{v\in L^q_{\sigma}(\Omega): \supp v \subset B_R\}$  in \eqref{local-decay-k}; here $\Omega_{R} = \Omega\cap B_R$, $j=0,1$ and $|\alpha|=0,1,2$.
\end{prop}

\vspace{1ex}

See {\em e.g.} Iwashita (\cite[Theorem 1.2]{Iwashita}) for the proof of \eqref{local-decay} and 
the monograph %(in Japanese) 
of Shibata \cite[Theorem 9.50]{Shi-monogr}.  Actually, \cite[Theorem 9.50]{Shi-monogr} is proved by a similar argument as used by Shibata and Shimizu \cite[Theorem 1.1]{Shibata-Shimizu} for \eqref{local-decay-k} dealing with the Neumann boundary condition. 
 A crucial idea is the expansion formula of the resolvent near the origin, \cite[(9.253)]{Shi-monogr}, based on the parametrix of the problem, \cite[(9.202)]{Shi-monogr}. 
The argument follows a similar discussion in \cite[Theorem 5.1]{Shibata-Shimizu} based on the parametrix  technique \cite[(5.18)]{Shibata-Shimizu}.
More precisely, \cite{Shibata-Shimizu} needs a 
detailed analysis of perturbation terms, $M(\lambda)$, in $\Omega_{R+1}$ in \cite[(5.4)]{Shibata-Shimizu} for the Navier-boundary condition which is not needed for the non-slip boundary condition.
The result by Iwashita yields the decay $t^{-\frac{n}{2q}}\|u\|_{L^q(\Omega)}$, but we note that $t^{-\frac{n}{2}} \leq t^{-\frac{n}{2q}}$ for $t>1$. The same restriction in \eqref{local-decay} applies to \cite[Theorem 1.1.1]{DKS} by Dan, Kubo and Shibata;  concerning \eqref{local-decay-k} these authors consider the cases $n=2$ and $n=3$ in \cite[Chapters 2 and 1]{DKS}. 
Detailed proofs of \eqref{local-decay-k} for any $n \geq 3$ are given in \cite{Shi-monogr}. Finally,  Enomoto and Shibata  \cite{EnoS} study the local decay estimate for the Oseen semigroup and $n \geq 3$. \\
\BLACK

In the following, let  $P_0$ and $A_0=-P_0\Delta$ denote the Helmholtz projection and the Stokes operator  on $\IR^n$, respectively, herewith omitting the exponents $q$ of integrability and $s$ of the weight for both $P_0$ and $A_0$.
\vspace{1ex}

\begin{prop}\label{whole-est-basis}
    Let $n \geq 3$ and $1 < p \leq q < \infty$. 
    Suppose $-\frac{n}{q} < s_0 \leq s < \frac{n}{p'}$. 
    Then there holds for $u_0\in L^p_{\sigma,s} (\IR^n)$ and any multi-index $\alpha\in\mathbb N_0^n$ \BLACK that %$|\alpha| \geq 0$ that 
\begin{align}\label{whole-pq}
\|\nabla^\alpha e^{-t A_0}u_0\|_{L^q_{s_0}(\IR^n)} 
\leq C t^{-\frac{n}{2}\big(\frac{1}{p}-\frac{1}{q}\big)-\frac{|\alpha|}{2}}(1+t)^{-\frac{s-s_0}{2}}\|u_0\|_{ L^p_{s}(\IR^n)}
    \end{align}
uniformly in $t>0$.    
\end{prop}

%\vspace{1ex}

{\rm 
\begin{rem}
It is well known that when $s=s_0=0$, there holds for any $\alpha\in\IN_0$ that 
\begin{align}\label{whole-pq2}
\|\nabla^\alpha e^{-t A_0}u_0\|_{L^q(\IR^n)} 
\leq C t^{-\frac{n}{2}\big(\frac{1}{p}-\frac{1}{q}\big)-\frac{|\alpha|}{2}}\|u_0\|_{ L^p(\IR^n)}
    \end{align}
uniformly in $t>0$ with $1 \leq  p \leq q \leq \infty$ and $p\neq \infty, q \neq 1$. 
\end{rem}
}

\vspace{1ex}

Since the proofs of Proposition \ref{whole-est-basis} in \cite[Theorem 4]{Kobayashi-Kubo}  and \cite[Lemma 5.1]{Kobayashi-Kubo2} %contain misleading misprints
lack some arguments, we give a rigorous proof for the convenience of the reader. Let $E_t(x)=(4\pi t)^{-n/2} e^{-|x|^2/(4t)}$\ denote the heat kernel on $\IR^n$. We note that on $L^p_\sigma(\IR^n)$ and corresponding weighted spaces of solenoidal vector fields the Stokes semigroup $e^{-tA_0}$ equals the heat semigroup $e^{t\Delta}$, {\em i.e.},  
\begin{equation}\label{E_t}
e ^{-tA_0} u_0(x) = E_t*u_0(x) = \int_{\IR^n} \frac{1}{(4\pi t)^{n/2}} e^{-\frac{|x-y|^2}{4t}} u_0(y)\dy ,\quad u_0\in L^p_\sigma(\IR^n) .
\end{equation}  
Hence it suffices to consider in \eqref{E_t} and \eqref{whole-pq} the semigroup $\{e^{t\Delta}\}_{t\geq 0}$ directly on $L^p_s(\IR^n)$.

\begin{proof}[Proof of Proposition \ref{whole-est-basis}] 
{\em  Step 1:}  Let $p=q$, $0\leq s<\frac{n}{p'}$ and $s_0=0$. \BLACK To prove that $e^{t\Delta}=E_t(\cdot)*$ maps $L^p_s$  into $L^p_0=L^p$ with norm bounded by $c(1+t)^{-s/2}$ we differ between the cases $0<t<1$ and $t\geq 1$. For $0<t<1$ and $u_0\in L^p_s$ there holds
$$ \|e^{t\Delta} u_0\|_{L^p} \leq \|E_t(\cdot)\|_{L^1} \|u_0\|_{L^p} = \|u_0\|_{L^p} \leq \|u_0\|_{L^p_s}. $$
For $t\geq 1$ note that for any $r\in(1,\infty)$ there holds $\|E_t(\cdot)\|_{L^r}^r\leq c t^{\frac{n}{2}(1-r)}$
so that $$\|E_t(\cdot)\|_{L^r}\leq c t^{-\frac{n}{2r'}}.$$
Further recall that $\|\langle y\rangle^{-s}\|_{L^{n/s,\infty}} \leq C_s$. 
Choosing $r$ such that $1 +\frac1p = \frac1r + \frac{s}{n} +\frac1p$, we get by H\"older's and the  Hardy-Littlewood-Sobolev inequality in Lorentz spaces that 
\begin{align*}
\|e^{t\Delta} u_0\|_{L^p} & = \|E_t(\cdot) * u_0\|_{L^{p,p}} \leq c\|E_t(\cdot)\|_{L^{r,\infty}} \|\langle y\rangle^{-s} (\langle y\rangle^{s}\,u_0)\|_{L^{\frac{pn/s}{p+n/s},p}} \\
& \leq c\|E_t(\cdot)\|_{L^{r}} \|\langle y\rangle^{-s}\|_{L^{n/s,\infty}} \|\langle y\rangle^{s}u_0\|_{L^{p,p}} \leq Ct^{-\frac{n}{2r'}} \|u_0\|_{L^p_s}. 
\end{align*}  
Since $\frac1{r'}=\frac{s}{n}$, the norm estimate of $e^{t\Delta}: L^p_s \to L^p_0$ is proven. 
 
{\em  Step 2:}  Let us show that for any $0\leq s_0\leq s<\frac{n}{p'}$ the map $e^{t\Delta}: L^p_s\to L^p_{s_0} $ is bounded with norm estimate
$$ \|e^{t\Delta}\| \leq C_{s,s_0}(1+t)^{-(s-s_0)/2}, \quad t>0 .$$
We start with the case when $s_0=s$, where the map $e^{t\Delta}: L^p_s\to L^p_s $ is uniformly bounded in $t>0$. Indeed, since $E_t(x) = t^{-n/2} E_1(x/t)$ is a smooth, radial and decreasing function such that $\int_{\IR^n} E_1(x)\dx =1$, by \cite[Theorem 2.1.10]{GrafakosI} 
$$ \sup_{t>0}\, (E_t*|u_0|)(x) \leq Mu_0(x), $$
{\em cf.} \eqref{rho-Mu}. Owing to \eqref{Muu}
we conclude that
$$ \|(E_t*u_0)(x)\|_{L^p_s} \leq \|Mu_0\|_{L^p_s} \leq C_{M,w}\|u_0\|_{L^p_s}. $$ 
To proceed to general $s_0\in(0,s)$ we use 
complex interpolation of weighted spaces, see Triebel \cite[Theorem 1.18.5]{Tribel}, namely  
$$ [L^p_0,\, L^p_s]_\theta = L^p_{\theta s}, \quad 0<\theta<1,$$
and choose $\theta=\frac{s_0}{s}$. Since by the above arguments $T=T(t) = e^{t\Delta}:L^p_s\to L^p_0$ with norm bounded by $C(1+t)^{-s/2}$ and, moreover, $T: L^p_s\to L^p_s$ with norm $C_{M,w}$, we get that 
$$ T:L^p_s \to L^p_{\theta s}, \quad \|T(t)\|\leq C_{M,w}^\theta C^{1-\theta}  (1+t)^{-(1-\theta) s/2} = C_{s,s_0} (1+t)^{-(s-s_0)/2}. $$

{\em  Step 3:}  Let $1<p\leq q<\infty$ and $0 \leq s_0<s <\frac{n}{p'}$. Then  
there exists a constant $C>0$ independent of $p,q$ and $t>0$ such that
\begin{equation}\label{LqLp-oneweight}
\|e^{t\Delta}u_0\|_{L^q_{s_0} } \leq C (1+t)^{-\frac{s-s_0}{2}}\, t^{-\frac{n}2\big(\frac{1}{p}-\frac{1}{q}\big)} \|u_0\|_{L^p_{s}}.
\end{equation}
For the proof we decompose $T= T_1\circ T_2 := e^{t\Delta/2}e^{t\Delta/2}$ where
$$ T_1: L^p_{s_0} \to L^q_{s_0}, \quad T_2: L^p_s \to L^p_{s_0}.$$
Owing to Step 2 applied to $T_2$ it suffices to consider $T_1$.
To prove that $T_1(t)$ has a norm bounded by $ C t^{-\frac{n}{2}\big(\frac1p-\frac1q\big)} $ for any $t>0$ we consider the heat kernel in  \eqref{E_t} which satisfies for any $\lambda\in (0,n)$ the estimate
$$ E_t(x) \leq \frac{C t^{-\lambda/2}} {|x|^{n-\lambda}} \; \frac{|x|^{n-\lambda}}{t^{(n-\lambda)/2}} \,  e^{-\frac{|x|^2}{4t}} \leq \frac{C}{|x|^{n-\lambda}} t^{-\frac{\lambda}{2}}. $$
Hence, as for mapping properties of $T_1(t)$, it suffices to prove that the fractional integral operator  
$$ I_\lambda f(x) = \int_{\IR^n} f(x-y) \,\frac{1}{|y|^{n-\lambda}} \dy $$ 
with $\lambda = n\big(\frac1p-\frac1q\big)$ is a bounded transformation from $L^p_{s_0}$ to $L^q_{s_0}$. 
Actually, this holds by \cite[Theorem 4]{MuWh} provided that the weight $v(x)=\langle x\rangle^{s_0}$ satisfies the generalized Muckenhoupt condition 
$$ \Bigg(\frac{1}{|Q|} \int_Q v^q\dx\Bigg)^{1/q}  \Bigg(\frac{1}{|Q|} \int_Q v^{-p'}\dx\Bigg)^{1/p'} \leq K $$
where $Q$ runs through the set of all non-empty cubes $Q\subset \IR^n$, $|Q|$ denotes the Lebesgue measure of $Q$, and $K>0$ is independent of $Q$. 
This condition is equivalent to the Muckenhoupt condition $w\in \mathscr{A}_r$ for $r=1+\frac{q}{p'}$ and $w=v^q$. In our case, $w(x)=\langle x\rangle^{s_0q}$, and $s_0q<\frac{nq}{p'}=n(r-1)$, {\em i.e.,} $w \in \mathscr A_r$.

Now \cite[Theorem 4]{MuWh} implies that
$$ \|e^{t\Delta} u_0\|_{L^q_{s_0}} \leq C t^{-\frac{\lambda}{2}} \|u_0\|_{L^p_{s_0}} = C t^{-\frac{n}{2}\big(\frac1p-\frac1q\big)} \|u_0\|_{L^p_{s_0}},\quad t>0. $$

{\em  Step 4:}  Let $0\leq s_0<s<\frac{n}{p'}$ and $1<p\leq q<\infty$ such that  
$T=e^{t\Delta}: L^p_{s}\to L^q_{s_0}$ by Step 3 satisfies \eqref{LqLp-oneweight}. 
%, \quad \|T\|\leq C t^{-\frac{n}{2}(\frac1p-\frac1q)} (1+t)^{-(s -s_0)/2}.$$
Since the adjoint $A_0^*$ on $L^{p'}_{\sigma,-s}$ formally equals $A_0$ on $L^{p}_{\sigma,s}$, also the adjoint of $e^{-tA_0}$ formally equals $e^{-tA_0}$, and we may consider $T^*=e^{t\Delta}$ having the same norm as $T$. \BLACK Consequently, $ T^*: L^{q'}_{-s_0} \to L^{p'}_{-s}$ has the norm bound 
$$ \|T^*\|\leq Ct^{-\frac{n}{2}(\frac1p-\frac1q)} (1+t)^{-(s-s_0)/2} = Ct^{-\frac{n}{2}(\frac1{q'}-\frac1{p'})} (1+t)^{-(-s_0-(-s))/2} $$
where $-s<-s_0$ and $q'\leq p'$. 
Thus $T$ satisfies the result also when $s_0<s\leq 0$.

{\em  Step 5:}  In the mixed case $s_0<0<s$ we use the composition $T= T_1 \circ T_2 := e^{t\Delta/2} e^{t\Delta/2}$ where
$$ T_1: L^p_{0}\to L^q_{s_0}, \quad T_2: L^p_s \to L^p_0.$$
Estimating norms of $T_1,T_2$ the desired result follows.

{\em  Step 6:} For any multi-index $\alpha\neq 0$ there holds the pointwise estimate $|\nabla^\alpha E_t(x)| \leq c_\alpha t^{-\frac{|\alpha|}{2}} E_t(\frac{x}{2})$ with a constant $c_\alpha$ independent of $t,x$. 

Now the proof is complete.
\end{proof}

\vspace{1ex}

We first consider the inner domain $\Omega_{R+2}$. 
\vspace{1ex}

\begin{prop}{(Local decay estimate)}\label{interior-domain2}
Let $1<  p <\infty$ and $-\frac{n}{p}< s_0 \leq s < \frac{n}{p'}$
with $s \geq 0$. %\BLUE(We do not need $s_0 \leq s$ due to the cut off property.)
\BLACK 
Set $u=e^{-t A}u_0$ for $u_0 \in  L^p_{\sigma,s}(\Omega)$, {\rm i.e.,} $u$, together with an associated pressure function $\pi$, is  a solution to the Stokes problem
\begin{align}\label{the-Stokes-eq-0}
\del_t u-\Delta u +\nabla \pi=0, \  \div u =0\ { on } \ \Omega, \ \
u=0 \mbox{ on } \del \Omega,  \mbox{ and }\ u(0)=u_{0}.  
\end{align} 
Then there exists a constant $C=C(\Omega,p,s,s_0,R)>0$ such that 
\begin{equation}\label{w-local-decay-R}
\|u(t)\|_{W^{1,p}_{s_0}(\Omega_{R+2})} \leq C t^{-\frac{n}{2p}-\frac{s}{2}} \|u_0\|_{L^p_s(\Omega)}, \quad t>1.
\end{equation}

\end{prop}

\vspace{1ex}
Note that the weight $\langle x\rangle^{s_0}$ plays no role in \eqref{w-local-decay-R}. The local $L^p$ estimate of $u$ in \eqref{w-local-decay-R} is an immediate consequence of \eqref{local-decay} since $-\frac{n}{2} < -\frac{n}{2p}-\frac{s}{2}$ and $t>1$ 
directly imply that 
\begin{align*}
  \| u(t)\|_{L^{p}_{s_0}(\Omega_{R+2})} \leq C \| u(t)\|_{L^{p}(\Omega_{R+2})} 
\leq C t^{-\frac{n}{2p}-\frac{s}{2}} \|u_0\|_{ L^p_s(\Omega)\BLACK}, \quad t>1. 
\end{align*}
However, to estimate $\nabla u$ we prepare an auxiliary local result in Lemma \ref{interior-domain1} below.   

%need the another proof of Proposition \ref{interior-domain2} in pp. 12. The proof uses the support of $\mathcal{U}_0$ and $F$. )\BLACK 

\vspace{1ex}

Let $\varphi_R\in C^\infty_0 (\Omega)$  be a smooth cut-off function satisfying $0 \leq \varphi_R \leq 1$ and 
$$\varphi_R(x)=\left\{
\begin{array}{ll}
 1  &  \mbox{ for } \  |x| \leq R+2\\
 0  &  \mbox{ for }  \ |x| \geq R+3. 
\end{array}  \right. 
$$ 
Hence ${\rm supp}\, \nabla\varphi_R\subset \overline{D_{R+2}}=\{x: R+2 \leq |x| \leq R+3\}$ and $(\nabla \varphi_R)\cdot u_0 \in L^q_m(D_{R+2})  \; %((not \in  {W}^{1, q}_m(D_{R+2})))\BLACK
$ 
for $u_0\in L^q_\sigma(\Omega)$. 
Recall the Bogovski\u{\i} operator $\mathbb{B}$ defined by \eqref{B} for $D_{R+2}$ and to be applied to $\nabla \varphi_R\cdot u_0$.
Similarly, we define $\psi_R \in C^\infty_0 (\Omega)$ by 
$$\psi_R(x)=\left\{
\begin{array}{ll}
 1  &  \mbox{ for } \  |x| \leq R+1\\
 0  &  \mbox{ for }  \ |x| \geq R+2. 
\end{array}  \right. 
$$ 
%
%$$\varphi_R \in C^\infty_0 (\Omega),  \ \ 0 \leq  \varphi_R \leq 1  $$
%and  
%Recall the Bogovski\u{\i} operator $\mathbb{B}$ defined by \eqref{B} We note that for $u_0 \in \mathcal D(A)$, we can take the cut off function $\varphi_R$ satisfying that $(\nabla \varphi_R)\cdot u_0 \in \dot{W}^{1, q}\BLUE (D_{R+1})$\BLACK. 
%Let $A_0$ denote the Stokes operator on the whole space generating the Stokes semigroup $\{e^{-t A_0}\}_{t\geq 0}$ on $L^q_{\sigma} (\mathbb{R}^n)$. %Further let $\tilde{v}=e^{-t A_0} v_0$ for $v_0 \in \mathcal D(A_0) \subset L^q_{\sigma} (\mathbb{R}^n)$ satisfying the heat equation 
%\begin{align}\label{heat-eq-1}
%\del_t \tilde{v} - \Delta \tilde{v}=0  \ \ \textrm{ on} \ \ \mathbb{R}^n,  \ \ \tilde{v}(0)=v_0. 
%\end{align}

Let $u_0 \in %\mathcal D( A_{q,s}\BLACK) = W^{2,q}_s(\Omega)\cap W^{1,q}_{0,s}(\Omega)\cap L^q_{\sigma,s}(\Omega)\subset 
L^p_{\sigma,s}(\Omega)$ and 
\begin{align}\label{v01}
v_{0} = (1-\varphi_R) u_0 + \mathbb{B}[(\nabla \varphi_R)\cdot u_0]\ \mbox{ on } \IR^n, 
\end{align}
{\em i.e.,} $v_0$ is a solenoidal extension of $u_0$ to the whole space. 
%Here for simplicity we use $A$ for the Stokes operator while $\mathcal D(A_{q,s})$ is used to describe the domain in weighted spaces. \BLACK 
We note that 
\begin{align}\label{est-v0}
\|v_{0}\|_{L^p_{s_1}(\mathbb{R}^n)} \leq C\|u_0\|_{L^p_s(\Omega)}
\end{align}
for $-\frac{n}{p}< s_0 \leq s < \frac{n}{p'}$ by \eqref{B}.  
In addition, $v_{0} \in  L^p_{\sigma,s}(\mathbb{R}^n)$.
Then we apply the Stokes semigroup $e^{-tA_0}$ to $v_{0}$ and see that $\tilde{v}(t) = e^{-tA_0}v_{0}$ even solves the heat equation 
$$\del_t \tilde{v}-\Delta \tilde{v} = 0\ \mbox{ on }\ \IR^n,\ \ \tilde{v}(0)={v}_{0}.$$ 
Finally we set
\begin{align}\label{whole-est-base}
v(t) := (1-\psi_R) \tilde{v}(t) + \mathbb{B}'\BLACK [(\nabla \psi_R)\cdot \tilde{v}(t)]\ \mbox{ on } \Omega, \quad t>0, 
\end{align}
where $\mathbb{B}'$ is the Bogovski\u{\i} operator defined on $D_{R+1}$ in order to get that  $\textrm{div}\,v(\cdot,t)=0$ and $\textrm{div}\,v(\cdot,0)=0$. 
%RED ((Here we need the Bogovskii operator $\mathbb{B}'$ defined on the set $D_{R+1}$ in order to get that  $\textrm{div}\,v(\cdot,t)=0$ and $\textrm{div}\,v(\cdot,0)=0$. On the other hand, $\RED \mathbb{B}\BLACK [(\nabla \psi_R)\cdot \tilde{v}(t)] \equiv 0$ and could be omitted, but $v$ will not be solenoidal. Also the last two terms in (3.16) could be omitted.)) \BLACK
%
%We note that 
%\begin{align}\label{heat-eq-3}
%\del_t \tilde{v}-\Delta \tilde{v}%\RED \nabla \pi \= 0 \ \mbox{on}  \  \mathbb{R}^n\ \  \mbox{ and } \ \tilde{v}(0)=\tilde{v}_{0}, \BLACK
%\end{align} 
%
Thus $v$ is a solution to the inhomogeneous heat equation   
\begin{align}\label{heat-eq-2}
\del_t v-\Delta v
= F(t) \ \mbox{on} \ \Omega, \ \ 
v=0 \mbox{ on } \del\Omega \ \  \mbox{ and } \ v(0)= v_{0}\big|_{\Omega},  
\end{align}
in the exterior domain $\Omega$; here, 
with $(\nabla \psi_R)\cdot \nabla \tilde{v} = \big(\partial_j\psi_R \,\partial_i\tilde v_j\big)_i $,  
\begin{align}\label{def-F}
F(t) = 2(\nabla \psi_R)\cdot \nabla \tilde{v} + (\Delta \psi_R)\tilde{v} +  \mathbb{B}' [(\nabla \psi_R)\cdot\Delta\tilde{v}] - \Delta
\mathbb{B}' [(\nabla \psi_R)\cdot\tilde{v}]. 
\end{align}
To see that $v$ in \eqref{whole-est-base} has the initial value $v(0) =  v_0\big|_{\Omega}$ we use \eqref{whole-est-base} to get
\begin{align*}
v(0) & = (1-\psi_R) v_0 + \mathbb{B}'[(\nabla \psi_R)\cdot v_0]\\
& =: V_1 +V_2. \BLACK 
\end{align*}
Note that $V_1 = v_0$ since $\supp v_0\subset \overline\Omega\setminus B_{R+2}$ and $\psi_R=0$ for $|x|\geq R+2$. On the other hand, $V_2=0$ since 
$\supp \nabla \psi_R \subset \overline{D_{R+1}}$ and   $v_0=0$ on $\overline{D_{R+1}}$. \BLACK 

%\BLACK $I_1 = (1-\varphi_R) u_0+\mathbb{B}[(\nabla \varphi_R)\cdot u_0] = v_0$ \BLACK because 
%$(1-\psi_R)(1-\varphi_R)=1-\varphi_R$ and $\supp \mathbb{B}[(\nabla \varphi_R)\cdot u_0] \subset \overline{D_{R+2}}$ which 
%implies that $(1-\psi_R)\mathbb{B}[(\nabla \varphi_R)\cdot u_0] = \mathbb{B}[(\nabla \varphi_R)\cdot u_0]$. 
%On the other hand, $I_2=0$ since 
%$\supp \nabla \psi_R \subset \overline{D_{R+1}}$ and  \RED $v_0=0$ on $\overline{D_{R+1}}$. \BLACK
%and $ v_0=0$ in $B_{R+2}$.  
%\RED and 
%$v_{02}= \BLUE (1-\varphi_R) u_0 +\mathbb{B}[(\nabla \varphi_R)\cdot u_0]=v_{01}$ \RED (Why? I do see by (3.11) that $v_{02}= (1-\varphi_R) u_0 +\mathbb{B}[(\nabla \varphi_R)\cdot u_0] =v_{01}$ provided that $\tilde{v}=e^{-tA_0}u_0$). 

Obviously, \eqref{heat-eq-2} implies the reduction of the homogeneous Stokes solution $\tilde{v}$ on $\mathbb{R}^n$ by the Bogovski\u{\i} operator to the heat equation for $v$ with source term $F(t)$ on $\Omega$.
\\[1ex] 

For $F$ defined by \eqref{def-F} and the solution $v$ we prepare estimates as follows.  

\vspace{1ex}

\begin{lem}\label{interior-domain1}
Let  $1< p <\infty$ and $-\frac{n}{p}< s_0 \leq s < \frac{n}{p'}$ with $s \geq 0$. %\BLUE(We do not need $s_0 \leq s$ due to the cut off property.)
\BLACK Then $F$ defined by \eqref{def-F} satisfies the estimates 
\begin{align}
\|F(t)\|_{L^p_{s_0}(\Omega)} 
 &  \leq C t^{-\frac{n}{2p}-\frac{s}{2}} \|u_0\|_{L^p_s(\Omega)},\,\, \mbox{ uniformly in }  1 < t,\label{est-F1}\\
\|F(t)\|_{L^p_{s_0}(\Omega)} & \leq C t^{-\frac{1}{2}} \|u_0\|_{L^p_s(\Omega)},  \mbox{ uniformly in }  0< t \leq 1.  \label{est-F2}
\end{align} 
In addition, 
$v$ defined in \eqref{whole-est-base}-\eqref{def-F} with initial value $v_0\big|_\Omega$, see \eqref{v01}, also satisfies that 
\begin{align}\label{est-v-u_0-1}
 \|v(t)\|_{W^{1,p}_{s_0} (\Omega_{R+2})} & \leq C t^{-\frac{n}{2p}-\frac{s}{2}} \|u_0\|_{L^p_s(\Omega)},\, \mbox{ uniformly in }  1 < t,\\
\|v(t)\|_{ W^{1,p}_{s_0} (\Omega_{R+2})} & \leq C t^{-\frac{1}{2}} \|u_0\|_{L^p_s(\Omega)},  \mbox{ uniformly in }  0< t \leq 1.  \label{est-v-u_0-2}
\end{align}
\end{lem}

\vspace{1pt}

\begin{proof} 
We start with the estimate of $F(t)$ %in \eqref{def-F} 
when $t>1$. Note that $\supp F(t) \subset \overline{D_{R+1}}$ due to \eqref{B} and the definition of $\psi_R$. \BLACK For terms involving $\mathbb B'$, note that 
\begin{align}\label{commutator}
(\nabla \psi_R)\cdot \Delta \tilde{v}
= \div\{ \nabla \psi_R \cdot  \nabla \tilde{v}\}- 
\nabla^2 \psi_R:\nabla \tilde v
\end{align}
and apply \eqref{B} and \eqref{B-2}. For the second term with $\Delta\mathbb B'$ we exploit \eqref{B}  with $k=1$. Summarizing we get that 
\begin{align}\label{F-est1}
\|F(t)\|_{L^p_{s_0}(\Omega)} \leq 
C\|\tilde{v}\|_{W^{1,p}_{s_0}(\Omega_{R+2})} \leq 
C\|\tilde{v}\|_{W^{1,\infty}(\mathbb{R}^n)}. 
\end{align}
%(\BLUE I think that rather than $s_0 \leq 0$ we use that $0 \leq s_0$ due to the bounded domain. 
%{\rm i.e., } we use 
%$\|\tilde{v}\|_{W^{1,q}_{s_0}(\Omega_{R+2})} \leq C\|\tilde{v}\|_{W^{1,q}(\Omega_{R+2})} $. 
%)\RED (In the first step of (3.22) $s_0$ does not matter since $F$ has compact support. Thus $\tilde v$ is needed only on a compact subdomain. Then also in the second step in (3.2) $s_0$ does not matter) \BLACK
Next \eqref{whole-pq2} with $q=\infty$ and Proposition \ref{whole-est-basis} for $p=q$ and $s_0=0$ applied to $\tilde v(t)$ yield 
\begin{align}\label{est-tildev}
\|\tilde v(t)\|_{L^\infty(\mathbb{R}^n)} = \|e^{-tA_0}v_0\|_{L^\infty(\mathbb{R}^n)} \leq C t^{-\frac{n}{2p}} \|e^{-t A_0/2}v_0\|_{L^p(\mathbb{R}^n)} \leq 
C t^{-\frac{n}{2p}-\frac{s}{2}} \|u_0\|_{L^p_s(\Omega)}.
\end{align}
Similarly, we treat $\nabla\tilde v(t)$ getting the same upper bound for $t>1$. Hence we obtain that  
\begin{align*}%\label{est-F1}
    \|F(t)\|_{L^p_{s_0}(\Omega)}\leq 
    Ct^{-\frac{n}{2p}-\frac{s}{2}} \|u_0\|_{L^p_s(\Omega)},\quad t>1.
\end{align*} 

By analogy, we derive from the definition of $v$, see \eqref{whole-est-base}, 
%\begin{align} v(t) = (1-\psi_R) \tilde{v}(t) + \mathbb{B}[(\nabla \psi_R)\cdot \tilde{v}(t)], \quad t>0, \end{align}
%a similar argument to that in the estimate of $F(t)$ {i.e., } 
%\eqref{F-est1} and \eqref{est-tildev} 
the estimate \eqref{est-v-u_0-1}. 

Next consider $0<t\leq 1$. By \eqref{F-est1} which also holds for $0<t\leq 1$ and \eqref{whole-pq}
\begin{align*}
\|F(t)\|_{L^p_{s_0}(\Omega)} &
\leq C\| \tilde{v}\|_{W^{1,p}_{s_0}(\mathbb{R}^n)} \leq Ct^{-\frac{1}{2}}\|v_0\|_{L^p_s(\IR^n)}\leq  Ct^{-\frac{1}{2}} \|u_0\|_{L^p_s(\Omega)}. 
\end{align*}  
%\begin{align*}
%\|F(t)\|_{L^q_{s_0}(\Omega)} &
%\leq C\|\nabla \tilde{v}\|_{L^q(\mathbb{R}^n)} + \|\tilde{v}\|_{L^q(D_{R+1})}\nonumber\\
%& \leq C\| \tilde{v}\|_{W^{1,q)(\mathbb{R}^n)}  
%\leq  Ct^{-\frac{1}{2}} \|v_0\|_{L^q_s(\Omega)}\\
%& \leq  Ct^{-\frac{1}{2}} \|u_0\|_{L^q_s(\Omega)}.\nonumber
%\end{align*}
%
%Indeed, \eqref{est-F2} is directly obtained by \eqref{whole-pq} for $p=q$ and $|\alpha|=1$ with the Sobolev embedding, \RED (Which Sobolev embedding? $q\in (1,\infty)$ is arbitrary, we do not have Poincar\'e, $\tilde v$ is defined on $\IR^n$) (\BLUE $\rightarrow$ I used the Gagliardo–Nirenberg–Sobolev inequality, however, I noticed that this needs that $q<n$ and thus I modified blow.)\BLACK {\em i.e.,} we estimate as 
%
By \eqref{whole-est-base} we  similarly get estimates  of $\|v(t)\|_{L^p_{s_0} (\Omega_{R+2})}$ and $\|\nabla v(t)\|_{L^p_{s_0} (\Omega_{R+2})}$  in \eqref{est-v-u_0-2}. 
%\BLUE $\rightarrow$ 
%By \eqref{heat-eq-2} $v$ is represented as the solution to the integral equation
%\begin{align}\label{form-v}
%v= e^{-t A_0}v_{01}+ \displaystyle\int_{0}^t e^{-(t-\tau) A_0}F(\tau) \dtau, 
%\end{align}
%which implies that 
%\RED (Since you omit $s_0$ we need a constant $C$) \BLACK
%$$
%\|v\|_{L^q_{s_0}(\Omega_{0,R+2})}\leq  C\|e^{-t A_0}v_{01}\|_{L^q(\Omega_{0,R+2})} + C\int_{0}^t \|e^{-(t-\tau) A_0}F(\tau)\|_{L^q(\Omega_{0,R+2})} \dtau. 
%$$
%This together with Propositions \ref{interior-est-base} and \ref{whole-est-basis},  \eqref{est-v0}, \eqref{est-F1} and \eqref{est-F2} directly shows Lemma \ref{interior-domain1}. 
%The estimate of $\|\nabla v\|_{L^q_{s_0}(\Omega_{0,R+2})}$ is similar.  
\end{proof}

%\vspace{1ex} 
%{\rm
%\begin{rem}
%We do not need that $v$ can be taken in $W^{2,q}_{s_0} %(\Omega_{R+2})$ due to \eqref{commutator} and the %property of the Bogovskii operator \eqref{B-2}.  
%On the other hand, 
%$u$ is needed only in $ W^{1,q}_{s_0} (\Omega_{R+2})$. 
%\end{rem}
%}
\vspace{1ex}

%\begin{lem}\label{interior-domain1-2}
%\RED To be omitted: \BLACK Let  $1<  q <\infty$ and $-\frac{n}{q}< s_0 \leq s < n\big(1-\frac{1}{q}\big)$. Then it holds that 
%$$
%\|v\|_{ W^{1,q}_{s_0} (\Omega_{0,R+2})} \leq C t^{-\frac{1}{2}} \|u_0\|_{L^q_s(\Omega)}, \quad t> 0,
%$$
%uniformly in $t$. 
%\end{lem}

%\vspace{1ex}
%\RED To be omitted: \BLACK 
%Lemma \ref{interior-domain1-2} is analogous to Lemma \ref{interior-domain1} with using Proposition \ref{whole-est-basis} and the Sobolev embedding. 
%and \eqref{est-F2}. 
%Here the condition $n\big(\frac{1}{p}-\frac{1}{q}\big)<1$ is used to get that 
%$$
%\displaystyle\int_1^{(t-1)/2} (t-\tau)^{-\frac{n}{2}\big(\frac{1}{p}-\frac{1}{q}\big)-\frac{1}{2}}\tau^{-\frac{n}{2p}-\frac{s}{2}} \dtau \leq C t^{-\frac{n}{2}\big(\frac{1}{p}-\frac{1}{q}\big)} \leq Ct^{-\frac{1}{2}}.  
%$$
\vspace{1ex}

We can derive the local decay estimate of the solution $u=e^{-t A}u_0 $ in $W^{1,p}(\Omega)$ also with the help of the solution $v$ to \eqref{heat-eq-2} on $\Omega$ as follows.
Indeed, for $\nabla u$, we need the following proof because \eqref{local-decay-k} is applicable to $u_0$ only if $\supp u \subset B_R$ in contrast to \eqref{local-decay} which holds for all $u\in L^p_\sigma(\Omega)$.
% 
%\BLUE The estimate is already obtained by Proposition \ref{interior-domain2}, however, for completeness of the proof we include another proof below. In fact the proof is helpful to derive Lemma \ref{est-mathcalU-tsmall}, and we will use the "approximation" $\mathcal{U}=u-v$ below in the proof of Proposition \ref{exterior-est2}. 
%(Maybe it is better to omit the another proof and prove Lemma \ref{est-mathcalU-tsmall} in details to emphasize our originality. )\BLACK 
%\vspace{1ex}

%\begin{prop}\label{local-energy-decay-u}{(Local decay estimate)}
%Let $t>1$, $1<  q <\infty$ and $-\frac{n}{q}< s_0 \leq s < n\big(1-\frac{1}{q}\big)$. 
%Assume that $u=e^{-t A_0}u_0$ for $u_0 \in D(A_{0,s})$, {\rm i.e., } $u$ is  a solution to the non-stationary Stokes problem
%\begin{align}\label{the-Stokes-eq}
%\del_t u-\Delta u +\nabla \pi=0 \  { on } \   \Omega_0,  \ \ 
%u=0 \mbox{ on } \del \Omega_0,  \ \  \div v =0\ \mbox{ and }\ u(0)=u_{0}.  
%\end{align} 
%Then it holds that 
%$$
%\|u\|_{W^{1, q}_{s_0}(\Omega_{0,R+2})} \leq C t^{-\frac{n}{2q}-\frac{s}{2}} \|u_0\|_{L^q_s(\Omega)}
%$$
%uniformy in $t$. 
%\end{prop}

\vspace{1ex}
 \begin{proof}[Proof of Proposition \ref{interior-domain2}]
Let $v$ be the solution to \eqref{heat-eq-2} and set $\mathcal{U}=u-v$. 
Then $\mathcal{U}$ solves the equation
\begin{align}\label{the-Stokes-eq-2}
\del_t \mathcal{U}-\Delta \mathcal{U} +\nabla\pi = -F, \ \div \mathcal{U} =0  \ \ \mbox{on} \  \Omega, \ \ \
\mathcal{U}\big|_{\del \Omega}=0, \ \  \  \mathcal{U}(0)=\mathcal{U}_{0}:=u_0-v_0, 
\end{align} 
where 
$\pi$ is defined by \eqref{the-Stokes-eq-0}, 
$F$ is defined by (3.16) and solenoidal by (3.14), (3.15),
and 
\begin{equation}\label{U_0}
\mathcal{U}_{0}= \varphi_R u_0- \mathbb{B}[(\nabla \varphi_R)\cdot u_0] \in L^p_{\sigma,s}(\Omega). 
\end{equation}
%(for better readablility and later use) 
%\RED ((Why:  since $\mathbb{B}' [(\nabla \psi_R)\cdot v_0]=0$ ??)\BLUE ($\rightarrow$ We compute $u(0)-v(0)=u_0-v(0)$. Since $V_1=v_0$ and $V_2=\mathbb{B}' [(\nabla \psi_R)\cdot v_0]=0$ as you wrote, we see that $v(0)=v_0$ on $\Omega$. This together with \eqref{v01} imply \eqref{U_0}. )\BLACK.  
Hence we obtain the integral equation
\begin{align}\label{form-mathcalU}
\begin{aligned}
\mathcal{U}(t) & = e^{-t A} \mathcal{U}_{0} - \displaystyle\int_0^t e^{-(t-\tau) A}F(\tau) \dtau\\
& = e^{-t A} \mathcal{U}_{0} - \int_0^{t/2} e^{-(t-\tau) A}F(\tau) \dtau - \int_{t/2}^{t} e^{-(t-\tau) A}F(\tau) \dtau\\
& = I_1+I_2 +I_3. 
\end{aligned}
\end{align}
By Lemma \ref{interior-domain1}, it is sufficient to show that 
\begin{equation}\label{UtW1q}
\|\mathcal{U}(t)\|_{W^{1,p}_{s_0}(\Omega_{R+2})} \leq C t^{-\frac{n}{2p}-\frac{s}{2}} \|u_0\|_{L^p_s(\Omega)}\quad \textrm{uniformly in }t>1.
\end{equation}
%together with an analogous estimate of $\nabla\mathcal{U}(t)$ in $L^{q}_{s_0}(\Omega_{R+2})$ uniformly in $t>1$. 
%Since we use the time integral form \eqref{form-mathcalU}, if we consider $\nabla^2 u$, we cannot control the singurality $(t-\tau)^{-1}$ in the time integral. 

Concerning $I_1$, \eqref{local-decay} in Proposition \ref{interior-est-base} and \eqref{U_0} \BLACK directly show that 
$$
\|e^{-t A}\mathcal{U}_{0}\|_{L^{p}(\Omega_{R+2})} \leq 
C t^{-\frac{n}{2p}-\frac{s}{2}} \|\mathcal{U}_0\|_{L^p(\Omega)} \leq   
C t^{-\frac{n}{2p}-\frac{s}{2}} \|u_0\|_{L^p(\Omega_{R+3}\BLACK)} \leq
Ct^{-\frac{n}{2p}-\frac{s}{2}} \|u_0\|_{L^p_s(\Omega)},\; t>1,
$$
%\RED (this holds for all admisssible $s$)
%\BLUE I think that by the same argument as that in the last iequality in pp.12, we need that $0 \leq s$. Could you teach me details for all $s$?) \RED (The first step by (3.4) used no weights, but $-\frac{n}{2} < -\frac{n}{2q}-\frac{s}{2}$ for all admissible $s<n/q'$.  The second step uses that $\mathcal{U}_0$ has compact support. In the third step, with $u_0$ used on $\Omega_{R+2}$ only, we can estimate by $u_0$ in $L^q_s(\Omega)$ for all $s$)
since by assumption 
$-\frac{n}{2} < -\frac{n}{2p}-\frac{s}{2}$.   

As for $I_2$, the Poincar\'e inequality on $\Omega_{R+2}$ implies that 
\begin{align*}%\label{local-decay-1}
\begin{aligned}
\|I_2\|_{L^{p}_{s_0}(\Omega_{R+2})} & \leq 
C \int_0^{t/2} \|\nabla e^{-(t-\tau) A}F(\tau)
\|_{L^{p}(\Omega_{R+2})} \dtau \\
 & \leq 
C \Big(\int_0^{1/2} + \int_{1/2}^{t/2}\Big) \|\nabla e^{-(t-\tau) A}F(\tau)
\|_{L^{p}(\Omega_{R+2})} \dtau\\
%&\quad + C\displaystyle\int_{1/2}^{t/2} \|\nabla e^{-(t-\tau)A}F(\tau)\|_{L^{q} (\Omega_{R+2})} \dtau\\
& =:  \mathcal I_{21} + \mathcal I_{22}.
\end{aligned}
\end{align*}
Here $\mathcal I_{21}$ is estimated by \eqref{local-decay-k} in Proposition \ref{interior-est-base} and \eqref{est-F2}, exploiting that $\supp F(t) \subset \overline{D_{R+1}}$, so that 
\begin{align*}
%\begin{aligned}
\mathcal I_{21}
& \leq C \displaystyle\int_0^{1/2} (t-\tau)^{-\frac{n}{2}} \tau^{-\frac{1}{2}} \dtau \,
\|u_0\|_{L^p(\Omega)} \\
& \leq C t^{-\frac{n}{2}}
\|u_0\|_{L^p_s(\Omega)}. 
%\end{aligned}
\end{align*}
Concerning $\mathcal I_{22}$, independent of the sign of $-\frac{n}{2p}-\frac{s}{2}+1$, we get by \eqref{local-decay-k} and \eqref{est-F1} that
 \begin{align*}
%\begin{aligned}
\mathcal I_{22}
& \leq C \displaystyle\int_{1/2}^{t/2} (t-\tau)^{-\frac{n}{2}} \tau^{-\frac{n}{2p}-\frac{s}{2}} \dtau \, \|u_0\|_{L^p_s(\Omega)} \\
& \leq C t^{-\frac{n}{2}} \int_{1/2}^{t/2} \tau^{-\frac{n}{2p}-\frac{s}{2}} \dtau\,\|u_0\|_{L^p_s(\Omega)}\\
& \leq C t^{-\frac{n}{2}} \Big( t^{-\frac{n}{2p}-\frac{s}{2}+1} +1 \Big)\,\|u_0\|_{L^p_s(\Omega)}\\
& \leq C \Big(t^{-\frac{n}{2p}-\frac{s}{2}+1-\frac{n}{2}} + t^{-\frac{n}{2}}\Big)\,\|u_0\|_{L^p_s(\Omega)}\\
& \leq Ct^{-\frac{n}{2p}-\frac{s}{2}} \|u_0\|_{L^p_s(\Omega)}.
%\end{aligned}
\end{align*}
In the special case that $-\frac{n}{2p}-\frac{s}{2}=-1$, we arrive at the bound $Ct^{-\frac{n}{2}}(1+\log t)$ which again decays faster than $ t^{-\frac{n}{2p}-\frac{s}{2}} $, since $n\geq 3$ and $-\frac{n}{2}< -\frac{n}{2p}-\frac{s}{2}$.

%$$
%(t-\tau)^{-\frac{n}{2}}=(t-\tau)^{-\frac{n}{2}\big(1-\frac{1}{q}\big)}(t-\tau)^{-\frac{n}{2q}} \leq \tau^{-\frac{n}{2}}(t-\tau)^{-\frac{n}{2q}} 
%$$
%by $\tau \leq t/2$, 
%this together with \eqref{est-F1} imply that 
%\begin{align}
%\begin{aligned}
%I_{22} & \leq C \displaystyle\int_{1/2}^{t/2} (t-\tau)^{-\frac{n}{2}} \tau^{-\frac{n}{2q}-\frac{s}{2}} \dtau \, \|u_0\|_{L^q_s(\Omega)} \\
%& \leq C \displaystyle\int_{1/2}^{t/2} \tau^{-\frac{n}{2}} (t-\tau)^{-\frac{n}{2q}-\frac{s}{2}} \dtau \,\|u_0\|_{L^q_s(\Omega)} \\
%(& \leq C t^{-\frac{n}{2q}-\frac{s}{2}}\displaystyle\int_{1/2}^{\infty}\tau^{-\frac{n}{2}}d\tau 
%\|u_0\|_{L^q_s(\Omega)})\\
%& \leq C t^{-\frac{n}{2q}-\frac{s}{2}}\|u_0\|_{L^q_s(\Omega)}.  
%\end{aligned}
%\end{align}
%
%Since $-\frac{n}{2}< -\frac{n}{2q}-\frac{s}{2}$ by assumption, 
Hence we proved that 
$$
\|I_2\|_{L^{p}_{s_0}(\Omega_{R+2})} \leq C t^{-\frac{n}{2p}-\frac{s}{2}} \|u_0\|_{L^p_s(\Omega)}, \quad t>1. 
$$
Finally, Poincar\'e's inequality implies that  
\begin{align*}
%\begin{aligned}
\|I_3\|_{L^{p}_{s_0}(\Omega_{R+2})} & \leq 
C \Bigg(\int_{t/2}^{t-1} + \int_{t-1}^{t}\Bigg) \|\nabla e^{-(t-\tau) A}F(\tau)
\|_{L^{p}(\Omega_{R+2})} \dtau \\
 %& \leq 
%C\displaystyle\int_{t/2}^{t-1} \|\nabla e^{-(t-\tau) A}F(\tau) \|_{L^{q}(\Omega_{R+2})} \dtau\\
%&\quad + 
%C\int_{t-1}^{t} \|\nabla e^{-(t-\tau) A}F(\tau)\|_{L^{q}(\Omega_{R+2})} \dtau\\
& =:  \mathcal I_{31} + \mathcal I_{32}.
%\end{aligned}
\end{align*} 
The estimate of $\mathcal I_{31}$ is analogous to that of $\mathcal I_{22}$ as 
\begin{align*}%\label{local-decay-2}
%\begin{aligned}
\mathcal I_{31}
& \leq C \int_{t/2}^{t-1}  (t-\tau)^{-\frac{n}{2}} \tau^{-\frac{n}{2p}-\frac{s}{2}} \dtau \, \|u_0\|_{L^p_s(\Omega)} \\
& \leq C  t^{-\frac{n}{2p}-\frac{s}{2}} \displaystyle\int_{t/2}^{t-1}  (t-\tau)^{-\frac{n}{2}} \dtau \, \|u_0\|_{L^p_s(\Omega)} \\
%& \leq C t^{-\frac{n}{2q}-\frac{s}{2}} \Big( t^{-\frac{n}{2}+1} +1 \Big)\,\|u_0\|_{L^q_s(\Omega)}\\
%& \leq C \Big(t^{-\frac{n}{2q}-\frac{s}{2}+1-\frac{n}{2}} + t^{-\frac{n}{2q}-\frac{s}{2}}\Big)\,\|u_0\|_{L^q_s(\Omega)}\\
& \leq Ct^{-\frac{n}{2p}-\frac{s}{2}} \|u_0\|_{L^p_s(\Omega)},
%\end{aligned*}
\end{align*}
since $n\geq 3$ and $t>1$. We see from $\eqref{e-tAt}_3$ \BLACK   
for %$p=q$ and 
$|\alpha|=1$ and \eqref{est-F1}  that 
\begin{align*}%\label{local-decay-2}
%\begin{aligned}
\mathcal I_{32}
& \leq C \displaystyle\int_{t-1}^{t}  (t-\tau)^{-\frac{1}{2}} \tau^{-\frac{n}{2p}-\frac{s}{2}} \dtau \, \|u_0\|_{L^p_s(\Omega)} \\
& \leq Ct^{-\frac{n}{2p}-\frac{s}{2}} \|u_0\|_{L^p_s(\Omega)}.
%\end{aligned*}
\end{align*}

The estimate of $\|\nabla \mathcal{U}\|_{L^p_{s_0}(\Omega_{R+1})}$ is obtained by analogy without using the Poincar\'e inequality. This completes the proof of \eqref{UtW1q} and of   Proposition \ref{interior-domain2}. 
\end{proof}

\vspace{1ex}
Next we state a uniform estimate for $\mathcal U$,   $0<t<2$, \BLACK which will be used in the proof of Lemma \ref{est-G2} below. 

\vspace{1ex}

\begin{lem}\label{est-mathcalU-tsmall}
    Let $1< p <\infty$ and  $-\frac{n}{p} < s_0 \leq s < \frac{n}{p'}$ with $s \geq 0$.%\BLUE(We do not need $s_0 \leq s$ due to the cut off property.)
    \BLACK Then it holds that 
$$
\|\mathcal{U}(t)\|_{W^{1,p}_{s_0}(\Omega_{R+2})} \leq Ct^{-\frac{1}{2}}\|u_0\|_{L^p_s(\Omega)}, \quad  0< t \leq 2.
$$
\end{lem}

\begin{proof}
We note that $\mathcal{U}=u-v$. 
Combining Poincar\'e's inequality and $\eqref{e-tAt}_3$ \BLACK for $u$ with
\eqref{est-v-u_0-2} for $v$ implies that for $0< t \leq 2$ 
$$
\|\mathcal{U}\|_{W^{1,p}_{s_0}(\Omega_{R+2})} \leq
C\big\{\|\nabla u\|_{L^p(\Omega)}+\|v\|_{W^{1,p}_{s_0}(\Omega_{R+2})}\big\}
\leq Ct^{-\frac{1}{2}}\|u_0\|_{L^p_s(\Omega)}. 
$$
The estimate of $\nabla \mathcal{U}$ is similar without using the Poincar\'e inequality. 
\end{proof}
\BLACK

%For $t >1$, we obtained the estimate by \RED \eqref{UtW-1q} since $-\frac{n}{2q}-\frac{s}{2}<-\frac12$ ? (\eqref{est-v-u_0-1}?, better: \eqref{UtW1q}). \BLACK 
%For $0<t \leq 1$, applying Proposition \ref{whole-est-basis} \RED (works on $\IR^n$) \BLACK and \eqref{est-F2} to \eqref{local-decay-1}, 
%and Proposition \ref{whole-est-basis}, \eqref{est-F1} respectively, with the Sobolev embedding \RED (better: Poincar\'e inequality?) \BLACK imply Lemma \ref{est-mathcalU-tsmall} for $\|\mathcal{U}\|_{L^p_{s_0}(\Omega_{0,R+2})}$.  The estimate of $\| \nabla\BLACK\mathcal{U}\|_{L^p_{s_0}(\Omega_{0,R+2})}$ is obtained by analogy without referring to Poincar\'e's inequality. 

%\RED ((Do you simply mean for $0<t\leq 1$ that
%$$\|\mathcal{U}\|_{L^{p}_{s_0}(\Omega_{0,R+2})} \leq C \|u_0\|_{L^{p}_{s}} + \int_0^1 C \tau^{-1/2}\|u_0\|_{L^{p}_{s}}\dtau  \leq C\|u_0\|_{L^{p}_{s}} \leq Ct^{-1/2}\|u_0\|_{L^{p}_{s}}
%$$ 
%using only (3.18) and the boundedness of the semigroup? But concerning $\nabla \mathcal{U}$ we have no estimate for $\nabla e^{-tA}\mathcal{U}_0$ on $\Omega, \Omega_R$ for $t<1$ (only for $t>1$) )) \BLACK
%\end{proof}

\vspace{1ex}

So far we proved estimates of $u,\,v$ and $\mathcal U$ on the inner domain $\Omega_{R+2}$ for all $t>0$. \BLACK Next we consider estimates on domains away from the boundary. As in the inner domain case, we estimate $v$ and $\mathcal U = u-v$. 

\BLACK 
\vspace{1ex}

\begin{lem}\label{exterior-est1}
Let  $1<  p \leq q <\infty$ and $-\frac{n}{q}< s_0 \leq s < \frac{n}{p'}$. %\BLUE (In \eqref{exterior-est1-ineq}, we do not need $s_0 \leq s$ due to the whole space estimate \eqref{differentweightest-derivative} \RED sure? It is (3.11)? \RED In Thm. 3.3 we always have $s_0\leq s$!) 
\BLACK  Then it holds for $v$   defined as the solution of the heat equation \eqref{heat-eq-2} that 
%\begin{align}\label{exterior-est1-ineq}
%\|\nabla^\alpha v(t)\|_{L^{q}_{s_0}(\Omega)} \leq C t^{-\frac{n}{2}\big(\frac{1}{p}-\frac{1}{q}\big)-\frac{s-s_0}{2}-\frac{|\alpha|}{2}} \|u_0\|_{L^p_{s}(\Omega)}, \quad t> 0,
%\end{align}
\begin{align}\label{exterior-est1-ineq}
&\| v(t)\|_{L^{q}_{s_0}(\Omega)} \leq C t^{-\frac{n}{2}\big(\frac{1}{p}-\frac{1}{q}\big)} (1+t)^{ -\frac{s-s_0}{2}\BLACK} \|u_0\|_{L^p_{s}(\Omega)}, \quad t> 0,\\
 &\|\nabla  v(t)\|_{L^{q}_{s_0}(\Omega)} \leq C\big( t^{-\frac{n}{2}\big(\frac{1}{p}-\frac{1}{q}\big)-\frac{1}{2}} (1+t)^{ -\frac{s-s_0}{2}} + t^{-\frac{n}{2q}-\frac{s}{2}}  \big)\|u_0\|_{L^p_{s}(\Omega)}, \quad t> 0, 
\end{align}
uniformly in $t$. 
\end{lem}

%\vspace{1ex}
\begin{proof} 
The estimate \eqref{exterior-est1-ineq} is directly obtained by the definition of $v$ in \eqref{whole-est-base} and estimates of $\tilde{v}=e^{-tA_0}v_0$ by Proposition \ref{whole-est-basis}, $|\alpha|=0$.  
For $|\alpha|=1$, we note that 
 \begin{align*}
\nabla v = (1-\psi_R) \nabla \tilde{v}  -(\nabla\psi_R) \cdot\tilde{v} \BLACK + \nabla \mathbb{B}' [(\nabla \psi_R)\cdot \tilde{v}],\quad t>0. 
\end{align*}
Then \eqref{B}, $\supp (\nabla \psi_R) \subset \Omega_{R+2}$, and \eqref{est-tildev}, which by Proposition \ref{whole-est-basis} holds for all $t> 0$, imply that 
$$
\|(\nabla\psi_R) \tilde{v} - \nabla \mathbb{B}' [(\nabla \psi_R)\cdot \tilde{v}]\|_{L^q_{s_0}(\Omega)} \leq 
C \| \tilde{v}\|_{L^\infty(\mathbb{R}^n)} \leq C t^{-\frac{n}{2q}-\frac{s}{2}} \|u_0\|_{L^p_{s}(\Omega)}. 
$$
Moreover, again by Proposition \ref{whole-est-basis},  one obtains that 
\begin{align*}
\|\nabla \tilde v(t)\|_{L^{q}_{s_0}(\Omega)} \leq C  t^{-\frac{n}{2}\big(\frac{1}{p}-\frac{1}{q}\big)-\frac{1}{2}} (1+t)^{ -\frac{s-s_0}{2}\BLACK} \|u_0\|_{L^p_{s}(\Omega)}
\end{align*}
for all $t>0$.
\end{proof}

%\begin{proof}
% Using \eqref{form-v},  we see that 
% \begin{align}
 %    \begin{aligned}
%\|v\|_{L^q_{s_0}(\Omega_{0})} & \leq  \|e^{-t A_0}v_{01}\|_{L^q_{s_0}(\Omega_{0})}+ \displaystyle\int_{0}^1 \|e^{-(t-\tau) A_0}F(\tau)\|_{L^q_{s_0}(\Omega_{0})} \dtau\\ 
%& \quad +
%\displaystyle\int_{1}^{t-1} \|e^{-(t-\tau) A_0}F(\tau)\|_{L^q_{s_0}(\Omega_{0})}\dtau
% + \displaystyle\int_{t-1}^{t} \|e^{-(t-\tau) A_0}F(\tau)\|_{L^q_{s_0}(\Omega_{0})}\dtau\\
%& := J_1 +J_2 +J_3+J_4.  
% \end{aligned}
%\end{align}

%Concerning $J_1$, by \eqref{v01} we can regard the Stokes semigroup on the whole space and thus we get that 
%$$
%J_1 \leq \|e^{-t A_0}v_{01}\|_{L^q_{s_0}(\mathbb{R}^n)}. 
%$$
%Hence \eqref{whole-est-base} with \eqref{est-v0} directly imply \eqref{exterior-est1-ineq}. 

%As for $J_2$, \eqref{whole-est-base} and \eqref{est-F2} also direclty derive \eqref{exterior-est1-ineq}. For $J_3$,     \eqref{whole-est-base} and \eqref{est-F1} are used as 
%\begin{align}
%     \begin{aligned}
%J_3 &\leq C \displaystyle\int_1^{(t-1)/2} (t-\tau)^{-\frac{n}{2}\big(\frac{1}{p}-\frac{1}{q}\big)-\frac{s-s_0}{2}}\tau^{-\frac{n}{2p}-\frac{s}{2}} \dtau \,
%\|v_{01}\|_{L^q_s(\Omega)}\\
%& \quad + \displaystyle\int_{(t-1)/2}^{t-1} (t-\tau)^{-\frac{n}{2}\big(\frac{1}{p}-\frac{1}{q}\big)}\tau^{-\frac{n}{2p}-\frac{s_0}{2}} \dtau \,
%\|v_{01}\|_{L^q_{s_0}(\Omega)}\\
%& \leq Ct^{-\frac{n}{2}\big(\frac{1}{p}-\frac{1}{q}\big)-\frac{s-s_0}{2}}\|u_0\|_{L^q_{s}(\Omega)}
%\end{aligned}
%\end{align}
%for $t \geq 2$. The estimate of $J_4$ follows the lines as those for $J_2$ and $J_3$.  
%\end{proof}

\vspace{1ex} 
 The crucial estimates of the solution $u=e^{-t A}u_0$ of the Stokes problem at spatial infinity are described in the next Proposition. \BLACK

\vspace{1ex}

\begin{prop}\label{exterior-est2}
Let  $1< p \leq q <\infty$, $ \frac{n}{2} \big(\frac{1}{p}-\frac{1}{q}\big)<1$, 
$-\frac{n}{q}< s_0 \leq s < \frac{n}{p'}\BLACK $ with $s \geq 0$ %\BLUE (We do not need $s_0 \leq s$ due to the following proof.)
\BLACK and $\Omega^{R}=\Omega \setminus\bar B_{R}.$ 
Assume that $u=e^{-t A}u_0$ for $u_0 \in L^p_{\sigma,s}(\Omega)$. 
Then it holds that 
\begin{align}\label{exterior-est2-ineq}
\|u(t)\|_{L^{q}_{s_0}(\Omega^{R+4}\BLACK)} \leq C t^{-\frac{n}{2}\big(\frac{1}{p}-\frac{1}{q}\big)-\frac{s-s_0}{2}} \|u_0\|_{L^p_{s}(\Omega)} \quad \textrm{uniformly in } t\geq  2. 
\end{align}
\end{prop}

\vspace{1ex}

\vspace{1ex}

For the proof of Proposition \ref{exterior-est2} we consider the reduction to the whole space problem,
\BLACK
\begin{equation}\label{mathscrU}
\mathscr{U} = \varphi^R \mathcal U-\widehat{\mathbb{B}}\BLACK [(\nabla \varphi^R)\cdot \mathcal U], 
\end{equation}
where $\varphi^R=1-\varphi_{R+1}$ such that $\supp\nabla \varphi^R \subset \overline{D_{R+3}}$ and where $\widehat{\mathbb{B}}$ is the Bogovski\u{\i} operator related to $D_{R+3}$.  
Then $\mathscr{U}(0)=0$ since $\supp\, \mathcal U(0)\subset \overline{D_{R+2}}$, see \eqref{U_0}. \BLACK Moreover, $ \mathscr{U}$ solves the modified inhomogeneous Stokes problem
\begin{align}\label{the-Stokes-eq-modify}
\del_t \mathscr{U}-\Delta \mathscr{U} +\nabla (\varphi^R \pi)=G, \ \  \div \mathscr{U}=0 \;\mbox{ on } \,  \mathbb{R}^n, \  \ 
% \mbox{ and }
\ \mathscr{U}(0)=0,   
\end{align} 
on the whole space where $\pi$ is an associated pressure to $u$ in \eqref{the-Stokes-eq-0} \BLACK and
\begin{align}\label{def-g}
\begin{aligned}
G(t) & = -2(\nabla \varphi^R)\cdot \nabla \mathcal U -(\Delta \varphi^R)\mathcal U-\widehat{\mathbb{B}}[\nabla\varphi^R)\Delta\mathcal U] + \Delta\widehat{\mathbb{B}}[(\nabla\varphi^R)\cdot\mathcal U] \\
& \qquad + (\nabla \varphi^R) \pi
+ \widehat{\mathbb{B}}[(\nabla \varphi^R)\cdot \nabla \pi]
+ \widehat{\mathbb{B}}[(\nabla \varphi^R) \cdot F]  - \varphi^R F\BLACK \\
& =: \sum_{j=1}^{ 8\BLACK} G_j(t). 
\end{aligned}
\end{align}

 We first estimate the inhomogeneous term $G$.

\vspace{1ex}

\begin{lem}\label{est-G1}
Let $1< \tilde p\leq p <\infty$,
$-\frac{n}{p} < s_0 \leq s < \frac{n}{p'}$  with $s \geq 0$.  %\BLUE(We do not need $s_0 \leq s$ due to the cut off property.)
\BLACK Then there holds 
$$
\|G(t)\|_{L^{\tilde p}_{s_0}(\mathbb{R}^n)} \leq Ct^{-\frac{n}{2p}-\frac{s}{2}}\|u_0\|_{L^p_s(\Omega)},\quad t\geq 2 .
$$
\end{lem}
%\vspace{1ex}

\begin{proof}
First let $\tilde p=p$. \BLACK Due to the support properties of $\varphi^R$ and $F$ we see immediately that $G_7=G_8=0$. 
Since $\mathcal U=u-v$, the cut-off property of $\nabla \varphi^R$, properties of the Bogovski\u{\i} operator \eqref{B} and \eqref{B-2},  rewriting the term $(\nabla \varphi^R)\cdot \Delta \mathcal U$ in the form $\div\!\{\nabla \varphi^R \cdot  \nabla \mathcal U\} - 
\nabla^2 \varphi^R:\nabla\mathcal U$, {\em cf.} \eqref{commutator}, and
 local decay estimates in Proposition \ref{interior-domain2} and Lemma \ref{interior-domain1}, yield the estimates for $G_{j}$, $j=1,2,3,4$.  

Concerning $G_5$ and $G_6$ we may assume that the pressure $\pi$ in \eqref{the-Stokes-eq-0} has a vanishing integral mean on $D_{R+3}$ for a.a. $t>0$ so that by the Poincar\'e-Friedrichs inequality 
\begin{align*} %\label{pressure-estimate}
\|\pi(t)\|_{L^p(D_{R+3})} \leq C \|\nabla \pi(t)\|_{L^p(D_{R+3})}. 
\end{align*}
Moreover, the identity $\nabla \pi = -\del_t  u+ \Delta u$, Proposition \ref{interior-est-base} and the assumption $s<\frac{n}{p'}$ imply that 
$$\|\nabla \pi(t)\|_{L^p(D_{R+3})} \leq Ct^{-\frac{n}{2}}\|u_0\|_{L^p_s(\Omega)} \leq Ct^{-\frac{n}{2p}-\frac{s}{2}} \|u_0\|_{L^p_s(\Omega)}\quad  \textrm{  for a.a. }\; t\geq 2.$$ 
Since $\supp G(t)\subset \overline \Omega_{R+4}$, by H\"older's inequality $\|G(t)\|_{L^{\tilde p}_{s_0}(\mathbb{R}^n)}  \leq c\|G(t)\|_{L^{p}_{s_0}(\mathbb{R}^n)}$ for any $1<\tilde p<p$. \BLACK  Now Lemma \ref{est-G1} is proved.
\end{proof}

\begin{rem}\label{rem1}
{\rm 
By virtue of the cut-off property of all the terms in $G$, see \eqref{def-g}, \BLACK we can take $s$ and 
$s_0$ satisfying  
$-\frac{n}{p}< s_0 \leq s < \frac{n}{p'}$ with $s \geq 0$.
Furthermore, we can choose $\tilde{p}$ close to $1$ independent of $s$ and $s_0$ for the inhomogeneous weight 
$(1+|x|)^s$.  
These facts also hold for Lemma \ref{est-G2} below. 
%(\BLUE K-K's \BLACK preprint  \BLACK do not write this remark. But this remark is important to prove Theorem \ref{Stokes-w} below. )
}
\end{rem}

\vspace{1ex}
For $0<t\leq 2$ we get another decay rate of $G(t)$.
\vspace{1ex}

\begin{lem}\label{est-G2}
Let $1< \tilde{p} \leq   p <\infty$ and $-\frac{n}{p}< s_0 \leq s < \frac{n}{p'}\BLACK $ with $s \geq 0$.  %\BLUE(We do not need $s_0 \leq s$ due to the cut off property.)
\BLACK Then  %for any $\alpha \in (0, 1/2p')$ 
there holds  
$$
\|G(t)\|_{L^{\tilde{p}}_{s_0}(\mathbb{R}^n)} \leq C t^{-\frac12-\frac1{2p}} \BLACK \|u_0\|_{L^p_s(\Omega)},\quad  0< t< 2.
$$  
\end{lem}
\vspace{-4mm}

In order to prove Lemma \ref{est-G2}, we prepare the following estimate for the pressure $\pi$. 

\vspace{1ex}

\begin{lem}\label{pressure-est-crucial}
Let $1<p<\infty$ and $D$ be a bounded smooth domain such that $\overline D  \subset\Omega$. Then the pressure term $\pi(t)$ defined by \eqref{the-Stokes-eq-0} satisfies the estimate
\begin{equation}\label{pilocal}
\|\pi(t)\|_{L^p(D)} \leq C \big(\|\nabla u\|_{L^p(\Omega)} + \|\nabla^2 u\|_{L^p(\Omega)}^{1/p}\, \|\nabla u\|_{L^p(\Omega)}^{1-1/p} \big),\quad 0<t<2. \BLACK  
\end{equation}
%\begin{align}\label{pressure-estimate-2}
%\|\pi(t)\|_{L^p(D_{R+3})} \leq C\|(-\Delta)^{1-\alpha} u(t)\|_{L^p(\Omega)} .
%\end{align}
In particular,  for any $s\geq 0$, \BLACK
\begin{align}\label{pressure-estimate-3}
\|\pi(t)\|_{L^p(D_{R+3})}
\leq C t^{-\frac12\big(1+\frac1p\big)} \BLACK\|u_0\|_{L^p_s(\Omega)},\quad 0<t<2. 
\end{align}
\end{lem}

\begin{proof}
%For \eqref{pressure-estimate-2} we refer to \cite[Lemma 3.4, Remark 3.5]{Gr}. 
%\RED ((Proof as in  \cite[Lemma 3.4, Remark 3.5]{Gr}:))  
Let $D$ be a bounded smooth domain such that $\overline D  \subset\Omega$ and assume that $\pi$ satisfies $\int_D \pi(t)\dx = 0$ for all $0<t<2$. 
Then with the help of Bogovski\u\i's operator for the domain $D$ for any $\varphi\in L^{p'}_m(D)$ 
there exists  $\Phi=\mathbb B\varphi \in W^{1,p'}_0(D)$ such that $\div \Phi=\varphi$ and $\|\Phi\|_{W^{1,p'}(D)}\leq c \|\varphi\|_{L^{p'}(D)}$. 
Obviously, we may extend $\Phi$ by $0$ from $D$ to $\Omega$ yielding a function $\Phi\in W^{1,p'}_0(\Omega)$ satisfying $\|\Phi\|_{W^{1,p'}(\Omega)} = \|\Phi\|_{W^{1,p'}(D)}$. Finally, let $\nabla\varpi=(I-P)\Phi = (I-P)\mathbb B\varphi$ denote the pressure part of $\Phi$ in its Helmholtz decomposition. Note that 
\begin{equation}\label{varpi}
\|\nabla\varpi\|_{L^{p'}(\Omega)}+\|\nabla^2\varpi\|_{L^{p'}(\Omega)} \leq c\|\Phi\|_{W^{1,p'}(\Omega)} \leq \|\varphi\|_{L^{p'}(D)}.
\end{equation}

Now we get for any $t\in(0,2)$ that $(I-P)\dfrac{\del}{\del t}u(t)=0$ in $\Omega$ so that with the outer normal $\textsl{n}$ on $\partial\Omega$
\begin{align}\label{pressure-id}
\begin{aligned}
\langle \pi(t),\varphi \rangle_D & = \langle \pi(t), \div\Phi \rangle_D \\
& = - \langle \nabla\pi(t), \Phi \rangle_\Omega\\
& = \langle (-\Delta) u(t), (I-P)\Phi \rangle_\Omega\\
& = \langle (-\Delta) u(t), \nabla\varpi \rangle_\Omega\\
&  = \langle \nabla u(t), \nabla^2\varpi\rangle_\Omega - \langle \textsl{n}\cdot\nabla u, \nabla\varpi\rangle_{\partial\Omega}. 
\end{aligned} 
\end{align}
Combining \eqref{varpi} and \eqref{pressure-id}  we get by trace estimates (\cite[Theorem II.4.1]{Galdi-1},  \cite[Prop. 1.50]{Shi-monogr.vol.1})  that
\begin{align*} %\label{pressure-est}
\begin{aligned}
|\langle \pi(t),\varphi \rangle_D| & \leq C(\|\nabla u(t)\|_{L^{p}(\Omega)} + \|\nabla u(t)\|_{L^{p}(\partial\Omega)}) \|\varphi\|_{L^{p'}(D)}\\
& \leq C\big( \|\nabla u\|_{L^p(\Omega)} + \|\nabla^2 u\|_{L^p(\Omega)}^{1/p}\, \|\nabla u\|_{L^p(\Omega)}^{1-1/p} \big) \,\|\varphi\|_{L^{p'}(D)}.
\end{aligned} 
\end{align*}
A final duality argument proves \eqref{pilocal}. 

 Then $\eqref{e-tAt}_3$ and \eqref{pilocal} imply  for $D=D_{R+3}$ and $0<t<2$ the estimate \eqref{pressure-estimate-3}. \BLACK
%\begin{align*} %\label{pressure-estimate-3}
%\|\pi\|_{L^p(D_{R+3})}\leq C\|(-\Delta) u\|_{L^p(\Omega)}^{1-2\alpha} \|\nabla u\|_{L^p(\Omega)}^{2\alpha}\leq Ct^{\alpha-1} \|u_0\|_{L^p_s(\Omega)}, \quad 0<t<2. \end{align*}
\end{proof}

%\vspace{1ex}

\begin{proof}[Proof of Lemma  \ref{est-G2}] 
Recall that $G_7=G_8=0$ in \eqref{def-g}.  The cut-off property of $\nabla \phi^R$, properties of the Bogovski\u{\i} operator \eqref{B} and \eqref{B-2}, the identity $(\nabla \varphi^R)\cdot \Delta \mathcal U = \div\!\{\nabla \varphi^R \cdot  \nabla \mathcal U\} - 
\nabla^2 \varphi^R:\nabla\mathcal U$, \BLACK and Lemma \ref{est-mathcalU-tsmall}  for $\mathcal U$ yield the estimates for $G_{j}$,  $j=1,2,3,4$.

%As for $G_{5}$ 
%\begin{align}\label{pressure-estimate-2}
%\|\pi\|_{L^p(D_{R+3})}
%\leq C(\|\nabla^2 u\|_{L^p(\Omega)}^{1/p}\|\nabla %u\|_{L^p}^{1-1/p}+\|\nabla u\|_{L^p}). 
%\end{align}
%(\eqref{pressure-estimate-2} 
%is written in the Shibata's book for the usual Stokes operator. A little long proof about it is based on the duality argument. I will investigate more details of this.  
%)
Obviously, \eqref{pressure-estimate-3}  yields the estimate $\|G_5(t)\|_{L^p(\IR^n)} = \|(\nabla \varphi^R) \pi(t)\|_{L^p(\IR^n)}\leq Ct^{-\frac12-\frac{1}{2p}}$. As for $G_6$, we apply the property of the Bogovski\u{\i} operator \eqref{B-2} to
$$
(\nabla \varphi^R)\cdot \nabla \pi= \div\{(\nabla \varphi^R)\pi\}-(\Delta \varphi^R)\pi
$$
and use the same argument as that in $G_5$.  
\end{proof}

%\vspace{1ex}
\begin{proof}[Proof of Proposition \ref{exterior-est2}]  
Let $P_0$ be the Helmholtz projection on the whole space. It follows from applying $P_0$ to \eqref{the-Stokes-eq-modify} and the Duhamel principle that 
\begin{align}\label{mathscr-est-form}
    \begin{aligned} 
\mathscr{U}(t) & = \int_0^t e^{-(t-\tau) A_0} P_0 G(\tau) \dtau\\  
& = \Bigg(\int_{0}^1 + \int_{1}^{t-1} + \int_{t-1}^t\Bigg)
  e^{-(t-\tau) A_0} P_0 G \dtau\\
& =: L_1(t) +L_2(t) +L_3(t). 
\end{aligned}
\end{align}
We see from Proposition \ref{whole-est-basis} and Lemma \ref{est-G2} for $L_1$ that
\begin{align*}
\|L_1(t)\|_{L^{q}_{s_0}(\mathbb{R}^n) } 
& \leq C\int_{0}^{1} (t-\tau)^{-\frac{n}{2}\big(\frac{1}{p}-\frac{1}{q}\big)-\frac{s-s_0}{2}} \tau^{-1+\frac1{2p'}\BLACK } \dtau\, \|u_0\|_{L^p_{s}(\Omega)}\\
& \leq C t^{-\frac{n}{2}\big(\frac{1}{p}-\frac{1}{q}\big)-\frac{s-s_0}{2}}\|u_0\|_{L^p_{s}(\Omega)},\quad t\geq 2.
\end{align*}
As for $L_2$, Proposition \ref{whole-est-basis} with exponents $(q, \bar{p})$, $\bar{p}<\frac{n}{2}$,  $\bar{p}\leq p$, \BLACK and Lemma \ref{est-G1} imply that   
$$
\|L_2(t)\|_{L^{q}_{s_0}(\mathbb{R}^n)} \leq C\Bigg(\int_{1}^{t/2} + \int_{t/2}^{t-1} \Bigg) (t-\tau)^{-\frac{n}{2}\big( \frac{1}{\bar{p}} - \frac{1}{q}\big)-\frac{s-s_0}{2}} \tau^{-\frac{n}{2p}-\frac{s}{2}} \dtau\, \|u_0\|_{L^p_{s}(\Omega)}.
$$
%\begin{align*}
%\|L_2\|_{L^{q}_{s_0}(\mathbb{R}^n) }  & \leq \displaystyle\int_{1}^{t-1} (t-\tau)^{-\frac{n}{2}\big( \frac{1}{\tilde{p}} - \frac{1}{q}\big)-\frac{s-s_0}{2}} \tau^{-\frac{n}{2p}-\frac{s}{2}} \dtau \|u_0\|_{L^p_{s}(\Omega)} \\
%& =  \Big(\int_{1}^{t/2} (t-\tau)^{-\frac{n}{2}\big(\frac{1}{\tilde{p}}-\frac{1}{q}\big)-\frac{s-s_0}{2}} \tau^{-\frac{n}{2p}-\frac{s}{2}} \dtau \\
%&\qquad + \int_{t/2}^{t-1} (t-\tau)^{-\frac{n}{2}\big(\frac{1}{\tilde{p}}-\frac{1}{q}\big)-\frac{s-s_0}{2}} \tau^{-\frac{n}{2p}}\dtau \Big) \|u_0\|_{L^p_{s}(\Omega)}\\
%& = \Big(\displaystyle\int_{1}^{t/2} (t-\tau)^{-\frac{n}{2}\big(\frac{1}{\tilde{p}}-\frac{1}{q}\big)-\frac{s-s_0}{2}}\tau^{-\frac{n}{2p}-\frac{s}{2}}\dtau\\
%&\qquad +\displaystyle\int_{1}^{t/2} \tau^{-\frac{n}{2}\big(\frac{1}{\tilde{p}}-\frac{1}{q}\big)-\frac{s-s_0}{2}}(t-\tau)^{-\frac{n}{2p}-\frac{s}{2}}\Big)\|u_0\|_{L^p_{s}(\Omega)}. 
%\end{align*}
Concerning the integral on $(1,\frac{t}{2})$ there holds $\frac{1}{t-\tau}\leq \frac{1}{\tau}$. Hence
$$
(t-\tau)^{-\frac{n}{2}\big(\frac{1}{ \bar{p}}-\frac{1}{q}\big)-\frac{s-s_0}{2}} \tau^{-\frac{n}{2p}-\frac{s}{2}}  \leq (t-\tau)^{-\frac{n}{2}\big(\frac{1}{p}-\frac{1}{q}\big)-\frac{s-s_0}{2}} \tau^{-\frac{n}{2 \bar{p}}-\frac{s}{2}} .
$$
For the integral on $(\frac{t}{2}, t-1)$ there holds $\frac{1}{t-\tau}\geq \frac{1}{\tau}$ so that
$$
(t-\tau)^{-\frac{n}{2}\big(\frac{1}{\bar{p}}-\frac{1}{q}\big)-\frac{s-s_0}{2}} \tau^{-\frac{n}{2p}-\frac{s}{2}}  \leq (t-\tau)^{-\frac{n}{2 \bar{p}}-\frac{s}{2}} \tau^{-\frac{n}{2}\big(\frac{1}{p}-\frac{1}{q}\big)-\frac{s-s_0}{2}} .
$$
%, it holds that 
%$$
%(t-\tau)^{-\frac{n}{2}\big(\frac{1}{\tilde{p}}-\frac{1}{q}\big)-\frac{s-s_0}{2}} =(t-\tau)^{-\frac{n}{2}\big(\frac{1}{p}-\frac{1}{q}\big)-\frac{s-s_0}{2}}(t-\tau )^{-\frac{n}{2}\big(\frac{1}{\tilde{p}}-\frac{1}{p}\big)}\leq (t-\tau)^{-\frac{n}{2}\big(\frac{1}{p}-\frac{1}{q}\big)-\frac{s-s_0}{2}} s^{-\frac{n}{2}\big(\frac{1}{\tilde{p}}-\frac{1}{p}\big)}
%$$
%and
%$$
%(t-\tau)^{-\frac{n}{2p}-\frac{s}{2}} =(t-\tau)^{-\frac{n}{2}\big(\frac{1}{p}-\frac{1}{q}\big)-\frac{s}{2}}(t-\tau)^{-\frac{n}{2q}} \leq (t-\tau)^{-\frac{n}{2}\big(\frac{1}{p}-\frac{1}{q}\big)-\frac{s}{2}}\tau^{-\frac{n}{2q}}. 
%$$
%\BLUE if $\min \big\{\frac{n}{2}\big(\frac{1}{p}-\frac{1}{q}\big)+\frac{s-s_0}{2}, \frac{n}{2p}+\frac{s}{2}\big\} > 1$ for $t \geq 2$, \BLACK 
Therefore, since $\bar{p}<\frac{n}{2}$, we see that both integrals below converge as $t\to\infty$ and  
\begin{align*}
\|L_2(t)\|_{L^{q}_{s_0}(\mathbb{R}^n) }  & \leq \Bigg( \int_{1}^{t/2} (t-\tau)^{-\frac{n}{2}\big(\frac{1}{p}-\frac{1}{q}\big)-\frac{s-s_0}{2}} \tau^{-\frac{n}{2 \bar{p}}-\frac{s}{2}\BLACK} \dtau \\
& \qquad + \int_{t/2}^{t-1} (t-\tau)^{-\frac{n}{2\bar{p}} -\frac{s}{2}}\tau^{-\frac{n}{2} \big(\frac{1}{p}-\frac{1}{q}\big)-\frac{s-s_0}{2}} \dtau \Bigg)
\|u_0\|_{L^p_{s}(\Omega)}\\
& \leq C t^{-\frac{n}{2}\big(\frac{1}{p}-\frac{1}{q}\big)-\frac{s-s_0}{2}}
\|u_0\|_{L^p_{s}(\Omega)}, \quad t\geq 2. 
\end{align*}  
As for $L_3$, since by assumption $\frac{n}{2} \big(\frac{1}{p}-\frac{1}{q}\big)\in [0,1)$,  
Proposition \ref{whole-est-basis} and Lemma \ref{est-G1} imply for $t>2$ that 
\begin{align*}
\|L_3(t)\|_{L^{q}_{s_0}(\mathbb{R}^n) } & \leq 
C\displaystyle\int_{t-1}^{t}
(t-\tau)^{ -\frac{n}{2} \big(\frac{1}{p}-\frac{1}{q}\big)} \tau^{-\frac{n}{2p}-\frac{s}{2}} \dtau \|u_0\|_{L^p_{s}(\Omega)}\\
& \leq C_\kappa  t^{-\frac{n}{2p} -\frac{s}{2}} \|u_0\|_{L^p_{s}(\Omega)}\BLACK \\
& \leq C_\kappa t^{-\frac{n}{2}\big(\frac{1}{p}-\frac{1}{q}\big)-\frac{s-s_0}{2}} \|u_0\|_{L^p_{s}(\Omega)}, \quad t\geq 2.
\end{align*}
By $\mathscr{U}=\mathcal U = u-v$ on $\Omega^{R+4}$, see \eqref{mathscrU}, and Lemma \ref{exterior-est1}, we obtain the desired estimate for $u$. This completes the proof of Proposition \ref{exterior-est2}. 
\end{proof}

\vspace{1ex}

Now we are in a position to prove the weighted $L^p$-$L^q$ estimates for the Stokes semigroup as given in Theorem \ref{Stokes-w}. Note that due to re-scalings of the time parameter $t$ it suffices to prove \eqref{etA>1}, \eqref{n-etA>} for $t>2$ and \eqref{etA<} for $t<1$. \BLACK

%\vspace{1ex}

\begin{proof}[Proof of Theorem \ref{Stokes-w}]
The estimate \eqref{etA>1} for $u(t)=e^{-tA}u_0$, \BLACK $t>2$, is directly obtained by combining Proposition \ref{interior-domain2} with Proposition \ref{exterior-est2}.  
Indeed, concerning estimates on the inner domain, 
\eqref{w-local-decay-R} with $p$ in Proposition \ref{interior-domain2} together with a Sobolev embedding imply that for $p<n$ and $n \big(\frac{1}{p}-\frac{1}{q}\big) < 1\BLACK$  %(($<1$ is needed in Prop. 3.9))
\BLACK
$$ \|u(t)\|_{L^q_{s_0}(\Omega_R)}  \leq c\|u(t)\|_{W^{1,p}_{s_0}(\Omega_R)} \leq ct^{-\frac{n}{2p} - \frac{s}{2}} \|u_0\|_{L^p_s(\Omega)} \leq ct^{-\frac{n}{2}\big(\frac1p -\frac1q\big) - \frac{s-s_0}{2}} \|u_0\|_{L^p_s(\Omega)}; $$ 
in the last step we used that $s_0>-\frac{n}{q}$ and $t>2$. \BLACK On the other hand, if 
$p\geq n$, there exists $\tilde p<n$ such that $\tilde p < q <\tilde p^*  = \frac{n\tilde p}{n-\tilde p}\BLACK$; then \eqref{w-local-decay-R} with $p$ in Proposition \ref{interior-domain2} together with a Sobolev embedding implies again that  
\begin{align*}
\|u(t)\|_{L^q_{s_0}(\Omega_R)} & \leq c\|u(t)\|_{W^{1,\tilde p}_{s_0}(\Omega_R)} \leq  c\|u(t)\|_{W^{1,p}_{s_0}(\Omega_R)}\\
& \leq  ct^{-\frac{n}{2p} - \frac{s}{2}} \|u_0\|_{L^p_s(\Omega)} \leq ct^{-\frac{n}{2}\big(\frac1p -\frac1q\big)- \frac{s-s_0}{2}} \|u_0\|_{L^p_s(\Omega)}.
\end{align*}
This together with Proposition \ref{exterior-est2}, the estimate on an outer domain away from the boundary,  say, on $\Omega^R$ instead of $\Omega^{R+4}$, \BLACK derives \eqref{etA>1} for $t>2$. 
We can remove the condition $n (\frac{1}{p}-\frac{1}{q})\leq1$ by properties of the semigroup. In fact, we first consider the case 
$$
1 \leq n\Big(\frac{1}{p}-\frac{1}{q}\Big)<2. 
$$
For fixed $p<q$, \BLACK there exists $p_1$ such that we can divide the interval $(p,q)$ by $p_1$ satisfying that 
$n\big(\frac{1}{p}-\frac{1}{p_1}\big)<1$ and $n\big(\frac{1}{p_1}-\frac{1}{q}\big)<1$. 
Hence, applying \eqref{etA>1} for $t>2$ twice under the condition $n\big(\frac{1}{p}-\frac{1}{q}\big)<1$, we obtain the estimate for fixed $p, q$.  
More precisely,  
\begin{equation}\label{u(t)-semigroup-est}
\|e^{-tA}u_0\|_{L^q_{s_0}(\Omega)} \leq C\|e^{-t A/2}\|_{\mathcal{L}(L^{p_1}_{s_0}, L^{q}_{s_0}\BLACK)} \|e^{-tA/2} u_0\|_{L^{p_1}_{s_0}(\Omega)}\leq 
C t^{-\frac{n}{2}\big(\frac{1}{p}-\frac{1}{q}\big)-\frac{s-s_0}{2}}\|u_0\|_{L^p_s(\Omega)}.      
\end{equation}
For the case $n\big(\frac{1}{p}-\frac{1}{q}\big)<k$, $k \geq 2$,  mathematical induction on $k$ yields the estimate. 

Concerning \eqref{n-etA>}, where $s_0\geq 0$, we first show that 
\begin{align}\label{mathscrU-est-by-semigroup}
\|\nabla \mathscr U(t)\|_{L^q_{s_0}(\mathbb{R}^n)} \leq C\max\big\{ t^{-\frac12}, t^{-\frac{n}{2q }-\frac{s_0}{2}}\big\} \|u_0\|_{L^{q }_{s_0}(\Omega)}, \quad t>2,  
\end{align}
In fact, by virtue of \eqref{mathscr-est-form}, 
\begin{align}
    \begin{aligned}
 \|\nabla \mathscr{U}(t)\|_{L^q_{s_0}(\mathbb{R}^n)} 
& \leq  \Bigg(\int_{0}^1 + \int_{1}^{t-1} + \int_{t-1}^t \Bigg)
 \|\nabla e^{-(t-\tau) A_0} P_0 G(\tau)\|_{L^q_{s_0}(\mathbb{R}^n)} \dtau\\ 
%& \leq  \int_{0}^1 \|\nabla e^{-(t-\tau) A_0}G(\tau)\|_{L^q_{s_0}(\mathbb{R}^n)} \dtau + \int_{1}^{t-1} \|\nabla e^{-(t-\tau) A_0}G(\tau) \|_{L^q_{s_0}(\mathbb{R}^n)}\dtau\\
%&\quad + \int_{t-1}^t \|\nabla e^{-(t-\tau) A_0} PG(\tau)\|_{L^q_{s_0}(\mathbb{R}^n)} \dtau\\
& =: L_4(t) +L_5(t) +L_6(t), 
\end{aligned}
\end{align}
where $L_4$ and $L_6$ are estimated by Proposition \ref{whole-est-basis} as well as Lemmata \ref{est-G2} and \ref{est-G1} with $\tilde p=q$ yielding the bounds $C t^{-\frac{1}{2}} \|u_0\|_{L^q_{s_0}(\Omega)}$ and $Ct^{-\frac{n}{2q}-\frac{s_0}{2}} \|u_0\|_{L^{q}_{s_0}(\Omega)}$, respectively. 
 The term $L_5$ is split into two integrals     $L_5=L_{5,1}+L_{5,2}$ defined on $[1, t/2]$ and $[t/2, t-1]$, respectively, and we exploit Proposition \ref{whole-est-basis} as well as Lemma \ref{est-G1} with $\tilde p\in (1,p]$.  
For $L_{5,1}$ note that $t-\tau\geq \frac{t}{2}$ on $(1,\frac{t}{2})$. 
Since $s_0\geq 0$, \BLACK %( For $s_0<0$ \BLACK  the details on convergence and decay rate are getting more complicated)
\begin{align*}
L_{5,1}(t) & \leq c\int_{1}^{t/2} (t-\tau)^{-\frac{n}{2}\big(\frac{1}{\tilde{p}}-\frac{1}{q}\big)-\frac{1}{2}} \tau^{-\frac{n}{2q} -\frac{s_0}{2}} \dtau \,\|u_0\|_{L^{q }_{s_0}(\Omega)} \\
& \leq c
t^{-\frac{n}{2}\big(\frac{1}{\tilde{p}}-\frac{1}{q}\big)-\frac{1}{2}}\!\! \int_1^{t/2} \!\!\! \tau^{-\frac{n}{  2q }-\frac{s_0}{2}} \dtau\,\|u_0\|_{L^{q }_{s_0}(\Omega)}\\
& \leq Ct^{-\frac{n}{2}\big(\frac{1}{\tilde{p}}-\frac{1}{q}\big)-\frac{1}{2}} \big(t^{-\frac{n}{2q }+1} + \log t\big)\|u_0\|_{L^{q }_{s_0}(\Omega)};
\end{align*}
for $\tilde p$ close to $1$ this product decays faster than $t^{-\frac12}$ as $t\to\infty$.

Concerning $L_{5,2}$, we will get the decay rate $\max\big\{t^{-\frac{1}{2}}, t^{-\frac{n}{2q}-\frac{s_0}{2}}\big\}$. Indeed, we use the simple estimates  $\tau^{-\frac{n}{2q} -\frac{s_0}{2}} \leq ct^{-\frac{n}{2q} -\frac{s_0}{2}}$
%$$ \int_{t/2}^{t-1} (t-\tau)^{-\frac{n}{2}\big(\frac{1}{\tilde{p}}-\frac{1}{q}\big)-\frac{1}{2}} \tau^{-\frac{n}{2q} -\frac{s_0}{2}} \dtau
% \leq ct^{-\frac{n}{  2q }-\frac{s_0}{2}} \int_{t/2}^{t-1} (t-\tau)^{-\frac{n}{2}\big(\frac{1}{\tilde{p}}-\frac{1}{q}\big)-\frac{1}{2}} \dtau.$$ 
and, with $\varepsilon_0 := \min\big(\frac{n}{2q},\frac12\big)$,
$$ (t-\tau)^{-\frac{n}{2}\big(\frac{1}{\tilde{p}}-\frac{1}{q}\big)-\frac{1}{2}} \leq ct^{\frac{n}{2q} - \varepsilon_0} \cdot (t-\tau)^{-\frac{n}{2\tilde{p}} - \frac12 +\varepsilon_0} \quad \textrm{ for }\tau\in(\tfrac{t}{2},t-1). $$
% ((we need \BLUE $\frac{n}{2q} - \varepsilon_0\geq 0$ \RED for the above inequality)) 
Since $-\frac{n}{2\tilde{p}}-\frac12+\varepsilon_0 < -1$ for $\tilde p$ close to $1$, the integral $\int_{t/2}^{t-1} (t-\tau)^{-\frac{n}{2}\big(\frac{1}{\tilde{p}}-\frac{1}{q}\big)-\frac{1}{2}} \dtau$ is bounded by $ct^{\frac{n}{2q} - \varepsilon_0}$ for all $t>2$.
Summarizing we get the upper bound
$$ L_{5,2}(t) \leq c t^{-\varepsilon_0-\frac{s_0}{2}} \|u_0\|_{L^{q }_{s_0}(\Omega)} \leq c %\max\big\{t^{-\frac{1}{2}}, 
t^{-\frac{n}{2q}-\frac{s_0}{2}}   \|u_0\|_{L^{q }_{s_0}(\Omega)}, \quad t>1. $$
%
%For $L_5$ we see from Proposition \ref{whole-est-basis} and  Lemma \ref{est-G1} with $\tilde{p}$ close to $1$ that 
%\begin{align}
%    \begin{aligned}
%    ct^{-\frac{n}{2} \big(\frac{1}{\tilde{p}}-\frac{1}{q}\big)-\frac{1}{2}} \int_1^{t/2}  \tau^{-\frac{n}{  2q }-\frac{s_0}{2}} \dtau &\leq Ct^{-\frac{n}{2}\big(\frac{1}{\tilde{p}}-\frac{1}{q}\big)-\frac{1}{2}} \big(t^{-\frac{n}{2q }+1} +1+ \log t\big) \\
%for $\tilde p$ close to $1$ this product decays faster than $t^{-\frac12}$ as $t\to\infty$.
%L_5(t) & \leq \displaystyle\int_{1}^{t-1} (t-\tau)^{-\frac{n}{2}\big(\frac{1}{\tilde{p}}-\frac{1}{q}\big)-\frac{1}{2}} \|G(\tau)\|_{L^{\tilde{p}}_{ s_0}(\Omega)} \dtau \\
%& \leq  \int_{1}^{t-1} (t-\tau)^{-\frac{n}{2}\big(\frac{1}{\tilde{p}}-\frac{1}{q}\big)-\frac{1}{2}} \tau^{-\frac{n}{  2q }-\frac{s_0}{2}} \dtau \,\|u_0\|_{L^{q}_{s_0}(\Omega)}\\
%& \leq C \max\big\{t^{-\frac{1}{2}}, t^{-\frac{n}{2q }-\frac{s_0}{2}}\big\} \|u_0\|_{L^{q}_{s_0}(\Omega)},\quad t>2.
%end{aligned}\end{align}
%
%For the second integral on $(\frac{t}{2},t-1)$ there holds $\tau\sim t$ and, since $\frac{n}{2\tilde{p}}+\frac{1}{2}>1$ for $\tilde p$ close to $1$,
%$$ ct^{-\frac{n}{2q}-\frac{s_0}{2}} \int_{t/2}^{t-1} (t-\tau)^{-\frac{n}{2\tilde{p}}-\frac{1}{2}} \dtau \leq Ct^{-\frac{n}{2q}-\frac{s_0}{2}}.$$
Hence we proved \eqref{mathscrU-est-by-semigroup}. 

Since $\mathscr{U}=\mathcal U = u-v$ on $\Omega^{ R+4}$, {\em i.e.} $u=\mathscr{U}+v$ on $\Omega^{R+4}$, and $u=\mathcal U+v$ on $\Omega_{R+4}$, \eqref{mathscrU-est-by-semigroup}, \eqref{UtW1q} \BLACK and
Lemma \ref{exterior-est1} \BLACK imply for $u(t)=e^{-tA}u_0$ that 
\begin{equation}\label{nablaut}
\|\nabla u(t)\|_{L^q_{s_0}(\Omega)} \leq 
C\max\big\{t^{-\frac{1}{2}}, t^{-\frac{n}{2 q \BLACK }-\frac{s_0}{2}}\big\}\|u_0\|_{L^{q }_{s_0}(\Omega)}, \quad t>2.
\end{equation}
Therefore, replacing $u_0$ by $u(t/2)=e^{-tA/2}u_0$ and applying the semigroup property as well as \eqref{u(t)-semigroup-est} we obtain that
\begin{equation}\label{nablaut2}
\big\|\nabla  e^{-t A}u_0 \big\|_{L^q_{s_0}(\Omega)} \leq C \max\big\{t^{-\frac{1}{2}}, t^{-\frac{n}{2q} -\frac{s_0}{2} }\big\} \,\cdot\, t^{-\frac{n}{2}\big(\frac{1}{p}-\frac{1}{q}\big)-\frac{s-s_0}{2}}\, \|u_0\|_{L^{p}_{s}(\Omega)}, \quad t>2. 
\end{equation}
Now \eqref{n-etA>} is completely proved for $t>2$ and hence even for $t>1$. 

%\RED ((Still needed? Next let us consider the assumption $u_0 \in L^{p_1}_{\sigma,s_1} \cap L^{p_2}_{\sigma,s_2}$  where $0 \leq s_0 \leq  s_j < \frac{n}{p_j'}$, $j=1,2$. If $t^{-\frac{n}{2 q }-\frac{s_0}{2}}< t^{-\frac{1}{2}}$, 
%we choose $p=p_1$ with $s=s_1$, and, in the other case, we take $p=p_2$ with $s=s_2$. Then it follows from \eqref{nablaut2} that ))
%\begin{equation*}
%\big\|\nabla  e^{-t A} u_0\big\|_{L^q_{s_0}(\Omega)} \leq C t^{-\frac{n}{2}\big(\frac{1}{p_1}-\frac{1}{q}\big)-\frac{1}{2}}t^{-\frac{s_1-s_0}{2} }\|u_0\big\|_{L^{p_1}_{s_1}(\Omega)}
%+ Ct^{-\frac{n}{2p_2}}  t^{-\frac{s_2}{2} } \; %(not:t^{-\frac{s_2-s_0}{2} }) \|u_0\big\|_{L^{p_2}_{s_2}(\Omega)}. 
%\end{equation*}
%\BLACK
%Furthermore, by the growth order estimate \eqref{etA>1-2} when we use \eqref{nablaut2} we can remove the conditions $s_0 \leq s_1$ or $s_0 \leq s_2$.  
%
Finally, the estimate \eqref{etA<} for $t<2$ is obtained by the resolvent estimate \eqref{equ:rse-w2}, an extension operator to the whole space, and the moment inequality  in $L^p_s(\IR^n)$.
Indeed, the resolvent estimates in Theorem \ref{res-weighted} and Remark \ref{01t} applied to $e^{-tA}$ show that 
\begin{align}\label{0t1-est}
\|u(t)\|_{L^p_{s}(\Omega)}+ t^{1/2}\|\nabla u(t)\|_{L^p_{s}(\Omega)}
+t\|\nabla^2 u(t)\|_{L^p_{s}(\Omega)} \leq C %e^{c_0 t}
\|u_0\|_{L^p_{s}(\Omega)}, \quad 0<t\leq 2,
\end{align}
where the constant $C$ is independent of $0<t<2$. 
%\BLUE ...as well as 
%Hence if $0< t \leq 2$, \BLUE the same argument as that in Remark \ref{01t} \BLACK yields Theorem \ref{Stokes-w}. 
%This %together with \RED \eqref{etA>1} (valid for $0<t<2$?)
%\eqref{etA<}    
%\BLUE (( \RED I think that 
%applying the weighted resolvent estimates \eqref{equ:rse-w2} in the Dunford integral with the interpolation inequality valid (3.3), not directly using $\eqref{e-tAt}_3$. 
%\BLUE ?))

 In the following let $E: W^{2,q}_{s}(\Omega) \cap W^{1,q}_{0,s}(\Omega) \to W^{2,q}_{s}(\IR^n) \cap W^{1,q}_{0,s}(\IR^n)$ denote an extension operator such that $Ev\big|_{\Omega} = v$ and 
\begin{equation}\label{homog-ext}
\|Ev\|_{\widehat W^{k,q}_{s}(\IR^n)} \leq c \|v\|_{\widehat W^{k,q}_{s}(\Omega)}, \BLACK\quad k=0,1,2,
\end{equation}
for all $v\in W^{2,q}_{s}(\Omega) \cap W^{1,q}_{0,s}(\Omega)$ with a constant $c>0$ independent of $v$, {\em cf.}  Lemma \ref{extention}. 
Moreover, we need properties of fractional Laplacians on $\IR^n$. 
Let $\beta=\frac{n}{2}\big(\frac1p-\frac1q\big)\in \big[0,\frac{n}{2}\big)$. 
Then the operator $(-\Delta)^{-\beta}$ equals in Fourier space a multiplier operator $c_\beta |\xi|^{-2\beta}$ with a constant $c_\beta>0$ so that by a classical result, see \cite[Theorem 2.4.6]{GrafakosI}, 
$$(-\Delta)^{-\beta} v(x) = c_\beta' (I_{2\beta}v)(x) = c_\beta' \int_{\IR^n} \frac{1}{|x-y|^{n-2\beta}} v(y)\dy,$$
where $c_\beta'>0$, {\em cf.} the proof of Proposition \ref{whole-est-basis}. Hence $(-\Delta)^{-\beta}: L^{p}_{s}(\IR^n) \to L^{q}_{s}(\IR^n)$ is a bounded, injective linear operator; actually, injectivity is a consequence of the condition $s>-\frac{n}{p}$.
Then any $v\in\mathcal D((-\Delta)^{\beta})$ satisfies the embedding estimate 
\begin{equation}\label{Deltabeta} 
\|v\|_{L^q_s(\IR^n)} \leq C\|(-\Delta)^\beta v\|_{L^p_s(\IR^n)}.
\end{equation}

Exploiting \eqref{Delta12} and \eqref{homog-ext}, \eqref{Deltabeta} we estimate $u(t)$ and $\nabla u(t)$, $0<t<2$, in \eqref{etA<} as follows: For $k=0,1$ and $\beta=\frac{n}{2}\big(\frac1p-\frac1q\big)\in \big[0,1-\frac{k}{2}\big)$ the moment inequality implies that  
\begin{align}\label{u(t)alpha}
\begin{aligned}
    \|\nabla^k u(t)\|_{L^q_{s_0}(\Omega)} &\leq \|\nabla^k Eu(t)\|_{L^q_{s}(\IR^n)}\\
    & \leq C\|(-\Delta)^{\beta+\frac{k}{2}} Eu(t)\|_{L^p_{s}(\IR^n)}\\
    & \leq C\|(-\Delta)Eu(t)\|_{L^p_{s}(\IR^n)}^{\beta+\frac{k}{2}} \,\|Eu(t)\|_{L^p_{s}(\IR^n)}^{1-\beta-\frac{k}{2}}\\
    & \leq C\| \nabla^2 \BLACK u(t)\|_{L^p_{s}(\Omega)}^{\beta+\frac{k}{2}} \, \|u(t)\|_{L^p_{s}(\Omega)}^{1-\beta-\frac{k}{2}}.
\end{aligned}
\end{align}
%\BLUE (We also can take $s_0=s$.) \BLACK  
Summarizing \eqref{u(t)alpha} and \eqref{0t1-est}  we obtain the estimate
$$
\|\nabla^k u(t)\|_{L^q_{s_0}(\Omega)} \leq C\|(-\Delta)u(t)\|_{L^p_{s}(\Omega)}^{\beta+\frac{k}{2}}  \|u(t)\|_{L^p_{s}(\Omega)}^{1-\beta-\frac{k}{2}}
\leq Ct^{-\beta-\frac{k}{2}} \|u_0\|_{L^p_{s}(\Omega)}. 
$$
Now the proof of \eqref{etA<} is complete.
\end{proof}

%\vspace{2ex}

\begin{section}{Proof of the main result}

Using Theorem \ref{Stokes-w}, we obtain  existence of time periodic solution of the Navier-Stokes equations. 
%\vspace{2ex}
%
%\begin{thm}\label{existence-per_2} 
%Let $n \geq 3$  for the exterior domain $\Omega_0$.  Moreover, let $ n/2< q_1 < n$ and  $n/2< q_2 < n$, and let ${q}_{12}=\frac{q_1 q_2}{q_1+q_2}$, ${q}^*_{22} =\frac{{q_2^*} q_2}{{q_2^*}+q_2}$ where
% $q_2^*=\frac{n q_2}{n-q_2}$. We suppose that 
% $$
% \frac{2}{n}\Big(1-\frac{s}{2}\Big)< \frac{1}{q_2}
% $$ 
% and 
%$0< s < \min\big\{n\big(1-\frac{1}{q_1}\big), n\big(1-\frac{1}{q_2}\big), \frac{n}{2}\big(1-\frac{1}{q_{12}}\big), n\big(1-\frac{1}{q^*_{22}}\big)\}$. 
%
%Suppose that the external force $F$ belongs to 
%$$ {F} \in  L^\infty_{\rm per}\big(\mathbb{R}; L^{{q}_{12}}_{2s}(\Omega)\cap L^{{q}^*_{22}}_{2s}(\Omega)\big). $$
%
%Set 
%\begin{align*}
%|F|_s  := \big\|\langle x\rangle^{2s} F(t)\big\|_{L^\infty_{\rm per}(\IR; L^{{q}_{12}}(\Omega)\cap L^{{q}^*_{22}}(\Omega)}. 
 %  \end{align*}
%If  $|F|_s$ is sufficiently small, then there exists a $T$-periodic solution $u \in C_{\rm per}(\mathbb{R};  L^{n}_{\sigma}(\Omega)$ with $\nabla u \in  L^\infty_{\rm per}(\mathbb{R};  L^{q_2}_{\sigma}(\Omega)$ 
%such that
%\begin{align}\label{est-u-nablau-2}
 %\sup\nolimits_{t\in [0,T]} \big(\big\|\langle x \rangle^s u(t)\big\|_{L^{q_1 }_{\sigma}(\Omega)} + \big\|\langle x \rangle^s \nabla u(t)\big\|_{L^{q_2}_{\sigma}(\Omega)}\big) \leq C|F|_s. \end{align}
%end
To prove Theorem \ref{existence-per_2},   we define the Poincar\'{e} map for Kato's solution,
\begin{align}\label{equ:ns-per}
	u\mapsto H[u](t) := \displaystyle\int_{-\infty}^t e^{-(t-\tau)A} \{B[u](\tau) + Pf(\tau)\} \dtau,\quad t\in\IR,
\end{align}
where $B[u](\tau)= - P (u(\tau)\cdot \nabla u(\tau))$.
\BLACK Then the periodic solution $u$ is considered as a fixed point of the nonlinear operator $H[u]$, {\em cf.} Kozono and Nakao \cite{Kozono-Nakao}. Actually, the following proposition holds:

%To explain \eqref{equ:ns-per} we consider \eqref{equ:ns-per} as a fixed point problem for the nonlinear operator 
%	$$ H[u](t):=\displaystyle\int_{-\infty}^{t} U(t,s) \{B[u](s) + P(s)F(s)\} \ds. $$

\begin{prop}\label{periodicity}
Let  $f(x,t)$ be a $T$-periodic function. 
Then a $T$-periodic solution $u$ is represented by \eqref{equ:ns-per}. Conversely, 
if there is a unique fixed point $u$ of the operator 
$H$, {\rm i.e.,} 
$$
u=H[u], 
$$
then $u(t+T)=u(t)$. 
\end{prop}
\vspace{1ex}

%The  formulation is used by Kozono and Nakao \cite{Kozono-Nakao}. 
Recall the notation $q_{12}, q_{2}^*,q_{22}^*$ and $|f|_s$, see \eqref{As2}, in Theorem \ref{existence-per_2}. To show existence of fixed points, we  work first of all in spaces of type $L^\infty_{\rm per}$ with respect to time and later focus on spaces $C^0_{\rm per}$ in time. \BLACK

%\vspace{1ex}

\begin{prop}\label{map H}
Under the assumption of Theorem \ref{existence-per_2}, it holds that 
$$ H[\cdot]: L^\infty_{\rm per}(\IR;L^{q_1}_{\sigma,s_1}(\Omega)\cap \widehat{W}^{1,q_2}_{0,s_2\BLACK}(\Omega)) \to L^\infty_{\rm per}(\IR;L^{q_1}_{\sigma,s_1}(\Omega)\cap \widehat{W}^{1,q_2}_{0,s_2\BLACK}(\Omega)), $$
	and $H$ satisfies the estimates
\begin{align}\label{contraction1}
	& \| H[u]\|_{L^\infty_{\rm per}(\IR; L^{q_1}_{s_1})} + \| \nabla H[u]\|_{L^\infty_{\rm per}(\IR; L^{q_2}_{s_2})}\\
	& \qquad\leq C\big\{ \|u\|_{L^\infty_{\rm per}(\IR; L^{q_1}_{s_1})} \|\nabla u\|_{L^\infty_{\rm per}(\IR; L^{q_2}_{s_2})} +  
\|\nabla u\|^2_{L^\infty_{\rm per}(\IR; L^{q_2}_{s_2})}	+ |f|_s \big\} \nonumber
	\end{align} 
and 
\begin{align}\label{contraction2}
\begin{aligned}
	\|H[u] & - H[v]\|_{L^\infty_{\rm per}(\IR; L^{q_1}_{s_1})} +  \|\nabla (H[u]-H[v])\|_{L^\infty_{\rm per}(\IR; L^{q_2}_{s_2})}  \\ 
	& \leq C\big\{  \|u\|_{L^\infty_{\rm per}(\IR; L^{q_1}_{s_1})} +   \|v\|_{L^\infty_{\rm per}(\IR; L^{q_1}_{s_1})} +  \|\nabla u\|_{L^\infty_{\rm per}(\IR; L^{q_2}_{s_2})} +  \|\nabla v\|_{L^\infty_{\rm per}(\IR; L^{q_2}_s)}\big\} \\
	& \quad \times \big\{ \|u-v\|_{L^\infty_{\rm per}(\IR; L^{q_1}_{s_1})} +   \| \nabla(u-v)\|_{L^\infty_{\rm per}(\IR; L^{q_2}_{s_2})}\big\}.
	\end{aligned}          
\end{align}
\end{prop}

\vspace{1ex}

\begin{proof}
We set 
\begin{align*}
		\begin{aligned}
H[u] & = \displaystyle\int_{-\infty}^{t-2} e^{-(t-\tau)A} B[u](\tau) \dtau  +\displaystyle\int_{t-2}^{t} e^{-(t-\tau)A} B[u](\tau) \dtau
+\displaystyle\int_{-\infty}^t e^{-(t-\tau)A} Pf(\tau) \dtau \\[1ex]
%&\qquad +\displaystyle\int_{-\infty}^t e^{-(t-\tau)A} Pf(\tau) \dtau \\
& =:I_1(t) +I_2(t) +I_3(t). 
\end{aligned}          
	\end{align*}
Concerning $I_1$, Theorem \ref{Stokes-w} and H\"older's inequality $\|u\cdot\nabla u\|_{L^{q_{12}}_{s_1+s_2}} \leq \|u\|_{L^{q_1}_{s_1}} \|\nabla u\|_{L^{q_2}_{s_2}}$ with weights and $q_{12} = \frac{q_1q_2}{q_1+q_2}$ imply \BLACK that  
\begin{align*}
	\begin{aligned} 	
\|I_{1}(t)\|_{L^{q_1}_{s_1}} & \leq C \|u\|_{L^\infty_{\rm per}(\IR; L^{ q_1}_{s_1})} \| \nabla  u\|_{L^\infty_{\rm per}(\IR; L^{q_2}_{s_2})} \; \int_{-\infty}^{t-2}  (t-\tau)^{-\frac{n}{2q_2}-\frac{s_2}{2}}
 \dtau  \\
& \leq C \|u\|_{L^\infty_{\rm per}(\IR; L^{ q_1}_{s_1})} \| \nabla  u\|_{L^\infty_{\rm per} (\mathbb{R};L^{q_2}_{s_2})}. 
	\end{aligned}
\end{align*}
Note that the above improper integral is uniformly bounded in $t$ since $\frac{n}{q_2} +s_2 > 2$ by assumption \eqref{As11}. To apply the Helmholtz projection $P$ we also need that 
\begin{align}\label{conditionq1-q2}
s_1 + s_2 < n\Big(1-\frac{1}{q_1}-\frac{1}{q_2}\Big) = \frac{n}{q_{12}'}. 
\end{align}
Thus we have to take $q_1$ and $q_2$ satisfying  
$2- s_2 < \frac{n}{q_2} < n-\frac{n}{q_1}-(s_1+s_2)$,
{\em i.e., }
\begin{align}\label{conditions_1-q1}
\frac{n}{q_1}+s_1 < n-2,
\end{align}
{\em cf.} \eqref{As1}. %This can be verified for $1 < q_1 < n$ and $n=3$ with $s_1<0$ and for $n\geq 4$ with negative as well as positive $s_1$, independent of $q_2$ and $s_2$. \BLACK

\vspace{2ex}

\vspace{2ex}

%\BLUE{\bf 20 October Idea 4}; \BLACK 

%\BLUE satisfying $q_2 < q'_1$ (In $n=3$ case, since $\frac{n}{2}< q'_1 < n)$ we can take such a $q_2$. Here I do not consider $n= 4$ not rigolously. \BLACK.  
%Since 
%$$
%\frac{n}{2q_1}+\frac{1}{2} > 1, \quad \frac{n}{2 q_{12}} > 1
%$$
%due to the assumptions $1<q_1<n$, \BLACK $n/2 < q_2 <n$, 
From Theorem \ref{Stokes-w} we get that %\BLUE if $s_1 >0$ when $n\geq 4$, \BLACK
\begin{align*}
	\begin{aligned} 	
\|\nabla I_{1}(t)\|_{L^{q_2}_{s_2}} & \leq C \|u\|_{L^\infty_{\rm per}(\IR; L^{ q_1}_{s_1})} \| \nabla  u\|_{L^\infty_{\rm per}(\IR; L^{q_2}_s)}\\
& \quad \times \int_{-\infty}^{t-2}  \big\{ (t-\tau)^{-\frac{n}{2q_1}-\frac{1}{2} -\frac{s_1}{2}} + (t-\tau)^{-\frac{n}{2q_{12}} -{(s_1+s_2)}\,%(not\, \frac{s}{2}
}\BLACK\big\}
 \dtau  \\
& \leq C \|u\|_{L^\infty_{\rm per}(\IR; L^{ q_1}_{s_1})} \|\nabla  u\|_{L^\infty_{\rm per}(\IR; L^{ q_2}_{s_2})}. 
	\end{aligned}
\end{align*}
Here we have to satisfy the condition $0 \leq s_1 +s_2 < n\big(1-\frac{1}{q_{12}}\big)$, see \eqref{conditionq1-q2},
%\begin{align}\label{condition-small-t}
%0< s_1 +s_2 < n\Big(1-\frac{1}{q_{12}}\Big). 
%\end{align}
%and 
%\begin{align}\label{condition2-small-t}
% s_2 < n\Big(1-\frac1{q_{22}^*}\Big) 
%\end{align}
to apply the projection $P$. Moreover, to get integrability, we require $\frac{n}{q_1}+s_1>1$ as in \eqref{As1} and recall that $q_1<n$, $q_2<n$. 
%Indeed, the condition in \eqref{condition-small-t} is already satisfied by \eqref{conditionq1-q2}. 
\BLACK 

%\BLUE{\bf 30 October; new idea} 
%We can take $0<s_2 < 1$ with $2-s_2 < \frac{n}{q_2}$ and 
%$$
%s_2 < n\Big(1-\frac{1}{q_1}-\frac{1}{q_2}\Big)
%$$
%for $n/2 < q_1 <n$ and $n/2 <q_2 <n$. 
%(When $2- s_2 < \frac{n}{q_2} < n-\frac{n}{q_1}-s_2$, 
%$$
%2< n-\frac{n}{q_1}
%$$

We recall the fractional Hardy inequality on weighted spaces; 
\BLACK

As for $I_2(t)$, we note that 
$$
\frac{n}{2q_2} =\frac{n}{2 q^*_{2}}+\frac{1}{2} < 1. 
$$
Hence \eqref{etA<} of Theorem \ref{Stokes-w} and a Sobolev embedding, see Proposition \ref{embed-weighted}, imply that 
\begin{align}
	\begin{aligned}\label{weak-singularity} 	
\|I_{2}\|_{L^{q_1}_{s_1}} + \|\nabla I_{2}\|_{L^{q_2}_{s_2}} 
& \leq C \big(\|u\|_{L^\infty_{\rm per}(\IR; L^{ q_1}_{s_1})} \|\nabla u\|_{L^\infty_{\rm per}(\IR; L^{q_2}_{s_2})}
+ \|u\|_{L^\infty_{\rm per}(\IR; L^{q^*_2}_{s_2})} \|\nabla u\|_{L^\infty_{\rm per}(\IR; L^{q_2}_{s_2})}\big)
\\
& \quad \times \int_{t-2}^{t}
\big\{ (t-\tau)^{-\frac{n}{2q_2}}+
(t-\tau)^{-\frac{n}{2q^*_{2}}-\frac{1}{2}} \big\} \dtau  \\
& \leq C \big(\|u\|_{L^\infty_{\rm per}(\IR; L^{ q_1}_{s_1})} \| \nabla  u\|_{L^\infty_{\rm per}(\IR; L^{q_2}_{s_2})} + \|\nabla u\|_{L^\infty_{\rm per}(\IR; L^{q_2}_{s_2})}^2\big). 
	\end{aligned}
\end{align}
Integrability is guaranteed, but to apply the Helmholtz projection we have to satisfy the conditions $0\leq s_1 +s_2 < n\Big(1-\frac{1}{q_{12}}\Big)$, see \eqref{conditionq1-q2}, and for $\nabla I_2$ with the two exponents $0,s_2$ rather than $s_2,s_2$
\begin{align*} %\label{condition-small-t2}
s_2 < n\Big(1-\frac1{q_{22}^*}\Big).
\end{align*}
%and 
%\begin{align}\label{condition2-small-t}
% s_2 < n\Big(1-\frac1{q_{22}^*}\Big) 
%\end{align} 
This condition is equivalent to $\frac{n}{q_2} +\frac{s_2}{2}<\frac{n+1}{2}$, {\em cf.} \eqref{As11}. \end{proof}

\BLACK 

%\BLUE 
%when $q_2=2n$, 
%$$
%1<s_2 < n.
%$$
\BLACK 
%\RED (the next lines I do not understand) \BLACK
%it is enough to take $q_2$ that 
%$$
%2-s_2 < \frac{n+1}{2}-\frac{s_2}{2}, 
%$$
%and thus 
%$$
%2-\frac{n+1}{2}< \frac{s_2}{2}. 
%$$
%This holds for any positive $s_2$. 
\BLACK
%\RED (In the second integral term \BLUE $2s<n(1-1/q_{22}^*)$, \RED not $s<n(1-1/q_{22}^*)$, was used) (I do not see where $s<n(1-1/q_1)$ or $s<n(1-1/q_2)$ were used. These conditions look natural, but estimates from Thm. 3.1 are exploited with $q_{12}, q_{22}^*$ only \BLUE I agree. 
%), \RED (only in the proof of continuity below, $s<n(1-1/q_1)$ is used for the semigroup $e^{-tA}$ on $L^{q_1}_s$) ($s<n(1-1/q_2)$ is never used, but maybe for the proof of continuity of $\nabla H(t)$) 
%\RED (We used $\|u\|_{L^{q^*_2}_{s_2}} \leq c\|\nabla u\|_{L^{q_2}_{s_2}}$ provided $-\frac{n}{q_2} < s_2 < \frac{n}{q_2'}$) \BLACK The estimates of $I_3$ using \eqref{As2} \BLACK and \eqref{contraction2} are analogous.  

\vspace{1ex}

%\BLUE {\bf 1 November, Idea 4;} \BLACK

%Let $n \geq 3$, \BLUE $\frac{3n}{4}<q_1<n $,  \BLACK $n/2 <q_2<n $,  close to $n$, $s_1=0$, $0< s_2< n/q_2'$, 
%$p> 2$ $(p'<2)$ close to $2$  satisfying that 
%\begin{align}\label{condition-p'}
%1-\frac{p's_2}{2} < \frac{np'}{2q_2} <1 
%\end{align}
%and 
%\BLUE 
%$$
%u\in L^\infty(\mathbb{T}; L^{q_1}_{\sigma}(\Omega)),  \ \ 
%\nabla u\in L^p (\mathbb{T}; L^{q_2}_{\sigma,s_2}(\Omega)), 
%$$
%where $\mathbb{T}=\mathbb{R}/[0,T]$, which is a torus group.  
%We consider a solution space 
%$$
%S_a=\{u; u\in L^\infty(\mathbb{T}; L^{q_1}_{\sigma}(\Omega)),  \ \ 
%\nabla u\in L^p(\mathbb{T}; L^{q_2}_{\sigma,s_2}(\Omega)),   \ \ u(t+T)=u(t) \mbox{ a.e. t} \} 
%$$
%with the norm 
%$$\|u\|_{S_a}=  \|u\|_{L^\infty (\mathbb{T}; L^{q_1}_{\sigma}(\Omega))}+ \|\nabla u\|_{L^p(\mathbb{T}; L^{q_2}_{\sigma,s_2}(\Omega))}\leq a
 
Now Proposition \ref{map H} and Banach's fixed point theorem in a sufficiently small closed ball in $L^\infty_{\rm per}(\IR; L^{q_1}_{\sigma,s_1}(\Omega)) \cap L^\infty_{\rm per}(\IR; \widehat W^{1,q_2}_{0,s_2}(\Omega))$ yield the existence of a $T$-periodic mild solution $u$  of \eqref{equ:ns} in Theorem \ref{existence-per_2}. 
Note that the estimates in the above proof and the well-definedness of $H[u](t)$ hold for all $t\in\IR$ and not only for almost all $t$.  This remark is helpful in some of the following arguments.

Next we prove that 
$u=H[u] \in C_{\rm per}(\IR;L^{q_1}_{s_1}(\Omega))$ and even  $H[u] \in C_{\rm per}(\IR;L^{q}_{\alpha\;\BLACK}(\Omega))$ for all $q_1\leq q<q_2^*$ and some positive  \BLACK $\alpha$.  
%For the property $u \in C_{\rm per}(\IR;L^{n}_{s}(\Omega))$ we refer to Remark \ref{finalremark} (ii). 
%\BLUE (In this case, we need $q_1 \leq q_2$ and $\frac{n}{2}(\frac{1}{q_1}-\frac{1}{q_2})<1$ to obtain $u\in C^0_{\rm per}(\IR; L^{n}_s(\Omega))$ . However, I think that a similar argument with Proposition 2.2 to that in the half space case removes the additional condition.)
\BLACK 
\vspace{2ex}

\noindent
{\bf Proof of continuity of $H[u](\cdot)$ in $L^{q}_{s_q}(\Omega)$ for each $q_1\leq q<q_2^*$:}\;  
Consider $H(t):=H[u](t)$ as a function of time $t$ for fixed $u$, and let $h(\tau):= B[u](\tau) + P f(\tau)$. Suppose without loss of generality that $B[u]$ and $P f(\cdot) $ satisfy pointwise bounds for all $\tau$ rather than $\tau$-a.e. 
For $t'>t$ we get that
\begin{align*}
	\begin{aligned}
H(t')-H(t) & = \int_{-\infty}^t \big(e^{-(t'-\tau)A} - e^{-(t-\tau)A}\big) h(\tau)\dtau   + \int_t^{t'} e^{-(t'-\tau) A}\, h(\tau) \dtau\\
& =: (J_1 + J_2)(t',t).
\end{aligned}
\end{align*}
Concerning $J_2$ we see from the above estimates that the norm of the integrand satisfies the estimate 
$$ \big\|e^{-(t'-\tau) A} h(\tau)\big\|_{L^{q_1}_{s_1}} \leq C  \begin{cases}
(t'-\tau)^{-\frac{n}{2q_2}-\frac{s_2}{2}} & \mbox{for} \ \ t'-\tau>1,\\[1ex]
(t'-\tau)^{-\frac{n}{2q_2}} & \mbox{for} \ \ t'-\tau \leq 1,\\
\end{cases} $$ 
where $C$ depends on the norms of $u$ and $f$. Since by assumptions $\frac{n}{2q_2}<1$ and $\frac{n}{2q_2}+\frac{s_2}{2}>1$, we get that $J_2$ tends to $0$ as $0<t'-t\to 0$. 
%A similar argument holds for 
%$\|\nabla U(t',s)\{B[u](s) + %P(s)F(s)\}\|_{L^{q_2}}$.

As for $J_1$, 
\begin{align*}
\|J_1\|_{L^{q_1}_{s_1}} \leq \int_{-\infty}^t \big\|(e^{-(t'-t)A}-I) \big[e^{-(t-\tau)A} h(\tau)\big]\big\|_{L^{q_1}_{s_1}} \dtau
\end{align*}
where $\big\|e^{-(t-\tau)A} h(\tau)\big\|_{L^{q_1}_{s_1}}$ is integrable with respect to $\tau \in(-\infty,t)$  by the above arguments for $J_2$. Then due to the strong continuity of the analytic Stokes semigroup $e^{-tA}$ on $L^{q_1}_{s_1}(\Omega)$ and its global-in-time boundedness Lebesgue's theorem on dominated convergence shows that $J_1$ tends to $0$ as $t'\to t+$.

For the limit $t\to t'-$ we estimate the term $J_2(t)$ as above. Moreover, we write $J_1(t)$ in the form 
$$ J_1(t) = \int_{-\infty}^{t-\eta}  \big(e^{-(t'-\tau)A} - e^{-(t-\tau)A}\big) h(\tau)\dtau + \int_{t-\eta}^{t} \big(e^{-(t'-\tau)A} - e^{-(t-\tau)A}\big) h(\tau)\dtau. $$
Given any $\varepsilon >0$ there exist by arguments as for $J_2(t)$ an $\eta>0$ such that the second integral on the right hand side is bounded in $L^{q_1}_{s_1}(\Omega)$ by $\varepsilon$ uniformly in $t\in (t'-\eta,t')$. The integrand of the first integral is rewritten as
$\big(e^{-\eta A} - e^{-(\eta-(t'-t))A}\big) e^{-(t'-\eta-\tau)A}\,h(\tau)$.
Then we apply Lebesgue's convergence theorem and the strong continuity of the Stokes semigroup at time $\eta$ to
$$ \int_{-\infty}^{t'-\eta} \big\|\big(e^{-\eta A} - e^{-(\eta-(t'-t))A}\big) e^{-(t'-\eta-\tau)A}\,h(\tau)\big\|_{L^{q_1}_{s_1}(\Omega)} \dtau $$
and let $t$ pass to $t'-$. This proves that $H[u] \in C_{\rm per}(\IR;L^{q_1}_{s_1}(\Omega))$.

We note that  $H(\cdot)$ is uniformly bounded in $\widehat W^{1,q_2}_{ 0,s_2}$ and thus also in $L^{q_2^*}_{s_2}$ by  Proposition \ref{embed-weighted}. Hence the interpolation estimate  
\begin{align*}
	\begin{aligned}
\|H(t')-H(t)\|_{L^q_{s_q}} \leq \|H(t')-H(t)\|_{L^{q_1}_{s_1}}^{1-\kappa} \|H(t')-H(t)\|_{L^{q_2^*}_{s_2}}^\kappa,\quad  q_1<q<q_2^*, 
\end{aligned}
\end{align*}
where $\kappa =\kappa(q,q_1,q_2^*) \in(0,1)$ is defined by $\frac{1}{q} = \frac{1-\kappa}{q_1} + \frac{\kappa}{q_2^*}$ and $s_q=s_1(1-\kappa)+s_2\kappa$, implies that  $H[u] \in C_{\rm per}(\IR;L^q_{\sigma,s_q}(\Omega))$ for each $q_1\leq q<q_2^*$. 
Therefore, we can consider $H[\cdot]$ also as a map from $C_{\rm per}(\IR; L^{q_1}_{\sigma,s_1}) \cap L^\infty_{\rm per}(\IR; \widehat W^{1,q_2}_{0,s_2})$ into itself, and Banach's fixed point theorem implies that the map $H[\cdot]$ admits a unique fixed point $u$ even in an adequate closed ball in $C_{\rm per}(\IR;L^{q_1}_{\sigma,s_1}(\Omega))\cap L^\infty_{\rm per}(\IR;\widehat W^{1,q_2}_{0,s_2}(\Omega))$.
Since by the above arguments, $u \in C_{\rm per}(\IR;L^q_{s_q}(\Omega))$ for each $q_1\leq q<q_2^*$, and $q_1<n<q_2^*$ because of the fact that  $q_2>\frac{n}{2}$, we have that $u \in C_{\rm per}(\IR;L^n_{s_n}(\Omega))$ where $s_n\in (s_1,s_2)$  with $s_n \geq 0$ for either $s_1\geq 0$ or $s_1<0$ but having small modulus. \BLACK 

\BLACK

%\RED ((The definitions $\frac{1}{q} = \frac{1-\kappa}{q_1} + \frac{\kappa}{q_2^*}$ and $s_q=s_1(1-\kappa)+s_2\kappa$ imply that $\kappa=\frac{\frac1{q_1}-\frac1q}{\frac1{q_1}-\frac{1}{q_2^*}}$ and $1-\kappa=\frac{\frac1{q}-\frac1{q_2^*}}{\frac1{q_1}-\frac{1}{q_2^*}}$. Hence $s_q=s_n$ is nonnegative if and only if with $q=n$ 
%$$ s_1 \geq -\frac{\kappa}{1-\kappa} s_2= -\frac{\frac1{q_1}-\frac1q}{\frac1{q_1}-\frac1{q_2}+\frac1n} s_2 = - \frac{\frac1{q_1}-\frac1n}{\frac2n-\frac1{q_2}} s_2 .$$
%It is difficult to see when this condition is satisfied.)) 
\BLACK

\vspace{1ex}

\noindent
{\bf Proof of continuity of $\nabla H[u](\cdot)$  provided $q_1\leq q_2$ and $\frac{1}{q_1}-\frac{1}{q_2} < \frac1n$:}  
First we prove pointwise estimates for $h(\tau)= B[u](\tau) + P(\tau)f(\tau)$. In view of the above estimates for $I_1(t)$, $I_2(t)$ (and $I_3(t)$) we obtain that
\begin{align*}
\| e^{-(t-\tau)A} h(\tau)\|_{L^{q_1}_{s_1}} \leq Cj_1(t,\tau),
\end{align*}
where 
\begin{equation*}
j_1(t,\tau)=
\begin{cases} 
 (t-\tau)^{-\frac{n}{2q_2}-\frac{s_2}{2}}  
 \;\mbox{for} \ \ t-\tau>1,\\[1ex] 
(t-\tau)^{-\frac{n}{2q_2}}\;\quad\mbox{for} \ \ t-\tau\leq 1,
\end{cases} 
\end{equation*}
and $C$ depends on the norms of $u$, $f$. Moreover, by the estimates of $\nabla I_j(t)$, $j=1,2,3$, above there holds for $t>\tau$
$$
\|\nabla e^{-(t-\tau)A} h(\tau)\|_{L^{q_2}_{s_2}} \leq C j_2(t,\tau),
$$
$$ j_2(t,\tau) =  
\begin{cases} 
(t-\tau)^{-\frac{n}{2q_1}-\frac{1}{2} -\frac{s_1}{2}} + (t-\tau)^{-\frac{n}{2q_{12}} -\frac{s_1+s_2}{2}} \BLACK \;\mbox{for} \ \ t-\tau>1,\\[1ex]
(t-\tau)^{-\frac{n}{2q^*_{2}}-\frac{1}{2}} 
\hspace{35mm} \mbox{for} \ \ t-\tau \leq 1.
\end{cases}
$$
In particular, both $j_1$ and $j_2$ are integrable in $\tau\in(-\infty,t')$.
If $t'>t$, then
$$ \nabla H(t')-\nabla H(t) = \int_{-\infty}^t \big(\nabla e^{-(t'-\tau)A}-\nabla e^{-(t-\tau)A}\big) h(\tau)\ \dtau + \int_t^{t'} \nabla e^{-(t'-\tau)A}  h(\tau)\dtau, $$
where the second integral converges to 0 as $t'\to t+$ due to the properties of $j_2$. 

In the first integral let $0<\eta<1$ and write 
\begin{align}\label{int-nablaU-nablaU}
	\begin{aligned}	
\Big\|\int_{-\infty}^t & \big(\nabla e^{-(t'-\tau)A}-\nabla e^{-(t-\tau)A}\big) h(\tau)\dtau\Big\|_{L^{q_2}_{s_2}} \\
& \leq 
\int_{-\infty}^{t-\eta} \big\|\nabla e^{-\eta A} (e^{-(t'-t)A}-I) \;e^{-(t-\eta-\tau)A} h(\tau)\big\|_{L^{q_2}_{s_2}}\dtau \\
& \quad + \int_{t-\eta}^t \big(\big\|\nabla e^{-(t'-\tau)A} h(\tau)\big\|_{L^{q_2}_{s_2}} + \big\|\nabla e^{-(t-\tau)A} h(\tau)\big\|_{L^{q_2}_s}\big) \dtau. \\
	\end{aligned}
\end{align}
Given any $\epsilon>0$ we fix $\eta>0$ such that due to the properties of $j_2$ the third integral in \eqref{int-nablaU-nablaU} is bounded by $\epsilon$. With this $\eta$ we consider the second integral in \eqref{int-nablaU-nablaU}. The term $e^{-(t-\eta-\tau)A} h(\tau)$ is integrable in $L^{q_1}_{s_1}$ due to $j_1$. Hence in order to apply Lebesgue's theorem of dominated convergence it suffices to show that $\nabla e^{-\eta A} (e^{-(t'-t)A}-I)$ is uniformly bounded with respect to time in ${\mathcal L(L^{q_1}_{\sigma,s_1},L^{q_2}_{s_2})}$ and strongly converges to 0 as $t'\to t$.
Indeed, by \eqref{etA<} in Theorem \ref{Stokes-w} there holds 
\begin{equation}\label{etaA}
\|\nabla e^{-\eta A} \|_{\mathcal L(L^{q_1}_{\sigma,s_1},L^{q_2}_{s_2})} \leq c\eta^{-\frac{n}{2}\big(\frac{1}{q_1}-\frac{1}{q_2}\big)-\frac12},
\end{equation}
provided that $q_1\leq q_2$ and $ \frac{1}{q_1}-\frac{1}{q_2} <\frac1n$. 
Moreover, $e^{-(t'-t)A}-I$ converges strongly on $L^{q_1}_{s,\sigma}(\Omega)$ to $0$ as $t'\to t+$ due to classical semigroup properties of $e^{-\tau A}$.

To get the same result when $t'\to t-$  we write
\begin{align*} 
\nabla H & (t) - \nabla H(t') = \int_{t'}^t \nabla e^{-(t-\tau)A} h(\tau)\dtau\\
& + \int_{-\infty}^{t'-\eta} \big(\nabla e^{-(t-\tau)A}-\nabla e^{-(t'-\tau)A}\big) h(\tau)\dtau + \int_{t'-\eta}^{t'} \big(\nabla e^{-(t-\tau)A} -\nabla e^{-(t'-\tau)A}\big) h(\tau)\dtau .
\end{align*}
The first integral is easily shown to converge  in $L^{q_2}_{s_2}$ \BLACK to $0$ as $t'\to t-$. As for the third one, we see as above that this integral is bounded by a prescribed $\epsilon>0$ uniformly in $t'\in (t-\eta,t)$, choosing $\eta>0$ sufficiently small. Now it suffices to consider the second integral with this $\eta$; its integrand will be written in the form 
$$
\nabla e^{-\eta A/2} \big\{e^{-\eta A/2}-  e^{-(t'-t+\eta/2)A}\big\}  e^{-(t-\eta-\tau)A}h(\tau),\quad \tau\in (-\infty,t'-\eta),\;  t'\in(t-\eta,t). 
$$ 
By the above estimate $\big\| e^{-(t-\eta-\tau)A} h(\tau)\big\|_{L^{q_1}_{s_1}} \leq Cj_1(t-\eta,\tau)$  with $j_1(t-\eta, \cdot)\in L^1(-\infty,t-\eta)$, the bound  $\big\|\nabla e^{-\eta A/2}\big\|_{\mathcal L(L^{q_1}_{\sigma,s_1},L^{q_2}_{s_2})} \leq C_\eta$ as in \eqref{etaA}, and the strong convergence $e^{-\eta A/2}-  e^{-(t'-t+\eta/2)A}$ to $0$ in $L^{q_1}_{\sigma,s_1}(\Omega)$ as $t'\to t-$ with $t'\in (t-\eta/4,t)$, we conclude that Lebesgue's theorem can be applied.  In summary, we proved that $\nabla H(t') \to \nabla H(t)$ in $L^{q_2}_{s_2}$ for $t'\to t-$ as well.

Now Theorem \ref{existence-per_2} is proved.
\hfill\qed

%\begin{rem}\label{finalremark}
%{\rm 
%Note that the condition $s<n\big(1-\frac{1}{q_1})$ in Theorem \ref{existence-per_2} has been used in the above proof of continuity in $L^{q_1}_s$, but not in the construction of the solution in $L^\infty(\IR;L^{q_1}_s(\Omega))$. \RED Are you sure? This condition was a part of many a priori estimates in §3.

%(ii) The property $H[u]\in C^0_{\rm per}(\IR; L^{q_1}_s(\Omega))$ is easily extended to continuity with respect to the space $ L^{q}_s(\Omega)$ for $q_1<q \leq q_2^*$  provided $q_1\leq q_2$ and $\frac{1}{q_1}-\frac{1}{q_2} < \frac1n$. \BLACK Indeed, by H\"older's inequality $\|v\|_{L^{q}_s} \leq \|v\|_{L^{q_1}_s}^\kappa\|v\|_{L^{q_2^*}_s(\Omega)}^{1-\kappa}$ with $\kappa\in [0,1)$, the Sobolev embedding $\|v\|_{L^{q_2^*}_s}\leq c\|\nabla v\|_{L^{q_2}_s}$ for $v\in\mathcal D(A)$, and the property $\nabla H\in L^\infty(\IR;L^{q_2}_s(\Omega))$ we conclude that $H[u]\in C^0_{\rm per}(\IR; L^{q}_s(\Omega))$. In particular, $u\in C^0_{\rm per}(\IR; L^{n}_s(\Omega))$.
%(\BLUE I think that we can remove thi remark (ii). )
%\RED DELETE: (iii) Concerning continuity with values in $L^{q_2^*}_s(\Omega)$ a density argument and the property $\nabla H\in L^\infty(\IR;L^{q_2}_s)$ imply the weak continuity $H[u]\in C^0_w(\IR; L^{q_2^*}_s(\Omega))$, {\em i.e.,} for any $\psi\in L^{(q_2^*)'}_{-s}(\Omega)= \big(L^{q_2^*}_s(\Omega)\big)^*$ we get that $t\mapsto \langle H[u](t),\psi\rangle \in C_{\rm per}(\IR)$. 
\BLACK

%\RED(((iv) (Only internal) The continuity $H[u]\in C^0(\IR; L^{q_2^*}_s(\Omega))$ and $\nabla H[u] \in C^0(\IR; L^{q_2}_s(\Omega))$ seem to be more involved. Either we have to prove that the map $t\mapsto \nabla \big(e^{-\tau A}-I\big)$ is strongly continuous in weighted $L^q$ spaces with adequate upper norm bounds or we work with $(A_{q_2,s})^{1/2}$ so that $A_{q_2,s}^{1/2} \big(e^{-\tau A}-I\big) = \big(e^{-\tau A}-I\big) A_{q_2,s}^{1/2}$. But then mapping properties of $A_{q_2,s}^{1/2}$ on weighted $L^q_s$ spaces with change of $q$ and $s$ are needed.))(\BLUE I agree.)\BLACK
%}
%\end{rem}
\BLACK

%
%Since $H(\cdot)$ is uniformly bounded in $\dot W^{1,q_2}$ and hence in $L^{q_2^*}$  %(Delete: and since $H(t')\to H(t)$ in $L^{q_1}$)
%we conclude from the interpolation estimate %\RED (Delete: that $u(t')\rightharpoonup u(t)$ in $L^{q_2^*}$ weakly. Thus, )
%\begin{align*}
%\|H(t')-H(t)\|_{L^q} \leq \|H(t')-H(t)\|_{L^{q_1}}^\kappa \|H(t')-H(t)\|_{L^{q_2^*}}^{1-\kappa},\quad  q_1<q<q_2^*, \BLACK
%\end{align*}
%where $\kappa =\kappa(q) \in(0,1)$, 
%that $H[u] \in C_{\rm per}(\IR;L^q(\Omega))$ for each $q_1\leq q<q_2^*$. 
%herefore, we can consider $H[\cdot]$ as a map from $C_{\rm per}(\IR; L^{q_1}) \cap L^\infty_{\rm per}(\IR; \dot W^{1,q_2})$ into itself. 
%This completes the proof. 

\end{section}

\vspace{2ex}

\noindent {\bf Acknowledgements.}
The second author is supported by JSPS grant
no. 22K13946.\vspace{3mm}  
 
 \noindent {\bf Conflicts of interest statement.}  There is no conflict of interest. \vspace{1mm} 

\noindent {\bf Data Availability statement.} No datasets were generated or analysed during the current study.

\end{document}